\documentclass[letterpaper]{article}
\usepackage{arxiv}
\usepackage[T1]{fontenc}
\usepackage[utf8]{inputenc}
\usepackage[english]{babel}
\usepackage[authoryear,round]{natbib}
\setcitestyle{semicolon,aysep={},yysep={;}}
\usepackage{caption}
\usepackage{amsmath,amssymb,amsthm,mathtools,bm,mathrsfs}
\usepackage{booktabs,array,tabularx}
\usepackage{graphicx}
\usepackage{adjustbox,etoolbox}
\BeforeBeginEnvironment{tabular}{\begin{adjustbox}{max width=\linewidth}}
\AfterEndEnvironment{tabular}{\end{adjustbox}}
\usepackage{algorithm,algpseudocode}
\usepackage{enumitem,placeins}
\usepackage{xurl}
\usepackage[unicode,hidelinks]{hyperref}
\usepackage{microtype}
\newcommand{\objectiveboundcontinuation}{}
\newcommand{\nonlinearcouplingcontinuation}{}
\newcommand{\couplingboundcontinuation}{}
\newcommand{\workboundcontinuation}{}
\newcommand{\liftedformulationequations}{%
  \begin{alignat}{2}
    \phi(z)&=h(z)-\frac14\norm{z}^2,
    &\qquad r(z,u)&=\frac14\norm{z}^2+\phi^*(u),\label{eq:main-lift}\\[6pt]
    \lambda_k&=\kappa_k-\frac12,
    &\qquad D&=\diag\!\left(\frac12 I,\Lambda^{-1}\right).\label{eq:main-D}
  \end{alignat}
}
\newcommand{\samplingparametersreference}{\eqref{eq:main-sampling}}
\newcommand{\couplinginputmetricequation}{%
  \begin{equation}\label{eq:main-input-metric}
    D_{H,i}^x=\diag\!\left(\frac{P_{I_{H,i}^x,J_{H,i}^x}}{4\Delta_{H,i}^x},0\right),\qquad
    D_{H,j}^y=\diag\!\left(\frac{P_{I_{H,j}^y,J_{H,j}^y}}{4\Delta_{H,j}^y},0\right).
  \end{equation}
}
\newcommand{\samplingparametersequations}{%
  \begin{equation}\label{eq:main-sampling}
    S=\sum_k c_k,\qquad p_k=\frac{c_k}{S},\qquad
    \eta=\frac{1}{8S},\qquad K=\lceil18S\rceil.
  \end{equation}
}
\newcommand{\beginmainratestable}{\begin{table}[!htbp]}
\newenvironment{mainratestabular}{%
  \normalsize\everymath{\displaystyle}%
  \renewcommand{\arraystretch}{1.6}%
  \begin{tabular}{@{}l@{\hspace{3em}}l@{}}%
}{\end{tabular}}
\newcommand{\mainratesrowspace}{\addlinespace[5pt]}
\newenvironment{maincomponentstabular}{%
  \normalsize\renewcommand{\arraystretch}{1.5}%
  \begin{tabular}{@{}l@{\hspace{2.5em}}l@{}}%
}{\end{tabular}}
\newcommand{\maincomponentsgroupseparator}{\midrule}
\newcommand{\maincomponentsheader}{\textbf{Component} $k$ & $\boldsymbol{G_k(\theta)}$\\}
\newenvironment{mainmatchingtabular}{%
  \normalsize\everymath{\displaystyle}%
  \renewcommand{\arraystretch}{1.6}%
  \begin{tabular}{@{}l@{\hspace{2.5em}}l@{\hspace{2.5em}}l@{}}%
}{\end{tabular}}
\newenvironment{mainconvextabular}{%
  \normalsize\everymath{\displaystyle}%
  \renewcommand{\arraystretch}{1.6}%
  \begin{tabular}{@{}l@{\hspace{3em}}l@{}}%
}{\end{tabular}}
\hypersetup{
 pdftitle={Complexity of Multiblock Convex–Concave Saddle-Point Problems with General Smooth Coupling},
 pdfauthor={Nhat Trung Nguyen, Ivan Dorofeev, Irina Podlipnova, Yuriy Dorn, Alexander Gasnikov},pdfsubject={Research preprint},pdflang={en}}
\newtheorem{theorem}{Theorem}[section]
\newtheorem{lemma}[theorem]{Lemma}

\newtheorem{corollary}[theorem]{Corollary}
\theoremstyle{definition}
\newtheorem{assumption}[theorem]{Assumption}
\newtheorem{definition}[theorem]{Definition}
\newtheorem{example}[theorem]{Example}
\theoremstyle{remark}

\newcommand{\RR}{\mathbb R}
\newcommand{\E}{\mathbb E}
\newcommand{\cI}{\mathcal I}
\newcommand{\cJ}{\mathcal J}
\newcommand{\cL}{\mathcal L}

\newcommand{\ip}[2]{\left\langle #1,#2\right\rangle}
\newcommand{\norm}[1]{\left\|#1\right\|}
\newcommand{\Gap}{\operatorname{Gap}}
\newcommand{\op}{\mathrm{op}}
\newcommand{\out}{\mathrm{out}}
\newcommand{\inpt}{\mathrm{in}}
\DeclareMathOperator{\diag}{diag}
\DeclareMathOperator{\Lip}{Lip}
\DeclareMathOperator{\spanop}{span}
\numberwithin{equation}{section}

\title{Complexity of Multiblock Convex--Concave Saddle-Point
Problems with General Smooth Coupling}
\author{
  Nhat Trung Nguyen \\
  MIRIAI\\
  \texttt{nguyen.n@miriai.org} \\
  \And
 Ivan Dorofeev \\
 Independent Researcher \\
 \texttt{qft98romanov@gmail.com} \\
 \AND
 Irina Podlipnova \\
 MIRIAI \\
 \texttt{podlipnova.i@miriai.org} \\
 \And
 Yuriy Dorn \\
 MSU Institute for Artificial Intelligence \\
 \texttt{dornyv@my.msu.ru} \\
 \AND
 Alexander Gasnikov \\
  Innopolis University, MIRIAI, IITP \\
  \texttt{gasnikov@yandex.ru} \\
}

\date{}
\renewcommand{\shorttitle}{Multiblock Convex--Concave Saddle-Point Problems}

\begin{document}
\maketitle
\raggedbottom
\setlength{\parskip}{3.5pt}
\begin{abstract}
\widowpenalty=10000
We study multiblock convex--concave saddle-point problems of the form
$\min_x\max_y\{f(x)+H(x,y)-g(y)\}$,
with separately charged partial-gradient oracles for all three objective
components. We impose joint smoothness and strong convexity bounds
on the functions $f$ and $g$, together with blockwise Lipschitz conditions
on the partial gradients of $H$, so that the resulting oracle complexities
reflect the smoothness and dependency parameters specific to each oracle.
Under strong convexity, we develop a randomized lifted extragradient method
that achieves linear convergence while providing simultaneous complexity
guarantees for all oracles. The partial-gradient complexities for $f$ and $g$
exhibit the expected square-root dependence on the corresponding block
condition numbers, whereas the coupling-gradient complexities depend on
the associated curvature-normalized smoothness parameters.
Finite-dimensional bilinear subclasses provide lower-complexity benchmarks
within a randomized blockwise linear-span model. When strong convexity
of $f$ and $g$ is not assumed, a regularization argument yields expected
restricted-gap guarantees with accuracy dependence $\varepsilon^{-1/2}$
for the objective-gradient oracles and $\varepsilon^{-1}$ for the
coupling-gradient oracles, up to logarithmic factors. Experiments show observed oracle
savings and runtime improvements in regimes of poor objective conditioning.
\end{abstract}

\providecommand{\samplingparametersreference}{\eqref{eq:main-sampling}--\eqref{eq:main-parameters}}
\section{Introduction}\label{sec:intro}
Learning tasks linked by a graph, inverse problems with several sensor
groups, and models with shared score perturbations all combine local
objectives with interactions between variable blocks. Evaluating a local
data loss can be much more expensive than evaluating a regularizer or an
interaction, and different interactions can involve different subsets of
the variables. These differences suggest that an optimization method
should provide a separate evaluation budget for each partial gradient.
We consider smooth convex--concave saddle-point problems of the form
\begin{equation}\label{eq:problem-intro}
 \min_{x\in\RR^{d_x}}\max_{y\in\RR^{d_y}}
\{\cL(x,y)=f(x)+H(x,y)-g(y)\}.
\end{equation}
We write $x=(x_1,\ldots,x_m)$ and $y=(y_1,\ldots,y_q)$ for the
primal and dual block partitions. The coupling $H$ is smooth, convex
in $x$ for each fixed $y$, and concave in $y$ for each fixed $x$.
We allow nonseparable objectives $f$ and $g$, as well as nonlinear
coupling $H$ with interactions across all blocks.

The partition specifies our oracle access: each requested block of an
objective or coupling gradient is evaluated and charged separately.
In applications, primal blocks may group model parameters by learning
task or demand variables by origin, while dual blocks may correspond
to graph edges, sensor groups, or score shifts. Objective gradients can
differ in conditioning and evaluation cost, while coupling gradients can
differ in smoothness and in the input blocks on which they depend.
Our goal is to capture these differences in simultaneous complexity
guarantees for every partial-gradient oracle. The following examples
show how these block structures arise and why separate evaluation
budgets matter for the resulting optimization problems.

\paragraph{Graph-regularized learning.}
Primal blocks contain task parameters, and dual blocks represent graph
edges penalizing disagreement between neighboring tasks. Objective
partials evaluate local data losses, whose cost and conditioning vary
across tasks; coupling partials use endpoint parameters or incident edge
variables. Separate counts capture these cost differences and sparse
input dependencies. Equation~\eqref{eq:experiment-graph} formulates our
benchmark; see \citet{MultitaskGraphs} for background.

\paragraph{Inverse problems with grouped observations.}
In traffic-demand estimation, primal blocks group demand corrections by
origin, and dual blocks group link-count sensors. Submatrices of the
routing matrix determine interaction strengths and dependencies, while
cross-block prior terms allow a nonseparable objective. This motivates
separate bounds for heterogeneous objective and coupling partials.
Equation~\eqref{eq:experiment-traffic} gives our fixed-routing model
with simulated counts; see \citet{ODPrimalDual} for related models.

\paragraph{Nonlinear interactions through group score shifts.}
Dual coordinates shift logistic scores within individual tasks or across
several tasks, creating nonlinear interactions with primal parameters.
A quadratic correction ensures coupling concavity. Local and shared
shifts produce different input dependencies, motivating overlap-aware
bounds beyond bilinear coupling. Equation~\eqref{eq:experiment-nonlinear-coupling}
defines their interaction; Appendix~\ref{supp:sec:experiment-models}
provides global certificates.

These examples lead to the same question: can one run solve the coupled
problem while assigning each objective and coupling partial gradient a
complexity bound determined by its own conditioning, smoothness, and
input dependencies? A bound on full-gradient evaluations alone does not
describe how often an expensive task loss or a weak interaction must be
queried. We address this question in a general oracle model that allows
both nonseparable objectives and nonlinear coupling.

\paragraph{Related work.}
Previous separation guarantees distinguish full objective and coupling
gradients; we seek individual bounds for every partial-gradient oracle
with nonseparable objectives and nonlinear coupling.
Complexity separation for nonlinear saddle-point problems was established
by \citet{Alkousa2019} and \citet{Borodich2023}. The latter distinguish
full objective gradients from the full coupling gradient, with strong
convexity--concavity supplied by the coupling. In the bilinear setting,
\citet[Corollary~2]{borodich2025linear} separate full objective gradients
and coupling products. Our one-block specialization recovers their rates for strongly convex
$f$ and $g$ when no positive lower singular-value bounds on the coupling
are used. Their result also covers objectives without strong convexity
when the coupling supplies the required curvature, beyond our
unregularized theorem. Mirror-prox
sliding on compact variational-inequality domains separates gradient and
operator accuracy powers $\varepsilon^{-1/2}$ and $\varepsilon^{-1}$
\citep[Corollary~4]{LanOuyang2021}. We extend this separation to individual objective and coupling
partial gradients.

\paragraph{Contributions.}
Our contributions are as follows.
\begin{enumerate}[leftmargin=*,itemsep=3pt,topsep=3pt]
\item \textbf{A method using only partial gradients.}
We develop a randomized lifted extragradient method for nonlinear
convex--concave coupling and nonseparable objectives. Its gradient-coordinate
implementation uses only the four partial-gradient oracle families,
without conjugate evaluations, proximal maps, inverse-gradient operations,
or inner solves.

\item \textbf{Simultaneous bounds for every oracle.}
For smooth strongly convex $f$ and $g$, one run achieves expected
evaluation bounds of
$O(\sqrt{\kappa}\,\log(2\kappa_{\max}/\varepsilon))$ for each objective
oracle and
$O((1+\kappa_H\sqrt{\Delta})\log(2\kappa_{\max}/\varepsilon))$
for each coupling oracle (Table~\ref{tab:main-rates}). Here $\kappa$
is the objective block's condition number~(\ref*{eq:main-kappa}),
$\kappa_H$ is the coupling partial's smoothness normalized by objective
curvature~(\ref*{eq:main-coupling-row-parameters}), and $\Delta$ is its
input-overlap parameter~(\ref*{eq:main-overlap}). Thus each bound reflects
the corresponding oracle's conditioning or coupling smoothness and
dependencies across blocks.

\item \textbf{Separate rates for merely convex objectives.}
Quadratic regularization yields an expected restricted saddle gap of at
most $\varepsilon$, with objective and coupling oracle complexities of
order $\varepsilon^{-1/2}$ and $\varepsilon^{-1}$, respectively, up to
logarithmic factors (Table~\ref{tab:convex-rates}).

\item \textbf{Matching lower bounds for bilinear subclasses.}
For bilinear couplings on disjoint block pairs, we establish matching
lower bounds in the randomized blockwise linear-span model under the
stated parameter conditions. These bounds establish optimality within
those subclasses (Corollary~\ref{supp:cor:optimality} and
Table~\ref{tab:matching-rates}).
\end{enumerate}

\paragraph{Technical approach.}
A Fenchel-conjugate reformulation transfers objective curvature into a
Bregman regularizer and yields block-supported operator components even
for nonseparable objectives. Input-supported seminorms capture coupling
dependencies, while a randomized correction controls errors from cached
evaluations. The approach
draws on bilinear lifting \citep{Thekumparampil2022}, variance reduction
\citep{AlacaogluMalitsky2022}, operator decomposition \citep{Lee2026},
and square-root coordinate sampling \citep{AllenZhu2016}; related
partial-gradient separation and lower-bound techniques appear in
\citet{Kovalev2025} and \citet{OuyangXu2018}.

\section{Problem data and oracle access}\label{sec:model}
\subsection{What a call returns}
Let $x=(x_1,\ldots,x_m)$, $y=(y_1,\ldots,y_q)$, and $z=(x,y)$,
with positive finite block dimensions.
The block partition, decomposition into $(f,H,g)$, and objective and
coupling bounds defined below are known. This known decomposition does
not provide analytic formulas for the functions; access is through
Table~\ref{tab:main-oracles}. We call partials of $f$ and $g$ objective
gradients.
Their queries supply a full $x$ or $y$;
coupling queries supply $(x,y)$. Each call returns one requested block
and is charged to its own counter.

\begin{table}[t]
\caption{The four exact partial-gradient oracle families. Each requested
block is one charged call. All counters refer to one saddle-point target.}
\label{tab:main-oracles}
\centering
\begin{tabular}{@{}lll@{}}
\toprule
Returned gradient & Evaluation point & Counter\\
\midrule
$\nabla_{x_i}f(x)$ & $x\in\RR^{d_x}$ & $N_i^f$\\
$\nabla_{y_j}g(y)$ & $y\in\RR^{d_y}$ & $N_j^g$\\
$\nabla_{x_i}H(x,y)$ & $(x,y)$ & $N_{H,i}^x$\\
$\nabla_{y_j}H(x,y)$ & $(x,y)$ & $N_{H,j}^y$\\
\bottomrule
\end{tabular}
\end{table}

Known rescalings, vector arithmetic, and sampling are free. A full gradient of $f$ or $g$ costs all its partial evaluations. No function-value, conjugate,
proximal, inverse-gradient, Jacobian, or inner-solver oracle is supplied.
Assembling the signed coupling operator
$T_H=(\nabla_xH,-\nabla_yH)$ costs $m+q$ coupling partial calls.
Oracle counts measure evaluation work; storage is recorded separately.

\subsection{Joint objective curvature and accuracy}
Norms without subscripts are Euclidean; $I$ denotes the identity.
For symmetric $A\succeq0$, $\norm{v}_A^2=\ip{Av}{v}$ is a squared
seminorm; $A\preceq B$ means $B-A\succeq0$.
Unmarked logarithms are natural.
Write $h(z)=f(x)+g(y)$ and
$D_\psi(v,w)=\psi(v)-\psi(w)-\ip{\nabla\psi(w)}{v-w}$.
\providecommand{\objectiveboundcontinuation}{\\ &}
\begin{assumption}[Joint objective curvature]\label{ass:players}
The objectives $f,g$ are differentiable on their full Euclidean spaces.
Known constants satisfy $0<\mu_i^f\le L_i^f<\infty$ for every $i$
and $0<\mu_j^g\le L_j^g<\infty$ for every $j$.
For all $x,x^\prime\in\RR^{d_x}$ and $y,y^\prime\in\RR^{d_y}$,
\begin{equation}\label{eq:main-model-f}
 \begin{aligned}
 \tfrac12\sum_{i=1}^m\mu_i^f\norm{x_i-x_i^\prime}^2
 &\le D_f(x,x^\prime)\objectiveboundcontinuation
 \le\tfrac12\sum_{i=1}^m L_i^f\norm{x_i-x_i^\prime}^2,
 \end{aligned}
\end{equation}
\begin{equation}\label{eq:main-model-g}
 \begin{aligned}
 \tfrac12\sum_{j=1}^q\mu_j^g\norm{y_j-y_j^\prime}^2
 &\le D_g(y,y^\prime)\objectiveboundcontinuation
 \le\tfrac12\sum_{j=1}^q L_j^g\norm{y_j-y_j^\prime}^2.
 \end{aligned}
\end{equation}
\end{assumption}
These joint bounds allow nonseparable objectives and must hold when
several blocks change simultaneously. Bounds obtained by varying only
one block at a time do not suffice.
Write $M_x=\diag(\mu_i^f I)$, $L_x=\diag(L_i^f I)$,
$M_y=\diag(\mu_j^g I)$, and $L_y=\diag(L_j^g I)$.
Put $M=\diag(M_x,M_y)$. Define
\begin{equation}\label{eq:main-kappa}
 \begin{gathered}
 \kappa_i^f=L_i^f/\mu_i^f,\qquad \kappa_j^g=L_j^g/\mu_j^g,\\
 \kappa_{\max}=\max\{\max_i\kappa_i^f,\max_j\kappa_j^g\}.
 \end{gathered}
\end{equation}
In coordinates $\widehat z=M^{1/2}z$, lower curvature is the identity.
For initial point $z^0$, saddle point $z^\star$, output $z^{\out}$,
and $0<\varepsilon\le1/2$, the target is
\begin{equation}\label{eq:main-target}
 \E\norm{z^{\out}-z^\star}_M^2
     \le\varepsilon\norm{z^0-z^\star}_M^2.
\end{equation}
Expectations are over the method's randomness. Neither $z^\star$ nor the initial distance is an
algorithm input. Normalized evaluations use one original evaluation and
known input and output rescalings.

\subsection{Coupling rows and shared inputs}
\begin{assumption}[Smooth convex--concave coupling]\label{ass:coupling}
The function $H$ is continuously differentiable, convex in the entire
vector $x$ for each $y$, and concave in the entire vector $y$ for each $x$.
A saddle point exists. Known finite nonnegative arrays
$L^{H,xx},L^{H,xy},L^{H,yx},L^{H,yy}$ satisfy, for every
$z=(x,y)$ and $z^\prime=(x^\prime,y^\prime)$,
\providecommand{\couplingboundcontinuation}{\\ &\quad}
\begin{equation}\label{eq:main-coupling-model}
\begin{aligned}
 \norm{\nabla_{x_i}H(z)-\nabla_{x_i}H(z^\prime)}
 &\le\sum_t L_{it}^{H,xx}\norm{x_t-x_t^\prime}\couplingboundcontinuation
 +\sum_\ell L_{i\ell}^{H,xy}\norm{y_\ell-y_\ell^\prime},\\
 \norm{\nabla_{y_j}H(z)-\nabla_{y_j}H(z^\prime)}
 &\le\sum_t L_{jt}^{H,yx}\norm{x_t-x_t^\prime}\couplingboundcontinuation
 +\sum_\ell L_{j\ell}^{H,yy}\norm{y_\ell-y_\ell^\prime},
\end{aligned}
\end{equation}
for every $i,j$, with $1\le t\le m$ and $1\le\ell\le q$ in the sums.
\end{assumption}
The first array index is the output block; the second is the input block.
These global bounds require no Hessian.
Define the normalized coupling parameters $\kappa_{H,i}^x$ and
$\kappa_{H,j}^y$ by
\begin{equation}\label{eq:main-coupling-row-parameters}
 \begin{aligned}
 \kappa_{H,i}^x
 &:=\left[\sum_{t=1}^m\frac{(L_{it}^{H,xx})^2}{\mu_i^f\mu_t^f}
       +\sum_{\ell=1}^q\frac{(L_{i\ell}^{H,xy})^2}{\mu_i^f\mu_\ell^g}\right]^{1/2},\\
 \kappa_{H,j}^y
 &:=\left[\sum_{t=1}^m\frac{(L_{jt}^{H,yx})^2}{\mu_j^g\mu_t^f}
       +\sum_{\ell=1}^q\frac{(L_{j\ell}^{H,yy})^2}{\mu_j^g\mu_\ell^g}\right]^{1/2}.
 \end{aligned}
\end{equation}
These certified bounds normalize both input and output curvature;
the coupling parameters can be below one. The quantity $\kappa_{\max}$
in~\eqref{eq:main-kappa} uses only objective parameters.

For $\nabla_{x_i}H$, let $I_{H,i}^x$ and $J_{H,i}^x$ denote the
index sets of $x$- and $y$-blocks with positive Lipschitz coefficients
in~\eqref{eq:main-coupling-model}; define $I_{H,j}^y$ and $J_{H,j}^y$
analogously for $\nabla_{y_j}H$. These sets specify the block dependency
structure of the coupling partial gradients. Let $\nu_t^x$ and $\nu_\ell^y$
be the numbers of coupling partial gradients whose dependency sets
include $x_t$ and $y_\ell$, respectively. The block overlap parameter
$\Delta_{H,i}^x$ is the maximum of these counts over the dependency
sets of $\nabla_{x_i}H$, bounded below by one:
\begin{equation}\label{eq:main-overlap}
 \Delta_{H,i}^x=\max\left\{1,
 \max_{t\in I_{H,i}^x}\nu_t^x,\ \max_{\ell\in J_{H,i}^x}\nu_\ell^y\right\},
\end{equation}
with the analogous definition for $\Delta_{H,j}^y$; an empty inner
maximum is zero, so an empty support has $\Delta=1$. These certified supports and overlaps may overestimate
actual dependence; they are independent of block dimensions.
Write $P_{I,J}$ for the projection retaining those input blocks.
Appendix~\ref{supp:sec:model} proves the local row and full Lipschitz bounds.
Convex--concavity makes $T_H$ monotone; positive joint lower curvature
makes $\nabla h+T_H$ strongly monotone, so its zero is the unique
saddle point. Empty rows are constant derivatives; the method treats
them by the same sampling and cache updates as every other row.
Example~\ref{supp:ex:dense} gives a globally smooth nonlinear coupling
with overlap $m+q$ that cannot generally be split into pairwise block terms.

\section{Simultaneous oracle guarantees}\label{sec:main}
The key distinction is between objective conditioning, coupling
smoothness, and input overlap. Table~\ref{tab:main-rates} separates their
effects; the logarithmic factor $\log(2\kappa_{\max}/\varepsilon)$
is common to every row.

\begin{theorem}[Separate counts for every partial gradient]\label{thm:strong}
Under Assumptions~\ref{ass:players} and~\ref{ass:coupling}, the method in
Section~\ref{sec:algorithm} attains~\eqref{eq:main-target} and all bounds
in Table~\ref{tab:main-rates} in the same run. These are expected counts
with universal hidden constants, including initialization and corrections.
The method uses
\begin{equation}\label{eq:main-epochs}
 E=\left\lceil\log_2\frac{2\kappa_{\max}}{\varepsilon}\right\rceil
 \quad\text{stages}.
\end{equation}
Coupling rows with zero declared Lipschitz bounds use $O(E)$ calls;
their partial gradients may be nonzero constants.
Appendix~\ref{supp:sec:main} gives the numerical constants.
\end{theorem}

\ifdefined\mainratestabular\else
\newenvironment{mainratestabular}{%
  \small\renewcommand{\arraystretch}{1.25}%
  \begin{tabular*}{\columnwidth}{@{}l@{\extracolsep{\fill}}l@{}}%
}{\end{tabular*}}
\fi
\providecommand{\mainratesrowspace}{}
\providecommand{\beginmainratestable}{\begin{table}[H]}
\beginmainratestable
\caption{Expected calls to each oracle for the relative squared-distance
target~\eqref{eq:main-target}. All four bounds hold in the same run.}
\label{tab:main-rates}
\centering
\begin{mainratestabular}
\toprule
\textbf{Oracle} & \textbf{Expected calls}\\
\midrule
$\nabla_{x_i}f$ & $O\!\left(\sqrt{\kappa_i^f}\,\log\frac{2\kappa_{\max}}{\varepsilon}\right)$\\
\mainratesrowspace
$\nabla_{y_j}g$ & $O\!\left(\sqrt{\kappa_j^g}\,\log\frac{2\kappa_{\max}}{\varepsilon}\right)$\\
\midrule
$\nabla_{x_i}H$ & $O\!\left(\left(1+\kappa_{H,i}^x\sqrt{\Delta_{H,i}^x}\right)\log\frac{2\kappa_{\max}}{\varepsilon}\right)$\\
\mainratesrowspace
$\nabla_{y_j}H$ & $O\!\left(\left(1+\kappa_{H,j}^y\sqrt{\Delta_{H,j}^y}\right)\log\frac{2\kappa_{\max}}{\varepsilon}\right)$\\
\bottomrule
\end{mainratestabular}
\end{table}

\paragraph{Dense and sparse inputs.}
Full input dependence gives $\Delta_{H,i}^x=\Delta_{H,j}^y=m+q$.
For a matching $\mathcal E$ of primal--dual pairs (no block appears twice),
let $H=\sum_{(i,j)\in\mathcal E}\ip{y_j}{B_{ji}x_i}$, where
$B_{ji}:\RR^{\dim(x_i)}\to\RR^{\dim(y_j)}$ is linear.
Each nonzero derivative uses only its opposite endpoint, so $\Delta=1$.
Disjoint nonlinear pair interactions have $\Delta\le2$.
Bounded input overlap therefore removes the growing block-count factor
from these coupling coefficients, including for nonlinear interactions.
Appendix~\ref{supp:sec:coupling-blocks} gives examples and the assembly comparison.

\subsection{Sampling versus full assembly}
For positive costs $\tau_i^f,\tau_j^g,\tau_{H,i}^x,\tau_{H,j}^y$,
define
\providecommand{\workboundcontinuation}{\\ &}
\begin{equation}\label{eq:main-work}
 \begin{split}
 W_\tau={}&\sum_i\tau_i^fN_i^f+\sum_j\tau_j^gN_j^g\workboundcontinuation
 +\sum_i\tau_{H,i}^xN_{H,i}^x+\sum_j\tau_{H,j}^yN_{H,j}^y.
 \end{split}
\end{equation}
Multiplying Theorem~\ref{thm:strong}'s counts by their costs and adding
gives a weighted work bound; no independence among counters is needed.
The method need not know these costs.

\paragraph{Full assembly.}
The same oracles can assemble $T_H$ as one coupling component.
Each assembly costs $m+q$ partial calls; all its coupling counters are
equal. Sampling can favor small row constants and sparse supports;
assembly can favor dense inputs. These are alternative runs: their
componentwise minimum is not guaranteed in one run.
Appendix~\ref{supp:sec:coupling-blocks} gives the global coupling bound,
the assembly guarantee, and a weighted certificate criterion.

\paragraph{A concrete cost comparison.}
For $M=I$ and $H=100x_1y_1+x_2y_2$, strong-edge partials cost $1$
and weak-edge partials cost $100$. All overlaps equal one. Sampling
reduces calls to the weak, expensive derivatives; full assembly calls
all four equally often.

\section{The method: lifting, input seminorms, and sampling}\label{sec:algorithm}
We first normalize by objective curvature, then lift the problem to
transfer that curvature into a regularizer. We allocate input metrics
to the lifted components and sample them in an extragradient update.
Work in $M^{1/2}$ coordinates, suppressing hats, with known rescalings
around each original oracle call.
The component set
$\cI$ contains one objective component associated with $f$ or $g$ per
variable block and one coupling component per row, including zero bounds. For an
objective component $k=(f,i)$ or $(g,j)$, let
$z_k,\nabla_kh,\kappa_k$ denote the corresponding block, partial
gradient, and condition number from~\eqref{eq:main-kappa}.
Let $U_i^x$ and $U_j^y$ insert one block into a full $z$ vector,
with zeros elsewhere; their transposes extract that block.
For objective components use $U_{(f,i)}=U_i^x$ and $U_{(g,j)}=U_j^y$.

\subsection{Lifted formulation and gradient coordinates}
Conjugate lifting moves objective curvature into $r$, making components
local in $w=(z,u)$ even for nonseparable $h$. Queries still receive
full $a^x$ or $a^y$. Set
\ifdefined\liftedformulationequations
\liftedformulationequations
\else
\begin{align}
 \phi(z)&=h(z)-\tfrac14\norm{z}^2,
 &r(z,u)&=\tfrac14\norm{z}^2+\phi^*(u),\label{eq:main-lift}\\
 \lambda_k&=\kappa_k-\tfrac12,
 &D&=\diag(\tfrac12I,\Lambda^{-1}),\label{eq:main-D}
\end{align}
\fi
Here $\phi^*(u)=\sup_v\{\ip{u}{v}-\phi(v)\}$ is the convex conjugate,
and $\Lambda=\diag(\lambda_k I)$ over the objective components.
The joint model implies $D_r(v,w)\ge\norm{v-w}_D^2/2$.
The lifted operator
\begin{equation}\label{eq:main-Jlift}
 \cJ(z,u)=(T_H(z)+u,-z)
\end{equation}
is monotone: its cross terms cancel, leaving the pairing for $T_H$.
The lifted solution $w^\star=(z^\star,\nabla\phi(z^\star))$ satisfies
$\nabla r(w^\star)+\cJ(w^\star)=0$.

The gradient coordinate and corresponding lifted point are
\[
 \theta=(z/2,a)=\nabla r(w),\qquad
 w(\theta)=(z,\nabla\phi(a)).
\]
Here $a=\nabla\phi^*(u)$. We store $z,a$; $u=\nabla\phi(a)$ is
implicit, and the conjugate serves only the analysis.
Table~\ref{tab:main-components} defines the split through
$G_k=\cJ_k\circ w$, equivalently $\cJ_k=G_k\circ\nabla r$.
Objective components sum to $(u,-z)$; coupling components sum to
$(T_H(z),0)$.

\ifdefined\maincomponentstabular\else
\newenvironment{maincomponentstabular}{%
  \footnotesize\renewcommand{\arraystretch}{1.3}%
  \tabularx{\columnwidth}{@{}l X@{}}%
}{\endtabularx}
\fi
\providecommand{\maincomponentsgroupseparator}{}
\providecommand{\maincomponentsheader}{Component $k$ & $G_k(\theta)$\\}
\begin{table}[H]
\caption{Normalized component values at $\theta=(z/2,a)$; one original partial-gradient call per row, with known rescalings.}\label{tab:main-components}
\centering
\begin{maincomponentstabular}
\toprule
\maincomponentsheader
\midrule
$(f,i)$ & $\bigl(U_i^x[\nabla_{x_i}f(a^x)-a_i^x/2],-U_i^xx_i\bigr)$\\
$(g,j)$ & $\bigl(U_j^y[\nabla_{y_j}g(a^y)-a_j^y/2],-U_j^yy_j\bigr)$\\
\maincomponentsgroupseparator
$(H,x,i)$ & $\bigl(U_i^x\nabla_{x_i}H(z),0\bigr)$\\
$(H,y,j)$ & $\bigl(-U_j^y\nabla_{y_j}H(z),0\bigr)$\\
\bottomrule
\end{maincomponentstabular}
\end{table}

All rows use the same rules, without separately detecting constant
partial gradients, and
$\cJ(w(\theta))=\sum_{k\in\cI}G_k(\theta)$.

\subsection{Input seminorms and sampling}
For $H=xy$, the difference between $\nabla_xH$ at $(x,y)$ and a
cached value evaluated at $(x',y')$ is $y-y'$. This error is measured on the $y$
input, although the derivative is returned in the $x$ block.
Choose input metrics $D_k\succeq0$ and constants $c_k>0$
to satisfy both a shared budget and a component variation bound:
\begin{align}
 \sum_{k\in\cI}D_k&\preceq D,\label{eq:main-allocation}\\
 \norm{\cJ_k(w)-\cJ_k(v)}_{D^{-1}}^2
 &\le c_k^2\norm{w-v}_{D_k}^2.\label{eq:main-component-bound}
\end{align}
The seminorms measure input changes in the lifted coordinates $w=(z,u)$;
the variation bound holds for all $w,v$.

For objective component $k$, let $\Pi_k$ orthogonally project onto block
$k$ in both lifted coordinates and set $D_k=\Pi_kD\Pi_k/2$.
For coupling components, set
\ifdefined\couplinginputmetricequation
\couplinginputmetricequation
\else
\begin{equation}\label{eq:main-input-metric}
 \begin{aligned}
 D_{H,i}^x&=\diag\left(\frac{P_{I_{H,i}^x,J_{H,i}^x}}{4\Delta_{H,i}^x},0\right),\\
 D_{H,j}^y&=\diag\left(\frac{P_{I_{H,j}^y,J_{H,j}^y}}{4\Delta_{H,j}^y},0\right).
 \end{aligned}
\end{equation}
\fi
Empty supports give $D_k=0$. Objective metrics sum to $D/2$.
An input in $\nu>0$ coupling supports has $\Delta\ge\nu$ for each
incident row, so their total weight is at most $\nu/(4\nu)=1/4$.
Coupling metrics thus use at most half of $D$'s $z$-coordinate metric
and none of its $u$-coordinate metric, proving~\eqref{eq:main-allocation}.

For these metrics, the component variation bounds hold with
\[
 c_k=\begin{cases}
 \sqrt{4\kappa_i^f-2},&k=(f,i),\\
 \sqrt{4\kappa_j^g-2},&k=(g,j),\\
 \max\{1,2\sqrt2\kappa_{H,i}^x\sqrt{\Delta_{H,i}^x}\},&k=(H,x,i),\\
 \max\{1,2\sqrt2\kappa_{H,j}^y\sqrt{\Delta_{H,j}^y}\},&k=(H,y,j).
 \end{cases}
\]
The floor ensures positive sampling weights.
Appendix~\ref{supp:sec:algorithm} verifies the bounds. Choose
\ifdefined\samplingparametersequations
\samplingparametersequations
\else
\begin{align}
 S&=\sum_k c_k,&p_k&=c_k/S,\label{eq:main-sampling}\\
\eta&=1/(8S),&K&=\lceil18S\rceil.\label{eq:main-parameters}
\end{align}
\fi
With $K=O(S)$ and $c_k\ge1$, expected per-stage counts are
$O(1+p_kK)=O(c_k)$. Section~\ref{sec:analysis} establishes stage
contraction and the number of stages needed for Theorem~\ref{thm:strong}.

\subsection{Variance-reduced stochastic extragradient}
Store component evaluations
$g_k^t=G_k(\theta_{s_k(t)})$ and auxiliary vectors
$r_k^t=R_k(\theta_{s_k(t)})$, where $s_k(t)$ is the last stored
iteration for component $k$, initially zero.
Define $R_k(\theta):=D_kw(\theta)$ by
\begin{equation}\label{eq:main-memory-difference}
 R_k(\theta)=\begin{cases}
 (U_kz_k/4,U_ku_k/(2\lambda_k)),&\text{objective},\\
 (P_{I_k,J_k}z/(4\Delta_k),0),&\text{coupling}.
 \end{cases}
\end{equation}
Here $U_k$ embeds block $k$. Objective components use
$u_k=\nabla_kh(a)-a_k/2$; coupling components use input support
$(I_k,J_k)$ and overlap bound $\Delta_k$, with $R_k=0$ for empty
supports. Evaluating $(G_k,R_k)$ uses one partial-gradient call
and does not require the full $u$. Each stage initializes every
table entry with one counted evaluation.

The stored sum $g^t=\sum_kg_k^t$ serves as a control variate.
Each iteration draws $i_t,j_t$ independently from $p$.
Sample $i_t$ supplies the memory correction
$d_t=R_{i_t}(\theta_t)-r_{i_t}^t$ and the table update.
Sample $j_t$ supplies the importance-sampled estimator
$\widehat g_t$ at the predictor $\bar\theta_t$.
Conditional on the current iterate, table, and $i_t$, it is unbiased
for $\sum_kG_k(\bar\theta_t)$. Each fresh evaluation is reused below.

\begin{algorithm}[H]
\caption{Variance-reduced extragradient stage}\label{alg:epoch}
\begin{algorithmic}[1]\footnotesize
\Require $\theta^{\inpt}$; parameters~\samplingparametersreference.
\State $\theta_0\gets\theta^{\inpt}$; $\tau\sim\operatorname{Unif}\{0,\ldots,K-1\}$.
\State $(g_k^0,r_k^0)\gets(G_k(\theta_0),R_k(\theta_0))$ \quad$(k\in\cI)$.
\State $g^0\gets\sum_{k\in\cI}g_k^0$.
\For{$t=0,\ldots,K-1$}
 \State $i_t,j_t\overset{\mathrm{iid}}{\sim}p$.
 \State Evaluate $G_{i_t}(\theta_t)$ and $R_{i_t}(\theta_t)$.
 \State $d_t\gets R_{i_t}(\theta_t)-r_{i_t}^t$.
 \State $\displaystyle\bar\theta_t\gets\frac{\theta_t-d_t/8-\eta g^t}{1+\eta}$. \Comment{Predictor step}
 \State \textbf{if} $t=\tau$ \textbf{then} $\theta^{\out}\gets\bar\theta_t$.
 \State $\displaystyle\widehat g_t\gets g^t+\frac{G_{j_t}(\bar\theta_t)-g_{j_t}^t}{p_{j_t}}$.
 \State $\theta_{t+1}\gets\bar\theta_t-\eta(\widehat g_t-g^t)$. \Comment{Corrector step}
 \State $g^{t+1}\gets g^t-g_{i_t}^t+G_{i_t}(\theta_t)$.
 \State $(g_{i_t}^{t+1},r_{i_t}^{t+1})\gets(G_{i_t}(\theta_t),R_{i_t}(\theta_t))$.
 \State $(g_k^{t+1},r_k^{t+1})\gets(g_k^t,r_k^t)$ \quad$(k\ne i_t)$.
\EndFor
\State \Return $\theta^{\out}$.
\end{algorithmic}
\end{algorithm}

Both steps use the old table. After correction, update only entry
$i_t$ using its evaluation at $\theta_t$. Each iteration counts two
partial-gradient calls, even if $i_t=j_t$ or the query points coincide.
Complete $K$ iterations and return a uniformly selected predictor iterate.
Appendix~\ref{supp:sec:records} gives implementation details.

Initialize $\theta^{(0)}=(\widehat z^0/2,\widehat z^0)$ and, for
$e=0,\ldots,E-1$ with $E=\lceil\log_2(2\kappa_{\max}/\varepsilon)\rceil$,
set $\theta^{(e+1)}=\operatorname{Stage}(\theta^{(e)})$.
Each stage restarts at the previous output, retaining both coordinates
$\theta=(z/2,a)$ to preserve the Bregman divergence
in~\eqref{supp:eq:physical-divergence}.
Return $z^{\out}=M^{-1/2}(2\theta_z^{(E)})$.
Zero initialization preserves the blockwise span restriction.

\subsection{Convergence and separate evaluation bounds}\label{sec:analysis}
The local-memory correction cancels cache transfers, while importance
sampling controls correction variance, giving the following guarantee.

\begin{lemma}[Stage contraction and component evaluation counts]\label{lem:descent}
Under Assumptions~\ref{ass:players} and~\ref{ass:coupling}, use
geometry~\eqref{eq:main-lift}--\eqref{eq:main-D} and
parameters~\samplingparametersreference.
Algorithm~\ref{alg:epoch} satisfies
\begin{equation}\label{eq:main-epoch}
 \E[D_r(w^\star,w(\theta^{\out}))\mid\theta^{\inpt}]
 \le\tfrac12D_r(w^\star,w(\theta^{\inpt})).
\end{equation}
For $E$ stages, let $N_k$ count every evaluation of component $k$. Then
\begin{equation}\label{eq:main-count}
 \E N_k=E(1+2p_kK)\le(3+36c_k)E.
\end{equation}
\end{lemma}

Iterating the contraction and converting lifted divergence to distance
yields Theorem~\ref{thm:strong}, with each $N_k$ counting its original
oracle. Appendix~\ref{supp:sec:analysis} gives the proof and divergence
comparisons;~\eqref{supp:eq:coupling-block-memory} gives storage bounds.

\section{Lower bounds and the matching regimes}\label{sec:lower}
Algorithms start at zero, query and output each block within the span of
gradients previously returned to that block, and stop almost surely.
Coefficients and stopping decisions may depend on past observations and
randomness (Definition~\ref{supp:def:span}). The problem class satisfies
Assumptions~\ref{ass:players} and~\ref{ass:coupling}, with the same fixed
objective bounds and coupling arrays as the upper bounds; these need not
be attained (Definition~\ref{supp:def:strong-class}).
For a matching $\mathcal E$ of primal--dual pairs,
define the admissible chain strength
\begin{equation}\label{eq:main-edge-strength}
 c_{ij}=\frac{\min\{L_{ij}^{H,xy},L_{ji}^{H,yx}\}}
                 {\sqrt{\mu_i^f\mu_j^g}}.
\end{equation}

Theorem~\ref{supp:thm:strong-lower} gives a cost-weighted lower bound
on one quadratic bilinear instance, together with separate oracle lower
bounds. The instance and finite block dimensions may depend on accuracy
and positive costs, but not on the realized random seed.

\paragraph{Matching regime and scope.}
Suppose $\kappa_i^f,\kappa_j^g\ge4$ for all objective blocks,
$c_{ij}\ge2$ on a nonempty matching $\mathcal E$, and
$\varepsilon\le\min\{1/128,1/\kappa_{\max}\}$. If the prescribed
array contains only symmetric cross entries on $\mathcal E$ and all
other entries are zero, then $c_{ij}=\kappa_{H,i}^x=\kappa_{H,j}^y$ and
$\Delta_{H,i}^x=\Delta_{H,j}^y=1$ on every matched pair.
Corollary~\ref{supp:cor:optimality} yields Table~\ref{tab:matching-rates}.

\ifdefined\mainmatchingtabular\else
\newenvironment{mainmatchingtabular}{%
  \small\renewcommand{\arraystretch}{1.25}%
  \begin{tabular*}{\columnwidth}{@{}l@{\extracolsep{\fill}}ll@{}}%
}{\end{tabular*}}
\fi
\begin{table}[!htbp]
\caption{Matching expected oracle counts for~\eqref{eq:main-target}
in the bilinear regime of Section~\ref{sec:lower}; see
Corollary~\ref{supp:cor:optimality}.
Coupling bounds shown apply to matched pairs.}
\label{tab:matching-rates}
\centering
\begin{mainmatchingtabular}
\toprule
\textbf{Oracle} & \textbf{Upper bound} & \textbf{Lower bound}\\
\midrule
$\nabla_{x_i}f$ & $O\bigl(\sqrt{\kappa_i^f}\log(1/\varepsilon)\bigr)$ & $\Omega\bigl(\sqrt{\kappa_i^f}\log(1/\varepsilon)\bigr)$\\
\mainratesrowspace
$\nabla_{y_j}g$ & $O\bigl(\sqrt{\kappa_j^g}\log(1/\varepsilon)\bigr)$ & $\Omega\bigl(\sqrt{\kappa_j^g}\log(1/\varepsilon)\bigr)$\\
\midrule
$\nabla_{x_i}H$ & $O\bigl(\kappa_{H,i}^x\log(1/\varepsilon)\bigr)$ & $\Omega\bigl(\kappa_{H,i}^x\log(1/\varepsilon)\bigr)$\\
\mainratesrowspace
$\nabla_{y_j}H$ & $O\bigl(\kappa_{H,j}^y\log(1/\varepsilon)\bigr)$ & $\Omega\bigl(\kappa_{H,j}^y\log(1/\varepsilon)\bigr)$\\
\bottomrule
\end{mainmatchingtabular}
\end{table}

Individual lower bounds may use distinct randomized blockwise span-model
instances; both coupling endpoint bounds hold on one instance per pair.
Weighted work matches up to $O(E)$ calls per unmatched coupling row;
general overlap optimality is not claimed. Appendix~\ref{supp:sec:lower}
gives the full theorem, construction, and aggregate one-edge bound.

\section{Convex--concave problems}\label{sec:convex}
Quadratic regularization extends Section~\ref{sec:algorithm} to convex
objectives and converts distance accuracy into a restricted-gap guarantee.

\begin{assumption}[Convex problem data]\label{ass:convex}
Differentiable convex $f$ and $g$ satisfy
\eqref{eq:main-model-f}--\eqref{eq:main-model-g} with lower matrices zero
and known finite positive upper-model weights; $H$ satisfies
Assumption~\ref{ass:coupling}. Compact product comparison sets
$Z=X\times Y$ contain a saddle point and satisfy
$\sup_{v\in Z}\norm{v-z^0}^2\le R^2$ for a supplied $R>0$.
All norms here are Euclidean; oracles are those of
Table~\ref{tab:main-oracles}.
\end{assumption}
The sets define the restricted gap without constraining the iterates:
\begin{equation}\label{eq:main-gap}
 \Gap(x,y)=\sup_{\eta\in Y}\cL(x,\eta)
         -\inf_{\xi\in X}\cL(\xi,y).
\end{equation}
The target is $\E\Gap(z^{\out})\le\varepsilon$, with expectation
outside the supremum.

The Euclidean row bounds $L_{H,i}^x,L_{H,j}^y$ are the expressions
in~\eqref{eq:main-coupling-row-parameters} with every curvature
denominator replaced by one. Supports and overlaps remain those of
Section~\ref{sec:model}. Put
$L_{\max}=\max\{\max_i L_i^f,\max_j L_j^g\}$ and let $L_F$ be the
certified Euclidean saddle-operator Lipschitz bound
in~\eqref{supp:eq:Lglobal}. The common logarithmic factor is
$\mathsf B_\varepsilon=\log(2+L_FR^2/\varepsilon)$.

\begin{theorem}[Restricted-gap guarantee]\label{thm:convex}
Under Assumption~\ref{ass:convex}, for every $\varepsilon>0$ the quadratic
regularization procedure below attains $\E\Gap(z^{\out})\le\varepsilon$.
If $\varepsilon/R^2$ is at most every objective bound $L_i^f,L_j^g$
and every coupling product $L_{H,i}^x\sqrt{\Delta_{H,i}^x}$,
$L_{H,j}^y\sqrt{\Delta_{H,j}^y}$ for positive rows, the same run obeys Table~\ref{tab:convex-rates}.
All original partial-gradient calls are counted. Coupling rows with zero
declared Lipschitz bounds use $O(\mathsf B_\varepsilon)$ calls. Appendix~\ref{supp:sec:convex}
gives numerical bounds valid for every $\varepsilon>0$.
\end{theorem}

\ifdefined\mainconvextabular\else
\newenvironment{mainconvextabular}{%
  \small\renewcommand{\arraystretch}{1.4}%
  \tabularx{\columnwidth}{@{}p{0.27\columnwidth}X@{}}%
}{\endtabularx}
\fi
\begin{table}[!htbp]
\caption{Simultaneous expected oracle counts for
$\E\Gap(z^{\out})\le\varepsilon$ under Theorem~\ref{thm:convex}.
Coupling rows shown have positive declared Lipschitz bounds.}
\label{tab:convex-rates}
\centering
\begin{mainconvextabular}
\toprule
\textbf{Oracle} & \textbf{Expected calls}\\
\midrule
$\nabla_{x_i}f$ & $O\!\bigl(R\sqrt{L_i^f/\varepsilon}\,\mathsf B_\varepsilon\bigr)$\\
\addlinespace[4pt]
$\nabla_{y_j}g$ & $O\!\bigl(R\sqrt{L_j^g/\varepsilon}\,\mathsf B_\varepsilon\bigr)$\\
\midrule
$\nabla_{x_i}H$ & $O\!\bigl(L_{H,i}^x\sqrt{\Delta_{H,i}^x}\,\tfrac{R^2}{\varepsilon}\,\mathsf B_\varepsilon\bigr)$\\
\addlinespace[4pt]
$\nabla_{y_j}H$ & $O\!\bigl(L_{H,j}^y\sqrt{\Delta_{H,j}^y}\,\tfrac{R^2}{\varepsilon}\,\mathsf B_\varepsilon\bigr)$\\
\bottomrule
\end{mainconvextabular}
\end{table}

\paragraph{Regularization.}
Set $\sigma=\varepsilon/R^2$ and add
$\sigma\norm{x-x^0}^2/2-\sigma\norm{y-y^0}^2/2$
to the saddle objective. Run Section~\ref{sec:algorithm} for
$O(\mathsf B_\varepsilon)$ stages with the distance tolerance
in~\eqref{supp:eq:regparams}. Each objective partial uses one original
call plus an affine term; coupling oracles are unchanged. The wrapper
uses the comparison sets only through $R$ and requires no gap or
projection oracle. Appendix~\ref{supp:sec:convex} gives the gap-transfer
proof and full procedure.

\paragraph{Lower bounds.}
For bilinear couplings on disjoint block pairs with symmetric declared
cross entries and all other entries zero,
Theorem~\ref{supp:thm:convex-lower} matches the objective
$\varepsilon^{-1/2}$ and coupling $\varepsilon^{-1}$ rates and their
blockwise coefficients up to the upper logarithmic factor, under its
parameter conditions and blockwise linear-span restrictions.
Every objective block in the construction has zero strong convexity.

\section{Numerical experiments}\label{sec:experiments}
We test three problem families with six base configurations and four
loss-weight variants. Appendix~\ref{supp:sec:experiment-models} gives
data construction, parameter values, and global certificates.

\paragraph{Traffic-demand estimation.}\label{sec:experiment-models}
On Sioux Falls and Anaheim, $x\in\RR^{48}$ has four demand-correction
blocks and $y\in\RR^{32}$ four sensor blocks. Fixed routing gives a
matrix $C$ and simulated residual counts $b$. We solve
\begin{equation}\label{eq:experiment-traffic}
 \min_x\max_y
 \left\{\frac12x^\top Qx+y^\top(Cx-b)
                    -\frac{0.4}{2}\norm{y}^2\right\}.
\end{equation}
The prior $Q=D+0.04vv^\top$, with diagonal $D\succ0$ and
$v=\boldsymbol1_{48}/\sqrt{48}$, is nonseparable.
Positivity is not enforced; coupling overlap is $\Delta=4$.

\paragraph{Learning objectives.}
Four artificial partitions of a9a Adult data
\citep{LIBSVMData,AdultData} have sizes $(128,256,512,1024)$ and
sixteen features including an intercept. Both learning families use
$\ell(s;c)=\log(1+e^s)-cs$ for binary labels and
\begin{equation}\label{eq:experiment-local-objective}
 \begin{aligned}
 f(x)&=\sum_{i=1}^4 f_i(x_i),\\
 f_i(x_i)&=0.1\norm{x_i}^2+
 \frac{w_i}{n_i}\sum_{r=1}^{n_i}\ell((X_ix_i)_r;c_{ir}).
 \end{aligned}
\end{equation}
Weights give certified condition numbers $(2,4,8,20)$.
The graph-cycle and nonlinear-local sweeps change only $w_4$, raising
$\kappa_4^f$ to $200$ and $2000$.

\paragraph{Graph-regularized learning.}
For a four-task cycle or hub, set $(Cx)_e=\sqrt{a_e}(x_i-x_j)$
on edge $e=(i,j)$, with $a_e\in[0.25,4]$ and $y_e\in\RR^{16}$:
\begin{equation}\label{eq:experiment-graph}
 \min_x\max_y
 \left\{f(x)+y^\top Cx-\frac{0.5}{2}\norm{y}^2\right\}.
\end{equation}
Cycle overlap is $2$; hub primal and dual overlaps are $2$ and $3$.

\paragraph{Nonlinear group score shifts.}
With the same $f$ and $g(y)=0.1\norm{y}^2$, set $N=512$ and
\providecommand{\nonlinearcouplingcontinuation}{\\ &}
\begin{equation}\label{eq:experiment-nonlinear-coupling}
 \begin{aligned}
 H(x,y)={}&\frac1N\sum_{i=1}^4\sum_{r=1}^{128}
 \ell((\widetilde X_ix_i+B_iy)_r;\widetilde c_{ir})\nonlinearcouplingcontinuation
 -\frac1{8N}\sum_{i=1}^4\norm{B_iy}^2.
 \end{aligned}
\end{equation}
Here $\widetilde X_i,\widetilde c_i$ use 128 scaled task rows and their
labels. Four local scalar shifts give
$B_iy=b_i y_i\boldsymbol1_{128}$; a fifth shared shift adds
$y_5\boldsymbol1_{128}$. The bound $\ell''\le1/4$ ensures
convex--concavity. Overlaps are $2$ and $9$, with dimension and
coupling variation also changing.

\paragraph{Protocol.}
We compare prescribed sampled and assembly methods with fused EG
using $\gamma=0.9/(\kappa_{\max}+L_H)$ and applicable primal AGD.
Five runs from zero stop at relative squared $M$-distance $10^{-6}$ to
an independent Newton reference. Stage methods check after complete
stages; EG/AGD every ten iterations, giving different overshoots.
Initialization is charged, diagnostics are untimed, and runtime ratios
report bootstrap standard errors. This observed stopping rule is
distinct from an expected-error guarantee.

\paragraph{Results.}
\begin{table}[!htbp]
\centering\footnotesize\setlength{\tabcolsep}{3pt}
\caption{Joint rerun against EG $0.9$: sampled (S), assembly (A). Entries divide five-run means; runtime ratios show bootstrap SE. Stages are checked on completion, EG every ten iterations.}
\label{tab:experiments}
\begin{tabular}{@{}lrlrrrr@{}}\toprule
Case & $\kappa_4$ & & $N^f$ & $N_H^x$ & $N_H^y$ & Time\\\midrule
Graph & 20 & S & 3.99 & 16.45 & 15.57 & $21.07\,\pm\,1.66$\\
Graph & 20 & A & 3.83 & 20.10 & 20.10 & $10.39\,\pm\,0.76$\\
Graph & 2000 & S & 0.30 & 0.26 & 0.24 & $0.84\,\pm\,0.09$\\
Graph & 2000 & A & 0.27 & 0.29 & 0.29 & $0.60\,\pm\,0.07$\\
Nonlin. & 20 & S & 4.13 & 4.43 & 5.29 & $9.71\,\pm\,0.36$\\
Nonlin. & 20 & A & 3.97 & 10.51 & 10.51 & $9.48\,\pm\,0.79$\\
Nonlin. & 2000 & S & 0.40 & 0.09 & 0.11 & $0.62\,\pm\,0.03$\\
Nonlin. & 2000 & A & 0.42 & 0.23 & 0.23 & $0.63\,\pm\,0.04$\\
\bottomrule\end{tabular}
\end{table}

EG is faster on all six base cases. At the $2000$ endpoints in
Table~\ref{tab:experiments}, sampled runtime is $84\%$ of EG's on
graph and $62\%$ on nonlinear. Assembly is faster on graph;
nonlinear stage times are close. AGD remains fastest on bilinear cases.
The sweep changes the saddle and absolute cutoff, but fixed-saddle
controls retain the crossover. Individual counts show easy-block
savings (Table~\ref{tab:individual-main}).
Appendix~\ref{supp:sec:experiment-controls} reports overlap and cost
controls, the selected-row ablation, and their limitations;
EG remains fastest in the overlap and cost controls.

\section{Conclusion}\label{sec:structure}
Our randomized lifted extragradient method gives simultaneous oracle
guarantees for nonseparable objectives and nonlinear convex--concave
coupling. Under strong convexity, it converges linearly with expected
partial-gradient bounds that separate objective conditioning, coupling
smoothness, and input overlap. Regularization yields restricted-gap rates $\varepsilon^{-1/2}$
for objective gradients and $\varepsilon^{-1}$ for coupling gradients,
up to logarithmic factors. Bilinear lower bounds match disjoint block
pairs under the stated parameter and span-model restrictions.

Experiments show runtime gains on selected poorly conditioned instances,
while full-oracle and structure-exploiting baselines are faster elsewhere.
These results highlight the role of evaluation costs and problem structure.
Optimal general overlap dependence remains open; further directions include
stochastic oracles, constrained domains, and asynchronous updates.

\section*{Acknowledgements}
The research is supported by the Ministry of Science and Higher
Education of the Russian Federation, project No.~FSMG-2024-0011.
\par
\setlength{\parskip}{5.5pt}
\FloatBarrier
\clearpage
\bibliographystyle{abbrvnat}
\bibliography{references}
\clearpage
\appendix
\section*{Appendices}
\setcounter{algorithm}{0}
\renewcommand{\thealgorithm}{S\arabic{algorithm}}
\renewcommand{\theHalgorithm}{S\arabic{algorithm}}
The supplement retains complete statements, algorithm descriptions, and proofs.
The table below identifies where to find each part of the argument.
All coupling Lipschitz bounds are derived from the original-coordinate
Lipschitz arrays. Appendix~\ref{supp:sec:algorithm} implements the method and verifies the
input metrics, Appendix~\ref{supp:sec:analysis} states and proves the
primary partial-gradient theorem with explicit constants, and
Appendix~\ref{supp:sec:coupling-blocks} gives the full-assembly comparison.

\begin{center}
\begin{tabularx}{\textwidth}{@{}lXlr@{}}
\toprule
Appendix & Contents & Key destination & Page\\
\midrule
\ref{supp:sec:model} & Problem data, normalization, exact oracles, and the span model & Bounds~\ref{supp:sec:coupling-normalization} & \pageref{supp:sec:model}\\
\ref{supp:sec:algorithm} & Gradient coordinates, component metrics, algorithm, and storage & Implementation~\ref{supp:sec:records} & \pageref{supp:sec:algorithm}\\
\ref{supp:sec:analysis} & Descent, stage contraction, explicit oracle bounds, and their proof & Upper bounds~\ref{supp:sec:main} & \pageref{supp:sec:analysis}\\
\ref{supp:sec:structure} & Coupling certificates, examples, and assembly comparison & Comparison~\ref{supp:sec:coupling-blocks} & \pageref{supp:sec:structure}\\
\ref{supp:sec:lower} & Lower statements, matching regimes, and finite constructions & Construction~\ref{supp:app:lower-full} & \pageref{supp:sec:lower}\\
\ref{supp:sec:convex} & Regularization, gap transfer, and convex-case results & Upper bounds~\ref{supp:sec:convex-upper} & \pageref{supp:sec:convex}\\
\ref{supp:app:aux} & Auxiliary geometry and normalized oracle formulas & Oracle formulas~\ref{supp:app:normalization} & \pageref{supp:app:aux}\\
\ref{supp:sec:experiments} & Benchmark models, solver results, and evaluation protocol & Results~\ref{supp:sec:experiment-base-results}--\ref{supp:sec:experiment-sweep-results} & \pageref{supp:sec:experiments}\\
\bottomrule
\end{tabularx}
\end{center}

\section{Problem formulation and information models}\label{supp:sec:model}
\subsection{Spaces and tagged notation}
Use the problem, block partition, and notation of Section~\ref{sec:model},
with $\mathcal X=\RR^{d_x}$, $\mathcal Y=\RR^{d_y}$,
$\mathcal Z=\mathcal X\times\mathcal Y$, and $n=m+q$.
For internal algorithm and proof formulas, define the tagged blocks
\begin{equation}\label{supp:eq:block-tags}
 V_f=\{(f,i):1\le i\le m\},\qquad
 V_g=\{(g,j):1\le j\le q\},\qquad V=V_f\mathbin{\dot\cup}V_g.
\end{equation}
Their coordinates and dimensions are
\[
 (z_{(f,i)},d_{(f,i)})=(x_i,d_{x,i}),\qquad
 (z_{(g,j)},d_{(g,j)})=(y_j,d_{y,j}),
\]
where $d_{x,i}=\dim(x_i)$ and $d_{y,j}=\dim(y_j)$.
For quadratic divergences, write $D_A(v,w)=\norm{v-w}_A^2/2$;
$\norm{\cdot}_\op$ denotes the induced operator norm.

\subsection{Objective and coupling data}\label{supp:sec:problem-data}
We use Assumptions~\ref{ass:players} and~\ref{ass:coupling} from
Section~\ref{sec:model}. In matrix notation, the joint objective models
\eqref{eq:main-model-f}--\eqref{eq:main-model-g} read
\begin{align}
 \tfrac12\norm{x-\widetilde x}_{M_x}^2
 &\le D_f(x,\widetilde x)
 \le\tfrac12\norm{x-\widetilde x}_{L_x}^2,\label{supp:eq:model-f}\\
 \tfrac12\norm{y-\widetilde y}_{M_y}^2
 &\le D_g(y,\widetilde y)
 \le\tfrac12\norm{y-\widetilde y}_{L_y}^2.\label{supp:eq:model-g}
\end{align}
The matrices are those of Section~\ref{sec:model}; the convex case uses
$M_x=M_y=0$ as in Assumption~\ref{ass:convex}.
The coupling arrays satisfy the global first-order
bounds~\eqref{eq:main-coupling-model}.

Define
\begin{equation}\label{supp:eq:TH}
 T_H(z)=\bigl(\nabla_xH(x,y),-\nabla_yH(x,y)\bigr).
\end{equation}

Let $L=\diag(L_x,L_y)$, with $h$ as in Section~\ref{sec:model}. The full saddle operator is
\begin{equation}\label{supp:eq:saddle-equation}
 F(z)=\bigl(\nabla_x\cL(x,y),-\nabla_y\cL(x,y)\bigr)
     =\nabla h(z)+T_H(z).
\end{equation}
Convex--concavity implies that $T_H$ is monotone; a first-order proof is given in Appendix~\ref{supp:app:monotonicity}. A point is a saddle point if and only if $F(z)=0$. Positive joint lower curvature of $f$ and $g$, together with monotonicity of $T_H$, makes $F$ strongly monotone and its zero unique. Existence is included in Assumption~\ref{ass:coupling}; in this global strongly monotone setting it also follows from Appendix~\ref{supp:app:existence}.

\subsection{Normalization and derived coupling bounds}\label{supp:sec:coupling-normalization}
Use the normalized coordinates of Section~\ref{sec:model}, together with
\[
 \widehat z=M^{1/2}z,\qquad
 \widehat h(\widehat z)=h(M^{-1/2}\widehat z),\qquad
 \widehat T_H(\widehat z)=M^{-1/2}T_H(M^{-1/2}\widehat z).
\]
The transformed functions $f$ and $g$ have lower curvature one.
Use the objective condition numbers and $\kappa_{\max}$ from~\eqref{eq:main-kappa}. These ratios have positive denominators and refer only to the declared joint models of $f$ and $g$, not to the curvature of $H$ or of the entire saddle objective.

\paragraph{Derived coupling parameters.}
With positive objective lower curvatures, $M$ from
Section~\ref{sec:model} is positive definite. Form the scalar matrices
\begin{equation}\label{supp:eq:coupling-scalar-matrix}
 \begin{gathered}
 \mathsf L_H=\begin{pmatrix}L^{H,xx}&L^{H,xy}\\L^{H,yx}&L^{H,yy}\end{pmatrix},\qquad
 \mathsf D_\mu=\diag(\mu_1^f,\ldots,\mu_m^f,\mu_1^g,\ldots,\mu_q^g),\\
 \mathsf A_H=\mathsf D_\mu^{-1/2}\mathsf L_H\mathsf D_\mu^{-1/2}.
 \end{gathered}
\end{equation}
Each entry of $L^{H,xx}$, for example, is $L_{it}^{H,xx}$. Unlike $M$, $\mathsf D_\mu$ has one scalar entry per original block. The normalized coupling parameters in~\eqref{eq:main-coupling-row-parameters} are the Euclidean row norms of $\mathsf A_H$. Define the full Lipschitz bound
\begin{equation}\label{supp:eq:coupling-derived-full}
 L_H:=\norm{\mathsf A_H}_\op.
\end{equation}
These are certified upper bounds, not necessarily the smallest Lipschitz moduli. Their normalization uses both output and input curvatures, and the coupling parameters can be below one. The quantity $\kappa_{\max}$ in~\eqref{eq:main-kappa} uses only objective parameters.

Use the primal and dual input supports of Section~\ref{sec:model}.
For $I\subseteq\{1,\ldots,m\}$ and $J\subseteq\{1,\ldots,q\}$,
define the original-variable projection by
\begin{equation}\label{supp:eq:original-input-projection}
 P_{I,J}(x,y)=(\bar x,\bar y),\qquad
 \bar x_t=\mathbf1_{\{t\in I\}}x_t,\quad
 \bar y_u=\mathbf1_{\{u\in J\}}y_u.
\end{equation}
For internal algorithmic operations, combine each support pair into a tagged set:
\begin{equation}\label{supp:eq:coupling-derived-supports}
 \begin{aligned}
 S_{H,i}^x&=\{(f,t):t\in I_{H,i}^x\}
       \cup\{(g,u):u\in J_{H,i}^x\},\\
 S_{H,j}^y&=\{(f,t):t\in I_{H,j}^y\}
       \cup\{(g,u):u\in J_{H,j}^y\}.
 \end{aligned}
\end{equation}
For $S\subseteq V$, let $P_S$ retain the blocks in $S$ and set the
others to zero. This agrees with the original-variable projection above:
\[
 P_{S_{H,i}^x}=P_{I_{H,i}^x,J_{H,i}^x},\qquad
 P_{S_{H,j}^y}=P_{I_{H,j}^y,J_{H,j}^y}.
\]
\begin{lemma}[Normalization of the original coupling bounds]\label{supp:lem:coupling-normalization}
Assumption~\ref{ass:coupling} and the derived parameters above imply
\begin{align}
 \frac{\norm{\nabla_{x_i}H(z)-\nabla_{x_i}H(v)}}{\sqrt{\mu_i^f}}
 &\le \kappa_{H,i}^x\norm{P_{S_{H,i}^x}(z-v)}_M,
                  \label{supp:eq:coupling-block-Lip-x}\\
 \frac{\norm{\nabla_{y_j}H(z)-\nabla_{y_j}H(v)}}{\sqrt{\mu_j^g}}
 &\le \kappa_{H,j}^y\norm{P_{S_{H,j}^y}(z-v)}_M,
                  \label{supp:eq:coupling-block-Lip-y}
\end{align}
and the full-operator bound
\begin{equation}\label{supp:eq:local-Lip}
 \norm{T_H(z)-T_H(v)}_{M^{-1}}
 \le L_H\norm{z-v}_M\qquad(z,v\in\mathcal Z).
\end{equation}
\end{lemma}
\begin{proof}
Collect the weighted input distances into a nonnegative vector $a$:
\[
 a_{(f,t)}=\sqrt{\mu_t^f}\norm{x_t-\widetilde x_t},\qquad
 a_{(g,u)}=\sqrt{\mu_u^g}\norm{y_u-\widetilde y_u}.
\]
For the primal derivative, divide the primal inequality in~\eqref{eq:main-coupling-model}
by $\sqrt{\mu_i^f}$. Its left side is at most row $i$ of
$\mathsf A_H$ applied to $a$. Cauchy--Schwarz on that row's nonzero
entries gives~\eqref{supp:eq:coupling-block-Lip-x}: the row norm is
$\kappa_{H,i}^x$, and the corresponding input norm is
$\norm{P_{S_{H,i}^x}(z-v)}_M$. Row $m+j$ proves the dual bound in
the same way.

For the full operator, the vector of normalized output-block difference
norms is bounded componentwise by $\mathsf A_Ha$. Since this vector is
nonnegative, taking Euclidean norms gives
\[
 \norm{T_H(z)-T_H(v)}_{M^{-1}}
 \le\norm{\mathsf A_Ha}
 \le L_H\norm{a}=L_H\norm{z-v}_M.
\]
The minus sign in the dual part of $T_H$ does not affect these norms.
\end{proof}
For internal tagged formulas, write
$\nu_{(f,t)}=\nu_t^x$ and $\nu_{(g,u)}=\nu_u^y$.
The overlap factors remain those of~\eqref{eq:main-overlap} and its
dual analogue, using the tagged supports above. Empty supports have
overlap one. These supports and overlap factors are certified
from the supplied arrays and may overestimate actual functional
dependence; they count shared declared inputs, independently of block
dimensions.

An all-zero row has an empty support and a constant partial gradient.
The algorithm initializes and samples every coupling row by the same
rules, with positive sampling weights even for zero bounds. No
classification of constant derivatives is required; an empty support
simply gives a zero component input seminorm.

The normalized coupling operator is $L_H$-Lipschitz in the Euclidean norm. Known block-scalar input and output rescalings implement each normalized oracle with exactly one corresponding original call; Appendix~\ref{supp:app:normalization} gives the formulas. These rescalings preserve the lower-model spans.
For bilinear coupling $H(x,y)=\ip{y}{Bx}$, the actual normalized modulus is
\begin{equation}\label{supp:eq:bilinear-gamma}
 \Lip(\widehat T_H)=\norm{M_y^{-1/2}BM_x^{-1/2}}_\op\le L_H.
\end{equation}
The scalar-array Lipschitz bound can be conservative when several blocks interact. On this subclass, a coupling partial-gradient evaluation returns the corresponding block of $B^\top y$ or $Bx$; unknown matrix entries are not supplied. The lower construction uses disjoint primal--dual pairs, with each chain's strength bounded by both corresponding declared cross coefficients. If the declared array consists only of symmetric cross entries on a nonempty matching, with all other entries zero, these strengths equal the corresponding row parameters, and their maximum is $L_H$.

\subsection{Declared-bound problem classes}\label{supp:sec:problem-classes}
\begin{definition}[Prescribed-array and aggregate problem classes]\label{supp:def:strong-class}
For fixed block dimensions, known block-scalar matrices $0\prec M\preceq L$, and a fixed nonnegative scalar coupling array $\mathsf L_H$, let $\mathfrak C_{\rm arr}(M,L,\mathsf L_H)$ contain triples satisfying Assumptions~\ref{ass:players} and~\ref{ass:coupling} with exactly these declared bounds. Thus the row parameters, supports, overlaps, and derived norm $L_H$ are fixed class data. The coefficients are upper bounds; functions need not attain them or depend on every declared input.

For a positive scalar bound $\overline L_H$, the aggregate class $\mathfrak C(M,L,\overline L_H)$ is the union of these prescribed-array classes over $\norm{\mathsf A_H}_\op\le\overline L_H$. Its coupling arrays may vary. The blockwise lower comparison below uses $\mathfrak C_{\rm arr}$; the aggregate single-edge result is a corollary. The decomposition, block structure, and bounds are declared information. Functions are accessible through the oracles below, not through free formulas, unknown matrix entries, or factorizations. Lower bounds permit finite block dimensions to depend on accuracy; the number of blocks and all declared scalar bounds remain fixed.
\end{definition}

\subsection{Exact oracles and complexity criteria}\label{supp:sec:oracle-model}
Use the exact oracles of Table~\ref{tab:main-oracles}, the positive costs
and total work $W_\tau$ in~\eqref{eq:main-work}, and the relative
mean-square target~\eqref{eq:main-target}. All initialization, update,
and correction calls are counted. The upper algorithm is independent
of the costs; the lower hard instance may depend on them.

\subsection{The lower-bound information model}\label{supp:sec:span-model}
\begin{definition}[Randomized blockwise span model]\label{supp:def:span}
Let $\mathscr S_{\iota,k}$ be the span available in original block $\iota$
after $k$ individual oracle calls, initially $\mathscr S_{\iota,0}=\{0\}$.
Every argument block of call $k+1$ lies in its pre-call span. If that
call returns a gradient vector $i_\iota$ to block $\iota$, set
$\mathscr S_{\iota,k+1}=\spanop(\mathscr S_{\iota,k}\cup\{i_\iota\})$;
other spans are unchanged. Oracle returns need not lie in the pre-call spans.
At an almost surely finite stopping time $K$, each output block lies in
$\mathscr S_{\iota,K}$. Query selection, scalar coefficients, and stopping
may depend measurably on the full history and an internal random seed;
counts may have unbounded support.

Arbitrary basis-vector creation, free vector transfer between blocks,
coordinatewise transformations outside the spans, and free matrix entries,
inverse products, or factorizations are prohibited. Scalar computation
does not enlarge the vector spans.
\end{definition}
The upper bounds below use only the stated exact oracles; the span restriction is needed for the lower bounds. The proposed method is also admissible in the span model at $z^0=0$, as shown after Algorithm~\ref{supp:alg:strong}. For a nonzero start, its initial block vectors may be included in the initial memories. All lower bounds use the zero start.

\FloatBarrier

\section{Construction of the method}\label{supp:sec:algorithm}
Work in $M^{1/2}$ coordinates and suppress hats in Appendices~\ref{supp:sec:algorithm}--\ref{supp:sec:analysis}. Only in internal algorithmic and proof formulas, use a tagged index $\iota\in V$ as defined in~\eqref{supp:eq:block-tags}. Its parameter and counter dictionary is
\begin{equation}\label{supp:eq:tag-dictionary}
 (\mu_\iota,L_\iota,\kappa_\iota,N_\iota,\tau_\iota)=
 \begin{cases}
  (\mu_i^f,L_i^f,\kappa_i^f,N_i^f,\tau_i^f),&\iota=(f,i),\\
  (\mu_j^g,L_j^g,\kappa_j^g,N_j^g,\tau_j^g),&\iota=(g,j).
 \end{cases}
\end{equation}
For $a=(a^x,a^y)$,
\begin{equation}\label{supp:eq:tag-gradient}
 \nabla_\iota h(a)=
 \begin{cases}
  \nabla_{x_i}f(a^x),&\iota=(f,i),\\
  \nabla_{y_j}g(a^y),&\iota=(g,j).
 \end{cases}
\end{equation}
Derived objective component quantities retain the same tag. The joint matrix bounds are
\begin{equation}\label{supp:eq:normalized-model}
 \tfrac12\norm{z-v}^2\le D_h(z,v)
 \le\tfrac12\sum_{\iota\in V}\kappa_\iota\norm{z_\iota-v_\iota}^2.
\end{equation}

\subsection{An implicit lifting and its representation}
For $w=(z,u)$, use $\lambda_\iota,\Lambda,D$ from~\eqref{eq:main-D}
and the following text, with the tagged objective indices above. Set
\begin{align}
 \phi(z)&=h(z)-\tfrac14\norm{z}^2,
 &r(z,u)&=\tfrac14\norm{z}^2+\phi^*(u).\label{supp:eq:lifting}
\end{align}
\begin{lemma}[Lifted geometry]\label{supp:lem:geometry}
The function $\phi$ in~\eqref{supp:eq:lifting} is $1/2$-strongly convex and has the joint upper model $\Lambda$. Its conjugate is finite and differentiable everywhere, and is $1$-strongly convex in the metric $\Lambda^{-1}$. Consequently,
\begin{equation}\label{supp:eq:r-strong}
 D_r(v,w)\ge\tfrac12\norm{v-w}_D^2.
\end{equation}
The maps $\nabla\phi$ and $\nabla\phi^*$ are mutual inverses.
\end{lemma}
\begin{proof}
Subtract the quadratic in~\eqref{supp:eq:lifting} from the joint model~\eqref{supp:eq:normalized-model} and conjugate the resulting upper model. Appendix~\ref{supp:app:conjugate} gives the complete argument, including finiteness, differentiability, and the inverse relations.
\end{proof}
The conjugate is an analytical device; it is not a supplied oracle.

Define
\begin{equation}\label{supp:eq:J}
 \cJ(z,u)=(T_H(z)+u,-z),\qquad
 w^\star=(z^\star,\nabla\phi(z^\star)).
\end{equation}
Then $\nabla r(w^\star)+\cJ(w^\star)=0$, and
\begin{equation}\label{supp:eq:J-monotone}
 \ip{\cJ(z,u)-\cJ(v,\widetilde u)}{(z-v,u-\widetilde u)}
 =\ip{T_H(z)-T_H(v)}{z-v}\ge0.
\end{equation}
The cross terms cancel. Thus the lifted operator is monotone even though it need not be linear or skew-symmetric.

The implementation stores the gradient coordinate
\begin{equation}\label{supp:eq:representation}
 \theta=\nabla r(w)=(z/2,a),\qquad
 w(\theta)=(2\theta_z,\nabla\phi(\theta_u)).
\end{equation}
Here $a=\nabla\phi^*(u)$. Assigning a represented vector gives $z=2\theta_z$ and $a=\theta_u$ by known arithmetic. The vector $u=\nabla\phi(a)$ is implicit; no evaluation of $\nabla\phi^*$ or inverse gradient is needed.

Let $U_\iota:\RR^{d_\iota}\to\mathcal Z$ be the canonical embedding of the tagged original-variable block $\iota$, so that $U_\iota^\top z=z_\iota$ and $\sum_{\iota\in V}U_\iota U_\iota^\top=I$. Its dimension and original-variable coordinates are specified after~\eqref{supp:eq:block-tags}. Use the component set
\[
 \begin{gathered}
 \mathcal A_H^x=\{(H,x,i):1\le i\le m\},\qquad
 \mathcal A_H^y=\{(H,y,j):1\le j\le q\},\\
 \cI=V\mathbin{\dot\cup}\mathcal A_H^x\mathbin{\dot\cup}\mathcal A_H^y.
 \end{gathered}
\]
Here $k\in\cI$ labels an objective component or one coupling partial gradient.
The lifted components are
\begin{align}
 \cJ_\iota(z,u)&=(U_\iota u_\iota,-U_\iota z_\iota),\label{supp:eq:Oi}\\
 \cJ_{H,i}^x(z,u)&=(U_{(f,i)}\nabla_{x_i}H(z),0),\label{supp:eq:OH}\\
 \cJ_{H,j}^y(z,u)&=(-U_{(g,j)}\nabla_{y_j}H(z),0).
\end{align}
Their gradient-coordinate evaluations $G_k(\theta)=\cJ_k(w(\theta))$
and stored local coordinates are given in Table~\ref{supp:tab:components}.

\paragraph{One component set for every row.}
All components belong to $\cI$ and use the same stage initialization,
sampling, and cache refresh. The positive floor in the coupling weights
ensures that no row needs to be classified as constant. For an empty
support, $D_k=0$ and the local-memory correction vanishes; a constant
component also has zero sampled evaluation difference.

\subsection{Component input seminorms and sampling parameters}
Use the matrices $D_k$ defining the main text's component input
seminorms and its variation constants $c_k$.
To allow alternative allocations, choose positive shares $t_k$ satisfying
\begin{equation}\label{supp:eq:packing}
 t_\iota
 +\sum_{i:\iota\in S_{H,i}^x}t_{(H,x,i)}
 +\sum_{j:\iota\in S_{H,j}^y}t_{(H,y,j)}
 \le1\qquad(\iota\in V).
\end{equation}
Every component has a positive share; empty coupling supports impose
no packing constraint. Define
\begin{equation}\label{supp:eq:Wi}
 \norm{w}_{D_\iota}^2=t_\iota\bigl(\tfrac12\norm{z_\iota}^2
       +\lambda_\iota^{-1}\norm{u_\iota}^2\bigr)
\end{equation}
and
\begin{equation}\label{supp:eq:WH}
 D_{H,i}^x=t_{(H,x,i)}\diag(\tfrac12P_{S_{H,i}^x},0),\qquad
 D_{H,j}^y=t_{(H,y,j)}\diag(\tfrac12P_{S_{H,j}^y},0).
\end{equation}
The packing constraints are equivalent to the metric allocation
\begin{equation}\label{supp:eq:allocation}
 D_k\succeq0,\qquad\sum_{k\in\cI}D_k\preceq D.
\end{equation}
Set
\begin{equation}\label{supp:eq:b}
 c_\iota=\sqrt{\frac{2\kappa_\iota-1}{t_\iota}},\qquad
 c_{(H,x,i)}=\max\left\{1,\frac{2\kappa_{H,i}^x}{\sqrt{t_{(H,x,i)}}}\right\},\qquad
 c_{(H,y,j)}=\max\left\{1,\frac{2\kappa_{H,j}^y}{\sqrt{t_{(H,y,j)}}}\right\}.
\end{equation}
For every pair of lifted points,
\begin{equation}\label{supp:eq:component-difference}
 \norm{\cJ_k(w)-\cJ_k(v)}_{D^{-1}}^2
 \le c_k^2\norm{w-v}_{D_k}^2.
\end{equation}
For an objective component, both sides equal
$2\norm{u_\iota-\widetilde u_\iota}^2+\lambda_\iota\norm{z_\iota-\widetilde z_\iota}^2$.
For a primal coupling partial, Lemma~\ref{supp:lem:coupling-normalization} gives
\begin{align}\label{supp:eq:global-component-bound}
 \norm{\cJ_{H,i}^x(w)-\cJ_{H,i}^x(v)}_{D^{-1}}^2
 &=2\norm{\nabla_{x_i}H(z)-\nabla_{x_i}H(\widetilde z)}^2\nonumber\\
 &\le2(\kappa_{H,i}^x)^2\norm{P_{S_{H,i}^x}(z-\widetilde z)}^2
 \le c_{(H,x,i)}^2\norm{w-v}_{D_{H,i}^x}^2.
\end{align}
The dual calculation is identical. These are differences of evaluations;
no linearity or individual-component monotonicity is assumed.

The default shares are
\begin{equation}\label{supp:eq:dense-weights}
 t_\iota=\tfrac12,\qquad
 t_{(H,x,i)}=\frac1{2\Delta_{H,i}^x},\qquad
 t_{(H,y,j)}=\frac1{2\Delta_{H,j}^y}.
\end{equation}
Every support containing $\iota$ has $\Delta\ge\nu_\iota$. Thus, at a
block of positive input multiplicity, the coupling share satisfies
\begin{equation}\label{supp:eq:dense-allocation}
 \sum_{i:\iota\in S_{H,i}^x}\frac1{2\Delta_{H,i}^x}
 +\sum_{j:\iota\in S_{H,j}^y}\frac1{2\Delta_{H,j}^y}
 \le\frac{\nu_\iota}{2\nu_\iota}=\frac12.
\end{equation}
It is zero when $\nu_\iota=0$. The objective component uses the other
half of its original-variable metric and half its auxiliary metric,
proving feasibility.

Choose $S,p_k,\eta,K$ by~\samplingparametersreference,
using the variation constants for the selected allocation.
All $c_k$ are positive, including for zero row bounds. Every component
is initialized once per stage, and those calls are counted.

\subsection{Stored coordinates and component formulas}\label{supp:sec:records}
At inner iteration $t$, the current representation is
$\theta_t=(z_t/2,a_t)$. Store a component value $g_k^t$ for each
$k\in\cI$ and the aggregate
$g^t=\sum_{k\in\cI}g_k^t$.
For an objective component $\iota$, store the two local coordinates
$\widetilde z_{\iota,t},\widetilde u_{\iota,t}$.
For a coupling component, store
$\widetilde z^x_{H,i,t}$ or $\widetilde z^y_{H,j,t}$ on its declared
input support. These latter vectors are understood as zero outside
that support; only their supported coordinates need storage.
Stage initialization evaluates every component once and stores these
local values; all initialization calls are counted. Subsequent inner
iterations need not reconstruct the current full vector $u=\nabla\phi(a)$.

Each row of Table~\ref{supp:tab:components} uses one original partial-gradient
call at the full point $a^x$, $a^y$, or $z$, with known normalization scalings.

\begin{table}[H]
\caption{Four normalized component formulas at $\theta=(z/2,a)$: one original oracle call with input and output rescaling.}\label{supp:tab:components}
\centering\small\renewcommand{\arraystretch}{1.5}
\begin{tabularx}{\textwidth}{@{}l l X l@{}}
\toprule
Component $k$ & Partial gradient and counted return & Fresh lifted value $G_k(\theta)$ & Local coordinates for an update\\
\midrule
$\iota=(f,i)$ & $\mathsf q=\nabla_{x_i}f(a^x)$; $N_i^f$ &
$u_\iota=\mathsf q-a_\iota/2$,\newline $(U_\iota u_\iota,-U_\iota z_\iota)$ & $(z_\iota,u_\iota)$\\
$\iota=(g,j)$ & $\mathsf q=\nabla_{y_j}g(a^y)$; $N_j^g$ &
$u_\iota=\mathsf q-a_\iota/2$,\newline $(U_\iota u_\iota,-U_\iota z_\iota)$ & $(z_\iota,u_\iota)$\\
$(H,x,i)$ & $\mathsf q=\nabla_{x_i}H(z)$; $N_{H,i}^x$ &
$(U_{(f,i)}\mathsf q,0)$ & $P_{S_{H,i}^x}z$\\
$(H,y,j)$ & $\mathsf q=\nabla_{y_j}H(z)$; $N_{H,j}^y$ &
$(-U_{(g,j)}\mathsf q,0)$ & $P_{S_{H,j}^y}z$\\
\bottomrule
\end{tabularx}
\end{table}

For a fresh objective gradient return at $\theta_t$, write
$u_{\iota,t}=\mathsf q-a_{\iota,t}/2$. The local-memory term is explicitly
\begin{align}
 d_t&=\bigl(t_\iota\tfrac12U_\iota(z_{\iota,t}-\widetilde z_{\iota,t}),
       t_\iota\lambda_\iota^{-1}U_\iota(u_{\iota,t}-\widetilde u_{\iota,t})\bigr)
       &&\text{if }i_t=\iota,\label{supp:eq:record-momentum-i}\\
 d_t&=\bigl(t_{(H,x,i)}\tfrac12P_{S_{H,i}^x}(z_t-\widetilde z^x_{H,i,t}),0\bigr)
       &&\text{if }i_t=(H,x,i),\label{supp:eq:record-momentum-H}\\
 d_t&=\bigl(t_{(H,y,j)}\tfrac12P_{S_{H,j}^y}(z_t-\widetilde z^y_{H,j,t}),0\bigr)
       &&\text{if }i_t=(H,y,j).\nonumber
\end{align}
These formulas use the fresh return and the old stored coordinates,
with no additional evaluation. In the proof, their conceptual lifted anchor
is $v_{k,t}$. Any completion outside the coordinates used by $D_k$ and
$\cJ_k$ is immaterial: $g_k^t=\cJ_k(v_{k,t})$,
$r_k^t=D_kv_{k,t}$, and
$d_t=D_{i_t}(w_t-v_{i_t,t})$. This notation does not require forming
$w_t=(z_t,u_t)$ in the implementation.

\subsection{One stage and the outer algorithm}
Every stage of Algorithm~\ref{alg:epoch} initializes each sampled component
once and completes all $K$ iterations. Keep the fresh $i_t$ value and
current-point coordinates until its cache update after correction, which
uses the old table. The $j_t$ evaluation at $\bar\theta_t$ never refreshes
the table; discard its local coordinates. Both prescribed evaluations are
charged even when $j_t=i_t$; they occur at the current and predicted
points, respectively, which may coincide. The prediction uses
$d_t=R_{i_t}(\theta_t)-r_{i_t}^t$ from
\eqref{supp:eq:record-momentum-i}--\eqref{supp:eq:record-momentum-H}.

\begin{algorithm}[H]
\caption{Algorithm under strong convexity of $f$ and $g$ in original coordinates}\label{supp:alg:strong}
\begin{algorithmic}[1]
\Require The oracles of Table~\ref{tab:main-oracles}; known $M,L$ and the four original coupling Lipschitz arrays; initial point $z^0$; $0<\varepsilon\le1/2$.
\Ensure An output in original coordinates satisfying Theorem~\ref{supp:thm:strong}.
\State Form normalized oracle wrappers using $\widehat z=M^{1/2}z$.
\State Compute the normalized bounds, supports, and overlaps from Appendix~\ref{supp:sec:model}; form the lifted data and default allocation~\eqref{supp:eq:Wi}--\eqref{supp:eq:dense-weights}.
\State Compute $S,p,\eta,K$ from~\samplingparametersreference.
\State $\widehat z^0\gets M^{1/2}z^0$; $\theta^{(0)}\gets(\widehat z^0/2,\widehat z^0)$.
\State Choose $E$ as in~\eqref{eq:main-epochs}.
\For{$e=0,\ldots,E-1$}
 \State $\theta^{(e+1)}\gets\operatorname{Stage}(\theta^{(e)};p,\eta,K)$, using fresh independent samples.
 \Statex \hspace{\algorithmicindent}\hspace{\algorithmicindent}$\theta^{(e+1)}=\bar\theta_{e,\tau_e}$: retain both returned coordinates; do not reset $a$ to $z$.
\EndFor
\State \Return $z^{\out}=M^{-1/2}(2\theta_z^{(E)})$.
\end{algorithmic}
\end{algorithm}

Any feasible allocation~\eqref{supp:eq:allocation} may replace the
default in Algorithm~\ref{alg:epoch}: use the corresponding $D_k,c_k$,
form $R_k=D_kw$ through the local records above, and recompute
$S,p,\eta,K$. Theorem~\ref{supp:thm:epoch} gives its counts.

At the zero start, the method is admissible in Definition~\ref{supp:def:span}. Each original-variable or auxiliary block is a scalar combination of vectors already available in that same original block and its oracle returns. The embeddings $U_\iota$ only place a block in its own position; they do not transfer information to another block. Coupling returns are the only vectors that can transmit information through $H$. This establishes span admissibility without assuming that the arithmetic updates are sparse.

\subsection{Storage requirements}
Let $d=d_x+d_y$ and let $\dim(S)$ be the total coordinate dimension of
the blocks in $S$. Storing component values in their output blocks and
anchors in their input blocks requires retained memory
\begin{equation}\label{supp:eq:coupling-block-memory}
 O\!\left(d+\sum_{i=1}^m\dim(S_{H,i}^x)
            +\sum_{j=1}^q\dim(S_{H,j}^y)\right).
\end{equation}
Objective records, coupling outputs, the aggregate, and represented vectors
occupy $O(d)$ coordinates; coupling anchors give
the support sums, equal to $\sum_{\iota\in V}\nu_\iota\dim(z_\iota)$.
The uniform output index can be drawn in advance, retaining only its
prediction. Sparse support records and row parameters are included;
supplied dense arrays and peak preprocessing memory are excluded.
Dense dependence gives $O((m+q)d)$ retained memory, and bounded input
overlap gives $O(d)$; no arithmetic-operation bound is asserted.

\clearpage

\section{Convergence analysis}\label{supp:sec:analysis}
\subsection{An abstract lifted system}
\begin{definition}[Admissible lifted data]\label{supp:def:lifted-data}
Let the lifted space be a finite-dimensional Euclidean space. The data consist of a finite differentiable function $r$, a matrix $D\succ0$, a finite family of maps $(\cJ_k)_{k\in\cI}$, fixed matrices $D_k\succeq0$ with $\sum_kD_k\preceq D$, and positive constants $c_k$. They satisfy
\begin{enumerate}
 \item $D_r(v,w)\ge\norm{v-w}_D^2/2$ for all $v,w$;
 \item $\cJ=\sum_k\cJ_k$ is monotone and $\nabla r(w^\star)+\cJ(w^\star)=0$ for some $w^\star$;
 \item $\norm{\cJ_k(v)-\cJ_k(w)}_{D^{-1}}^2\le c_k^2\norm{v-w}_{D_k}^2$ for all $k,v,w$.
\end{enumerate}
The parameters are given by~\samplingparametersreference. The represented updates use $\theta=\nabla r(w)$. Strong convexity makes $\nabla r$ invertible onto the full dual space. For an oracle-complexity interpretation, the component evaluations in gradient coordinates and local-memory corrections must additionally be implementable with the declared, counted information.
\end{definition}

For the abstract system, use the prediction--correction updates of Algorithm~\ref{alg:epoch} with fresh component values $\cJ_k(w(\theta))$, local-memory vectors $D_k(w-v_k)/8$, and the delayed anchor assignment $v_{i_t,t+1}=w_t$; other anchors remain unchanged. For the saddle problem, Table~\ref{supp:tab:components} and~\eqref{supp:eq:record-momentum-i}--\eqref{supp:eq:record-momentum-H} implement these quantities with local records.

The data of Appendix~\ref{supp:sec:algorithm} satisfy this definition by Lemma~\ref{supp:lem:geometry},~\eqref{supp:eq:J-monotone}, and~\eqref{supp:eq:component-difference}. In this construction the implementation condition follows from Table~\ref{supp:tab:components} and the indexed updates in Algorithm~\ref{alg:epoch}. The proof below does not require the individual components to be monotone.

\subsection{A localized memory inequality}
Let $\mathscr F_t$ denote the history before the two samples in iteration $t$, including the current representation and table. Write $\E_t[\cdot]=\E[\cdot\mid\mathscr F_t]$. Within an iteration, $\E_{i}$ and $\E_{j}$ indicate averaging over the first and second fresh samples, with all previously specified information fixed. Conditional on $i$, the prediction is fixed and $j$ still has law $p$.

For $W\succeq0$, abbreviate $D_W(v,w)=\norm{v-w}_W^2/2$. Let $w_t=w(\theta_t)$, $\bar w_t=w(\bar\theta_t)$, and
\begin{equation}\label{supp:eq:potential}
 \Phi_t=D_r(w^\star,w_t)+\tfrac18\sum_kD_{D_k}(w^\star,v_{k,t}).
\end{equation}
The parameter identities are
\begin{equation}\label{supp:eq:identities}
 \frac{\eta^2c_k^2}{p_k}=\frac{p_k}{64},\qquad
 \sum_kD_k\preceq D.
\end{equation}
At a coincident initialization,
\begin{equation}\label{supp:eq:potential-init}
 \Phi_0\le\frac98D_r(w^\star,w_0).
\end{equation}
This is the point of local quadratic memories: the initial error is not duplicated once per component.

\begin{lemma}[Two-sample descent for nonlinear monotone sums]\label{supp:lem:descent}
For the admissible lifted data of Definition~\ref{supp:def:lifted-data}, condition on $\mathscr F_t$ in Algorithm~\ref{alg:epoch}. Then
\begin{align}\label{supp:eq:descent}
 \E_t\Phi_{t+1}
 &+\eta\E_t\bigl[D_r(w^\star,\bar w_t)+D_r(\bar w_t,w^\star)\bigr]
 +\tfrac12\E_tD_r(\bar w_t,w_t)\nonumber\\
 &+\tfrac1{32}\sum_kp_kD_{D_k}(w_t,v_{k,t})
 \le\Phi_t.
\end{align}
Only monotonicity of $\cJ$ is needed; the individual components need not be monotone or linear.
\end{lemma}
\begin{proof}
The argument has five steps. We derive an exact three-point identity,
control the correction energy, bound the local-memory correction, use monotonicity,
and cancel the memory transfers. Suppress $t$ throughout and write
$g=\sum_k\cJ_k(v_k)$ for the cached sum $g^t$ of Algorithm~\ref{alg:epoch}.
Write
\[
 \vartheta=\nabla r(w)-\tfrac18D_{i}(w-v_{i}),\qquad
 \widehat g=g+\frac{\cJ_{j}(\bar w)-\cJ_{j}(v_{j})}{p_{j}}.
\]

\textbf{Step 1: an exact update identity.}
For the first sample $i$, define
\begin{equation}\label{supp:eq:A-function}
 \mathcal A_{i}(v)=D_r(v,w)-\tfrac18D_{D_{i}}(v,w)
                    +\tfrac18D_{D_{i}}(v,v_{i}).
\end{equation}
The quadratic parts in $v$ cancel, giving
$\nabla\mathcal A_{i}(v)=\nabla r(v)-\vartheta$.
The represented updates imply
\begin{equation}\label{supp:eq:update-identities}
 (1+\eta)\nabla r(\bar w)=\vartheta-\eta g,\qquad
 \nabla r(w^+)=\vartheta-\eta(\nabla r(\bar w)+\widehat g).
\end{equation}
Expanding the divergences, or using the three-point identity, gives the exact equality
\begin{align}\label{supp:eq:threepoint}
 D_r(w^\star,w^+)
 &=\mathcal A_{i}(w^\star)-\mathcal A_{i}(\bar w)
 -\eta\ip{\nabla r(\bar w)+\widehat g}{\bar w-w^\star}
 +D_r(\bar w,w^+).
\end{align}

\textbf{Step 2: correction energy.}
Conditional on $i$, the prediction is fixed and $j$ is independent. Hence
\begin{equation}\label{supp:eq:unbiased}
 \E_{j}[\widehat g\mid i]=\cJ(\bar w),\qquad
 \E_{j}\norm{\widehat g-g}_{D^{-1}}^2
 =\sum_k\frac{\norm{\cJ_k(\bar w)-\cJ_k(v_k)}_{D^{-1}}^2}{p_k}.
\end{equation}
Strong convexity of $r$ implies that $r^*$ has a $D^{-1}$ upper quadratic model. Also~\eqref{supp:eq:update-identities} gives
$\nabla r(\bar w)-\nabla r(w^+)=\eta(\widehat g-g)$.
Conjugate divergence duality therefore yields
\begin{align}\label{supp:eq:correction-bound}
 \E_{j}D_r(\bar w,w^+)
 &\le\frac{\eta^2}{2}\E_{j}\norm{\widehat g-g}_{D^{-1}}^2
 \le\eta^2\sum_k\frac{c_k^2}{p_k}D_{D_k}(\bar w,v_k)\nonumber\\
 &=\frac1{64}\sum_kp_kD_{D_k}(\bar w,v_k).
\end{align}
No linearity is used in~\eqref{supp:eq:unbiased} or~\eqref{supp:eq:correction-bound}.

The quadratic triangle inequality and $D_k\preceq D$ imply
\begin{equation}\label{supp:eq:compare-current}
 \E_{i,j}D_r(\bar w,w^+)
 \le\frac1{32}\sum_kp_kD_{D_k}(w,v_k)
       +\frac1{32}\E_{i}D_r(\bar w,w).
\end{equation}
This comparison at the old current point avoids asserting independence between $i$ and $\bar w$.

\textbf{Step 3: subtract the local-memory correction.}
For any $W\succeq0$, expansion and Young's inequality give
\begin{align}\label{supp:eq:quadratic-lower}
 D_W(\bar w,v)-D_W(\bar w,w)
 &=D_W(w,v)+\ip{W(w-v)}{\bar w-w}\nonumber\\
 &\ge\tfrac12D_W(w,v)-2D_W(\bar w,w).
\end{align}
Applying this to~\eqref{supp:eq:A-function} gives
\begin{equation}\label{supp:eq:momentum-bound}
 \E_{i}\mathcal A_{i}(\bar w)
 \ge\tfrac34\E_{i}D_r(\bar w,w)
     +\frac1{16}\sum_kp_kD_{D_k}(w,v_k).
\end{equation}
Subtracting~\eqref{supp:eq:compare-current} leaves a movement coefficient
$23/32\ge1/2$ and a mismatch coefficient $1/32$.

\textbf{Step 4: monotone drift.}
Because $\nabla r(w^\star)+\cJ(w^\star)=0$,
\begin{align}\label{supp:eq:monotone-drift}
 \ip{\nabla r(\bar w)+\cJ(\bar w)}{\bar w-w^\star}
 &=D_r(w^\star,\bar w)+D_r(\bar w,w^\star)\nonumber\\
 &\quad+\ip{\cJ(\bar w)-\cJ(w^\star)}{\bar w-w^\star}\nonumber\\
 &\ge D_r(w^\star,\bar w)+D_r(\bar w,w^\star).
\end{align}
For a linear skew operator the last pairing vanishes. For a nonlinear monotone operator it is nonnegative, which is sufficient and favorable in~\eqref{supp:eq:threepoint}.

\textbf{Step 5: memory cancellation.}
Finally,
\begin{align}\label{supp:eq:A-star}
 \E_{i}\mathcal A_{i}(w^\star)
 &=D_r(w^\star,w)-\tfrac18\sum_kp_kD_{D_k}(w^\star,w)
                 +\tfrac18\sum_kp_kD_{D_k}(w^\star,v_k),\\
 \E_{i}\tfrac18\sum_kD_{D_k}(w^\star,v_k^+)
 &=\tfrac18\sum_kD_{D_k}(w^\star,v_k)\nonumber\\
 &\quad+\tfrac18\sum_kp_k
  [D_{D_k}(w^\star,w)-D_{D_k}(w^\star,v_k)].\label{supp:eq:memory-cancel}
\end{align}
The update at the old current point makes the transferred terms cancel exactly. Taking expectations in~\eqref{supp:eq:threepoint}, applying the preceding inequalities, and adding~\eqref{supp:eq:memory-cancel} proves~\eqref{supp:eq:descent}.
\end{proof}

\subsection{Stage contraction and exact accounting}
\begin{theorem}[One stage for admissible lifted data]\label{supp:thm:epoch}
Under Definition~\ref{supp:def:lifted-data}, initialize every anchor at the input representation in Algorithm~\ref{alg:epoch}. Then
\begin{equation}\label{supp:eq:epoch-contract}
 \E\!\left[D_r(w^\star,w(\theta^{\out}))\mid\theta^{\inpt}\right]
 \le\tfrac12D_r(w^\star,w(\theta^{\inpt})).
\end{equation}
For $E$ stages with independent fresh samples, let $N_k$ count all evaluations of component $k$, including the initial component evaluations. Then
\begin{equation}\label{supp:eq:count-general}
 \E N_k=E(1+2p_kK)\le(3+36c_k)E.
\end{equation}
The contraction and all count identities hold for the same run.
\end{theorem}
\begin{proof}
Summing Lemma~\ref{supp:lem:descent} over the deterministic $K$ iterations and using~\eqref{supp:eq:potential-init} gives
\[
 \eta\sum_{t=0}^{K-1}\E D_r(w^\star,\bar w_t)
 \le\frac98D_r(w^\star,w_0).
\]
The output index is uniform and does not affect the update sequence. Hence its expected divergence is at most $9/(8\eta K)$ times the input divergence. Since $\eta K\ge18/8$, this factor is at most $1/2$.

Each stage initializes every component once. Each of the two samples in an iteration has marginal law $p$ and causes one component evaluation, so its expected contribution to component $k$ is $2p_k$. Therefore $\E N_k=E(1+2p_kK)$. Finally,
\[
 1+2p_kK\le1+36c_k+2p_k\le3+36c_k,
\]
because $K\le18S+1$ and $p_k\le1$.
\end{proof}

\subsection{Explicit oracle complexity bounds}\label{supp:sec:main}\label{supp:sec:upper-proof}
\begin{theorem}[Simultaneous partial-gradient bounds under strong convexity of $f$ and $g$]\label{supp:thm:strong}
Under Assumptions~\ref{ass:players}--\ref{ass:coupling}, with strong convexity of $f$ and $g$, Algorithm~\ref{supp:alg:strong} with the allocation in~\eqref{supp:eq:dense-weights} and one coupling component per partial gradient attains~\eqref{eq:main-target}
for $0<\varepsilon\le1/2$, using
\begin{equation}\label{supp:eq:strong-epoch-budget}
 \left\lceil\log_2\frac{2\kappa_{\max}}{\varepsilon}\right\rceil
 \quad\text{stages}.
\end{equation}
In the same run, for all $1\le i\le m$
and $1\le j\le q$,
\begin{align}
 \E N_i^f
 &\le75\sqrt{\kappa_i^f}
       \left\lceil\log_2\frac{2\kappa_{\max}}{\varepsilon}\right\rceil,
                  \label{supp:eq:strong-gradient}\\
 \E N_j^g
 &\le75\sqrt{\kappa_j^g}
       \left\lceil\log_2\frac{2\kappa_{\max}}{\varepsilon}\right\rceil,
                  \label{supp:eq:strong-gradient-g}\\
 \E N_{H,i}^x
 &\le\left(39+72\sqrt2\,\kappa_{H,i}^x\sqrt{\Delta_{H,i}^x}\right)
       \left\lceil\log_2\frac{2\kappa_{\max}}{\varepsilon}\right\rceil,
                  \label{supp:eq:coupling-block-strong-x}\\
 \E N_{H,j}^y
 &\le\left(39+72\sqrt2\,\kappa_{H,j}^y\sqrt{\Delta_{H,j}^y}\right)
       \left\lceil\log_2\frac{2\kappa_{\max}}{\varepsilon}\right\rceil.
                  \label{supp:eq:coupling-block-strong-y}
\end{align}
For declared zero rows, $\Delta=1$ and the coupling product is zero;
the same algorithm and bounds apply. All initialization, update,
and correction calls are included. All expected-count inequalities
concern the same output and run; they do not assert pathwise bounds
on the realized counters.
The objective gradient bounds acquire no coupling-overlap factor.
The coupling-block counters need not be equal, since each derivative
has its own sampling probability.
\end{theorem}

\begin{proof}[Proof of Theorem~\ref{supp:thm:strong}]
For any representation, conjugate duality gives
\begin{equation}\label{supp:eq:physical-divergence}
 D_r(w^\star,w(\theta))
 =\tfrac14\norm{z-z^\star}^2+D_\phi(a,z^\star).
\end{equation}
Retaining both $z$ and $a$ between stages preserves this potential;
resetting $a=z$ changes it. At $a=z=z^0$, it is $D_h(z^0,z^\star)$, at most $\kappa_{\max}\norm{z^0-z^\star}^2/2$. Iterating the conditional stage contraction therefore yields
\[
 \E\norm{z^{\out}-z^\star}^2
 \le2\kappa_{\max}2^{-\bigl\lceil\log_2(2\kappa_{\max}/\varepsilon)\bigr\rceil}\norm{z^0-z^\star}^2
 \le\varepsilon\norm{z^0-z^\star}^2.
\]
Undoing the block scaling proves the distance target.

Apply~\eqref{supp:eq:count-general} with the stage budget~\eqref{supp:eq:strong-epoch-budget}.
For an objective component, $c_\iota=\sqrt{4\kappa_\iota-2}\le2\sqrt{\kappa_\iota}$ and $3+72\sqrt{\kappa_\iota}\le75\sqrt{\kappa_\iota}$. For a coupling partial, $c_{(H,x,i)}=\max\{1,2\sqrt2\kappa_{H,i}^x\sqrt{\Delta_{H,i}^x}\}$, with the analogous dual expression. Using $\max\{1,b\}\le1+b$ gives the displayed coefficient $39+72\sqrt2\kappa_{H,i}^x\sqrt{\Delta_{H,i}^x}$. Each component call evaluates exactly its corresponding derivative once, including for constant rows. This proves Theorem~\ref{supp:thm:strong}.

\end{proof}

For the saddle implementation, the component counts satisfy, pathwise,
\[
 N_i^f=N_{(f,i)}\quad(1\le i\le m),\qquad
 N_j^g=N_{(g,j)}\quad(1\le j\le q),\qquad
 N_{H,i}^x=N_{(H,x,i)},\qquad
 N_{H,j}^y=N_{(H,y,j)}.
\]
These identities apply to every row. Each component evaluation has
one specified original partial-gradient cost.

For any other feasible allocation, put
$E=\lceil\log_2(2\kappa_{\max}/\varepsilon)\rceil$. Then
\begin{equation}\label{supp:eq:work-upper}
 \E W_\tau\le E\sum_{k\in\cI}\tau_k(3+36c_k).
\end{equation}
Here $\tau_{(H,x,i)}=\tau_{H,i}^x$ and
$\tau_{(H,y,j)}=\tau_{H,j}^y$. The sum includes all objective and coupling components, with all
stage initializations and sampled calls charged.
Reducing the lifted divergence by a factor $\varepsilon$ needs only
$\lceil\log_2(1/\varepsilon)\rceil$ stages; converting to relative
original-variable distance introduces the $\kappa_{\max}$ factor above.

\paragraph{Weighted work.}
Multiplying the simultaneous count bounds by any fixed positive costs and adding gives
\begin{equation}\label{supp:eq:work-main}
 \begin{aligned}
 \E W_\tau\le{}&\left\lceil\log_2\frac{2\kappa_{\max}}{\varepsilon}\right\rceil
 \biggl[75\sum_i\tau_i^f\sqrt{\kappa_i^f}+75\sum_j\tau_j^g\sqrt{\kappa_j^g}\\
 &\quad{}+\sum_i\tau_{H,i}^x(39+72\sqrt2 \kappa_{H,i}^x\sqrt{\Delta_{H,i}^x})\\
 &\quad{}+\sum_j\tau_{H,j}^y(39+72\sqrt2 \kappa_{H,j}^y\sqrt{\Delta_{H,j}^y})\biggr].
 \end{aligned}
\end{equation}

\section{Coupling examples and assembly}\label{supp:sec:structure}

\subsection{Bounds from original coupling data}
If $H$ is twice continuously differentiable, uniform Hessian-block bounds
\[
 \begin{gathered}
 \norm{\nabla_{x_t}(\nabla_{x_i}H)}_\op\le L_{it}^{H,xx},\qquad
 \norm{\nabla_{y_j}(\nabla_{x_i}H)}_\op\le L_{ij}^{H,xy},\\
 \norm{\nabla_{x_i}(\nabla_{y_j}H)}_\op\le L_{ji}^{H,yx},\qquad
 \norm{\nabla_{y_u}(\nabla_{y_j}H)}_\op\le L_{ju}^{H,yy}
 \end{gathered}
\]
certify~\eqref{eq:main-coupling-model} by integration along successive
input blocks. Mixed-block bounds may be chosen with
$L_{ji}^{H,yx}=L_{ij}^{H,xy}$. Lemma~\ref{supp:lem:coupling-normalization}
then gives the normalized row and full-operator bounds.

For the general bilinear coupling
$H(x,y)=\sum_{j,i}\ip{y_j}{B_{ji}x_i}$, supplied scalar bounds
$\norm{B_{ji}}_\op\le C_{ji}$ certify
$L_{it}^{H,xx}=L_{ju}^{H,yy}=0$ and
$L_{ij}^{H,xy}=L_{ji}^{H,yx}=C_{ji}$. Consequently,
\[
 \kappa_{H,i}^x=\left[\sum_j\frac{C_{ji}^2}{\mu_i^f\mu_j^g}\right]^{1/2},\qquad
 \kappa_{H,j}^y=\left[\sum_i\frac{C_{ji}^2}{\mu_j^g\mu_i^f}\right]^{1/2}.
\]
For a single edge both reduce to $C_{ji}/\sqrt{\mu_i^f\mu_j^g}$,
recovering the bilinear curvature normalization. These computations
use only declared scalar parameters, not unknown matrix entries.

\subsection{A dense nonlinear interaction}
\begin{example}[A dense nonlinear coupling]\label{supp:ex:dense}
Let $a,b$ be unit vectors with a nonzero component in every declared
block, and $0<\beta\le1$. Consider
\begin{equation}\label{supp:eq:dense-nonlinear-coupling}
 H(x,y)=\tfrac12\norm{x}^2-\tfrac12\norm{y}^2
       +\beta\sin(a^\top x)\sin(b^\top y).
\end{equation}
Writing $s=\sin(a^\top x)\sin(b^\top y)$, its Hessians with respect
to $x$ and $y$ separately are $I-\beta s aa^\top\succeq0$ and
$-I-\beta s bb^\top\preceq0$. Thus it is globally smooth and
convex--concave, with a nonconstant mixed Hessian. For example,
$\nabla_{x_i}H=x_i+\beta a_i\cos(a^\top x)\sin(b^\top y)$
depends on every primal and dual input. All $m+q$ coupling rows
therefore use all inputs: every multiplicity and row overlap is $m+q$.

Set $\alpha_i=\norm{a_i}$ and $\eta_j=\norm{b_j}$.
Valid original-coordinate Lipschitz arrays are
\[
 \begin{aligned}
 L_{it}^{H,xx}&=\mathbf1_{\{i=t\}}+\beta\alpha_i\alpha_t,&
 L_{ij}^{H,xy}=L_{ji}^{H,yx}&=\beta\alpha_i\eta_j,\\
 L_{ju}^{H,yy}&=\mathbf1_{\{j=u\}}+\beta\eta_j\eta_u.
 \end{aligned}
\]
Then $\mathsf L_H=I+\beta cc^\top$ for $c=(\alpha,\eta)$, with $\norm{c}^2=2$. In the identity metric,~\eqref{supp:eq:coupling-derived-full} gives $L_H=1+2\beta$. This is a globally smooth convex--concave instance covered by the main theorem, with no restriction on the number of blocks.

It is not generally a sum of pairwise block interactions plus functions of $x$ alone or $y$ alone. For example, with scalar blocks and distinct primal indices $i,t$ and a dual index $j$, the mixed third derivative is
\[
 \partial_{x_i}\partial_{x_t}\partial_{y_j}H
 =-\beta a_ia_tb_j\sin(a^\top x)\cos(b^\top y),
\]
which is not identically zero. Such a derivative vanishes for a sum of functions each depending on just one primal and one dual block, even after adding arbitrary functions of $x$ alone or $y$ alone.

\end{example}

\subsection{Partial sampling and full assembly}\label{supp:sec:coupling-blocks}
We compare partial sampling from Theorem~\ref{supp:thm:strong} with
full assembly using the same separately charged coupling partials.
Both retain the same objective gradient guarantees.

\subsubsection{Full-assembly guarantee}\label{supp:sec:assembly-guarantee}
The global normalized coupling bound is
$L_H=\norm{\mathsf A_H}_\op$, as defined
in~\eqref{supp:eq:coupling-derived-full}.
For every $L_H\ge0$, retain the objective components with
$t_\iota=1/2$ and set
\[
 \cJ_H(w)=(T_H(z),0),\qquad D_H=\diag(I/4,0),\qquad c_H=\max\{1,2\sqrt2L_H\}.
\]
The full Lipschitz bound gives
\[
 \norm{\cJ_H(w)-\cJ_H(v)}_{D^{-1}}^2
 \le2L_H^2\norm{z-\widetilde z}^2
 \le c_H^2\norm{w-v}_{D_H}^2.
\]
The objective metrics use half of $D$; adding $D_H$ preserves
$\sum_kD_k\preceq D$. The sum remains the monotone lifted operator
$\cJ$, so Definition~\ref{supp:def:lifted-data} holds.
A component evaluation assembles all $m+q$ separately charged partials
at the same point and records the full $z$; its local-memory vector is
$((z-z^{\rm anc})/4,0)$. No additional oracle is supplied.
Theorem~\ref{supp:thm:epoch} and the distance conversion in
\eqref{supp:eq:physical-divergence} therefore give
\begin{equation}\label{supp:eq:coupling-block-assembly}
 \E N_{H,i}^x,\ \E N_{H,j}^y
 \le(39+72\sqrt2L_H)
       \left\lceil\log_2\frac{2\kappa_{\max}}{\varepsilon}\right\rceil.
\end{equation}
All coupling counters equal the assembly count pathwise, so one
assembly costs $m+q$ partial calls. The two implementations give
alternative guarantees; selecting the smaller certificate for each
component does not establish a run attaining all those minima.

\subsubsection{Choosing by certified weighted work}\label{supp:sec:weighted-comparison}
Let $E=\lceil\log_2(2\kappa_{\max}/\varepsilon)\rceil$ and define
\[
 \begin{aligned}
 C_\tau&=\sum_i\tau_{H,i}^x+\sum_j\tau_{H,j}^y,\\
 R_\tau&=\sum_i\tau_{H,i}^x\kappa_{H,i}^x\sqrt{\Delta_{H,i}^x}
          +\sum_j\tau_{H,j}^y\kappa_{H,j}^y\sqrt{\Delta_{H,j}^y}.
 \end{aligned}
\]
Empty supports have overlap one, so zero-row products vanish.
The displayed numerical coupling-work certificates are
$E(39C_\tau+72\sqrt2R_\tau)$ for partial sampling and
$E(39C_\tau+72\sqrt2L_HC_\tau)$ for assembly. Thus sampling has the
smaller or equal certificate exactly when
\begin{equation}\label{supp:eq:weighted-assembly-criterion}
 \sum_i\tau_{H,i}^x\kappa_{H,i}^x\sqrt{\Delta_{H,i}^x}
 +\sum_j\tau_{H,j}^y\kappa_{H,j}^y\sqrt{\Delta_{H,j}^y}
 \le L_H\left(\sum_i\tau_{H,i}^x+\sum_j\tau_{H,j}^y\right).
\end{equation}
The objective certificates agree and cancel in this comparison.
At $L_H=0$ both displayed certificates agree; neither implementation
requires a special constant-row branch. This criterion
compares these upper bounds, not actual counts, runtimes, or an optimal
policy. Known costs can select one complete implementation before the run.

\begin{example}[An unequal-cost comparison]\label{supp:ex:weighted-cost}
For $M=I$ and $H=100x_1y_1+x_2y_2$, the row bounds are
$(100,1,100,1)$, $\Delta=1$, and $L_H=100$. Give each strong-edge
partial unit cost and each weak-edge partial cost $100$. Then
$C_\tau=202$ and $R_\tau=400$, so the coupling-work certificates are
about $48{,}607E$ for sampling and $2{,}064{,}710E$ for assembly,
with identical objective bounds. These are upper bounds on charged
work, not runtime measurements.
\end{example}

\subsubsection{A dense nonlinear reversal}\label{supp:sec:dense-sparse-examples}
Specialize Example~\ref{supp:ex:dense} to $\beta=1$, identity curvature,
$m=q=p$ scalar blocks, and $|a_i|=|b_j|=1/\sqrt p$. Write $n=2p$.
The declared certificate is $\mathsf L_H=I+cc^\top$ with
$c_k^2=2/n$ and $\norm c^2=2$, so $L_H=3$.
Every row has full support, hence $\Delta=n$, and its squared norm is
\[
 1+2c_k^2+c_k^2\norm c^2=1+8/n.
\]
Consequently, every sampled coupling coefficient is
$\kappa_{H,k}\sqrt\Delta=\sqrt{n+8}$, where $k$ indexes the $n$ primal and
dual rows. It exceeds the assembled coefficient $L_H=3$ whenever $n>1$.
For any positive coupling costs, assembly therefore has the smaller
displayed coupling-work certificate. In particular, $p\ge2$ retains the
dense nonlinear interaction beyond pairwise block sums.

\section{The lower-bound construction}\label{supp:sec:lower}
\subsection{Strongly convex lower statements and comparison}
Fix the original-coordinate coupling arrays and let $\mathcal E$ be a matching of primal--dual pairs: each primal and dual block belongs to at most one pair. For $e=(i,j)$, write $c_e=c_{ij}$ for the admissible normalized chain strength in~\eqref{eq:main-edge-strength}. In scalar-array notation, it is the minimum of $(\mathsf A_H)_{i,m+j}$ and $(\mathsf A_H)_{m+j,i}$.
Using their minimum respects both cross-derivative bounds, including asymmetric declarations. The matching may be empty.

\begin{theorem}[One blockwise hard instance for prescribed arrays and costs]\label{supp:thm:strong-lower}
Fix $m,q\ge1$, a prescribed-array class $\mathfrak C_{\rm arr}(M,L,\mathsf L_H)$ of Definition~\ref{supp:def:strong-class}, a matching $\mathcal E$, positive costs in~\eqref{eq:main-work}, $R>0$, and $0<\varepsilon\le1/128$. Suppose
\begin{equation}\label{supp:eq:lower-regime}
 \kappa_i^f\ge4\ (1\le i\le m),\qquad
 \kappa_j^g\ge4\ (1\le j\le q),\qquad c_e\ge2\ (e\in\mathcal E).
\end{equation}
In sufficiently large finite block dimensions there exists one quadratic bilinear instance in this prescribed-array class, with $z^0=0$ and $\norm{z^\star}_M^2=R^2$, on which every algorithm of Definition~\ref{supp:def:span} attaining~\eqref{eq:main-target} obeys
\begin{equation}\label{supp:eq:lower-work-strong}
 \E W_\tau\ge\frac18
 \left[\sum_{i=1}^m\tau_i^f\sqrt{\kappa_i^f}
       +\sum_{j=1}^q\tau_j^g\sqrt{\kappa_j^g}
 +\sum_{(i,j)\in\mathcal E}(\tau_{H,i}^x+\tau_{H,j}^y)c_{ij}\right]
 \log\frac1{32\varepsilon}.
\end{equation}
The instance is independent of the realized random seed but may depend on the prescribed parameters, costs, radius, and accuracy. The individual statements have separate quantifiers: for each fixed $i\in\{1,\ldots,m\}$ there is an instance requiring the first bound below, and for each fixed $j\in\{1,\ldots,q\}$ there is an instance requiring the second:
\begin{align}
 \E N_i^f&\ge\frac{\sqrt{\kappa_i^f}}8\log\frac1{32\varepsilon},\label{supp:eq:lower-individual-strong}\\
 \E N_j^g&\ge\frac{\sqrt{\kappa_j^g}}8\log\frac1{32\varepsilon}.\label{supp:eq:lower-individual-strong-g}
\end{align}
These individual instances need not coincide. For each $e=(i,j)\in\mathcal E$, there is also an instance in the same prescribed-array class, with only this coupling pair active, on which both
\begin{equation}\label{supp:eq:lower-H-strong}
 \E N_{H,i}^x\ge\frac{c_{ij}}8\log\frac1{32\varepsilon},\qquad
 \E N_{H,j}^y\ge\frac{c_{ij}}8\log\frac1{32\varepsilon}
\end{equation}
hold simultaneously. Instances for different edges need not coincide. All declared coefficients are retained, including those unused by the hard function.
\end{theorem}
For the rate parameters $\beta_\ell$ of Lemmas~\ref{supp:lem:smoothchain}
and~\ref{supp:lem:directional}, sufficient private chain dimensions are
$d_\ell\ge\max\{2,\lceil\beta_\ell\log(1/\varepsilon)\rceil\}$,
as in~\eqref{supp:eq:dimensions}. Each original block needs its
objective-chain dimension; each endpoint of $e\in\mathcal E$ additionally needs $d_e$
coordinates for that coupling chain. These dimensions are fixed before
any oracle calls.
\begin{proof}
The complete construction and cost-weighted reduction are given in Section~\ref{supp:app:lower-full}.
\end{proof}

\begin{corollary}[Aggregate single-edge lower bound]\label{supp:cor:aggregate-lower}
Fix an aggregate bound $\overline L_H\ge2$ and a selected pair $(i_0,j_0)$, together with the objective, cost, radius, and accuracy assumptions of Theorem~\ref{supp:thm:strong-lower}. In $\mathfrak C(M,L,\overline L_H)$, one common instance forces
\[
 \E W_\tau\ge\frac18\left[
 \sum_i\tau_i^f\sqrt{\kappa_i^f}+\sum_j\tau_j^g\sqrt{\kappa_j^g}
 +(\tau_{H,i_0}^x+\tau_{H,j_0}^y)\overline L_H\right]
 \log\frac1{32\varepsilon}.
\]
The individual objective lower bounds and simultaneous endpoint lower bounds hold with $c_{i_0j_0}=\overline L_H$ on their respective instances.
\end{corollary}
\begin{proof}
Declare only the two normalized cross coefficients on $(i_0,j_0)$ nonzero, both equal to $\overline L_H$, and take $\mathcal E=\{(i_0,j_0)\}$. The derived norm is $\overline L_H$, so this prescribed-array class is contained in the aggregate class. Apply Theorem~\ref{supp:thm:strong-lower}.
\end{proof}

\begin{corollary}[Matching objective and blockwise coupling orders]\label{supp:cor:optimality}
Under~\eqref{supp:eq:lower-regime}, suppose also
\begin{equation}\label{supp:eq:matching-regime}
 \varepsilon\le\min\{1/128,1/\kappa_{\max}\}.
\end{equation}
Suppose the prescribed normalized array has only the symmetric cross entries
$(\mathsf A_H)_{i,m+j}=(\mathsf A_H)_{m+j,i}=c_{ij}$ for $(i,j)\in\mathcal E$, with all other entries zero. On this fixed-array class, the objective gradient orders and every active coupling-row order of Theorem~\ref{supp:thm:strong} match their respective lower benchmarks. Its weighted work upper bound also matches~\eqref{supp:eq:lower-work-strong}, up to universal constants and the $O(E)$ call cost per unmatched row. No lower bound for general input-overlap factors is asserted.
\end{corollary}
\begin{proof}
In~\eqref{supp:eq:matching-regime}, $\left\lceil\log_2\frac{2\kappa_{\max}}{\varepsilon}\right\rceil=O(\log(1/\varepsilon))$, and $\log(1/(32\varepsilon))$ is bounded below by a universal positive multiple of $\log(1/\varepsilon)$. On every matched edge, $\kappa_{H,i}^x=\kappa_{H,j}^y=c_{ij}$ and both overlaps are one; $L_H=\max_{e\in\mathcal E}c_e$ when the matching is nonempty. The conditions $\kappa_\iota\ge4$ and $c_e\ge2$ absorb the additive constants for objectives and active coupling rows. Compare Theorems~\ref{supp:thm:strong} and~\ref{supp:thm:strong-lower}, together with~\eqref{supp:eq:work-main}. Unmatched rows contribute $O(E)$ calls each under the uniform rules. With all blocks matched, there is no unmatched-row term.
\end{proof}
The matching comparison permits heterogeneous strengths, for example
$c_{11}=100$ and $c_{22}=2$: the second pair's lower coefficient is $2$
even though $L_H=100$. For general prescribed arrays,
$c_{ij}\le\kappa_{H,i}^x,\kappa_{H,j}^y$ can be strict, so the lower
certificate need not match every row--overlap product.

\FloatBarrier

The bounds concern worst-case declared-bound classes under
Definition~\ref{supp:def:span}. They do not assert hardness for every
interaction or for unrestricted black-box algorithms. The construction
below isolates private chains and allocates their radii by costs and
chain rates to prove the common-instance work bound directly. Its finite
tail subtraction covers almost surely finite stopping times, including
counts beyond the chosen chain dimensions.

\subsection{Complete finite-dimensional lower-bound constructions}\label{supp:app:lower-full}
The symbol $\varrho$ denotes a geometric chain ratio in this construction; the number of dual blocks remains $q$. Chain dimensions and forcing coefficients are local to the stated construction.
The finite-chain construction uses a matching $\mathcal E$ of primal--dual pairs and the fixed declared coupling arrays. Work first in coordinates normalized by the curvature bounds for $f$ and $g$, so the lower curvatures are one, the objective upper bounds are $\kappa_\iota$, and the declared coupling array is $\mathsf A_H$. Each pair has strength $c_e$ from~\eqref{eq:main-edge-strength}. Known block rescalings preserve all evaluation counts, block spans, and squared errors in the curvature metric.

\subsubsection{Information granted in the lower proof}\label{supp:app:lower-model}
Grant the first coordinate direction of each private objective gradient chain and the first dual direction of each private coupling chain for free. No primal coupling direction is granted. Also grant projections onto the fixed private subspaces. Such a projection cannot transfer a vector between subspaces or reveal a new coordinate within one. A lower bound with these privileges is also a lower bound in Definition~\ref{supp:def:span}.

All support statements hold pathwise, for every seed and every finite realized count. Infinite expected positive-cost work already satisfies the lower bounds. Otherwise all expected individual counts are finite. Dimensions are chosen before the algorithm runs, independently of its seed, evaluations, and realized stopping time.
\subsubsection{A private smooth-gradient chain}\label{supp:app:gradient-chain}
\begin{lemma}[Finite smooth chain]\label{supp:lem:smoothchain}
Let $\kappa\ge4$, $d\ge2$, and $A>0$. Define
\begin{equation}\label{supp:eq:smoothchainparams}
 a=\frac{\kappa-1}{4},\qquad
 \varrho=\frac{\sqrt\kappa-1}{\sqrt\kappa+1},\qquad
 \beta=\frac1{-2\log \varrho}.
\end{equation}
Let $G\in\RR^{d\times d}$ be tridiagonal with
\begin{equation}\label{supp:eq:smoothG}
 G_{jj}=1+2a\ (j<d),\quad G_{dd}=1+2a-a\varrho,\quad
 G_{j,j+1}=G_{j+1,j}=-a.
\end{equation}
The quadratic
\begin{equation}\label{supp:eq:smoothphi}
 \vartheta(u)=\tfrac12u^\top Gu-bu_1,\qquad b=Aa/\varrho,
\end{equation}
satisfies $I\preceq G\preceq\kappa I$ and has minimizer
\begin{equation}\label{supp:eq:smoothsol}
 u_j^\star=A\varrho^{j-1},\qquad
 R_\vartheta^2:=\norm{u^\star}^2=A^2\frac{1-\varrho^{2d}}{1-\varrho^2}.
\end{equation}
If only this gradient oracle can introduce new directions in its private subspace, then every admissible output after $Q_g$ gradient evaluations satisfies
\begin{equation}\label{supp:eq:smoothtail}
 \frac{\norm{u-u^\star}^2}{R_\vartheta^2}
 \ge\frac1{16}e^{-Q_g/\beta}-\varrho^{2d},\qquad
 \beta\ge\frac{\sqrt\kappa}{8}.
\end{equation}
The statement holds for random $Q_g$ and for counts exceeding $d$.
\end{lemma}
\begin{proof}
Write $G=I+aT$, where $T$ has diagonal $2$ except for final entry $2-\varrho$, and off-diagonal entries $-1$. Then
\begin{align*}
 u^\top Tu&=\sum_{j=1}^{d-1}(u_j-u_{j+1})^2+u_1^2+(1-\varrho)u_d^2\ge0,\\
 u^\top(4I-T)u&=\sum_{j=1}^{d-1}(u_j+u_{j+1})^2+u_1^2+(1+\varrho)u_d^2\ge0.
\end{align*}
Thus $0\preceq T\preceq4I$ and $I\preceq G\preceq\kappa I$.
The relation
\[
 1+2a=a(\varrho+\varrho^{-1})
\]
shows that $Gu^\star=be_1$: the first row equals $Aa/\varrho$, every interior row vanishes, and the final row vanishes because of the adjusted last diagonal. Summing the geometric series proves the radius formula.

Let $\mathcal K_k=\spanop\{e_1,\ldots,e_{\min(k,d)}\}$, with $\mathcal K_0=\{0\}$. Tridiagonality implies
$G\mathcal K_k\subseteq\mathcal K_{k+1}$. Even with $e_1$ granted for free, $Q_g$ evaluations leave every available vector in $\mathcal K_{Q_g+1}$. Other private oracles, scalar arithmetic, and the granted summand projections cannot enlarge this support.

For $k<d$, the missing relative energy is
\[
 \frac{\sum_{j=k+1}^d(u_j^\star)^2}{R_\vartheta^2}
 =\frac{\varrho^{2k}-\varrho^{2d}}{1-\varrho^{2d}}
 \ge \varrho^{2k}-\varrho^{2d}.
\]
For $k\ge d$, the right side is nonpositive, so the inequality remains valid. Take $k=Q_g+1$. Since $\varrho\ge1/3$, we have $\varrho^2\ge1/16$, and $\varrho^{2Q_g}=e^{-Q_g/\beta}$, proving the tail bound. Finally, with $s=\sqrt\kappa\ge2$,
\[
 -\log \varrho=\log\frac{s+1}{s-1}
 =\log\left(1+\frac2{s-1}\right)\le\frac2{s-1}.
\]
Consequently $\beta\ge(s-1)/4\ge s/8$.
\end{proof}

\subsubsection{A private bilinear coupling chain}\label{supp:app:edge-chain}
\begin{lemma}[Finite bilinear coupling chain]\label{supp:lem:edgechain}
Let $c\ge2$, $d\ge2$, and $A>0$. Let $S_d$ be the lower shift, $S_de_j=e_{j+1}$ for $j<d$ and $S_de_d=0$. Set
\begin{equation}\label{supp:eq:edgeparams}
 B=\frac c2(I-S_d),\qquad
 \varrho=\frac{\sqrt{1+c^2}-1}{\sqrt{1+c^2}+1},\qquad
 \beta_H=\frac1{-2\log \varrho}.
\end{equation}
Consider
\begin{equation}\label{supp:eq:edgeproblem}
 \cL_{\rm c}(\xi,\eta)=\tfrac12\norm{\xi}^2+\ip{\eta}{B\xi}
                    -\tfrac12\norm{\eta}^2+b\eta_1,
 \qquad b=\frac{A(1+\varrho^{2d+1})}{1-\varrho}.
\end{equation}
Then $\norm{B}_\op\le c$ and the saddle point is
\begin{align}
 \eta_j^\star&=A(\varrho^{j-1}-\varrho^{2d-j+1}),\label{supp:eq:edgesoly}\\
 \xi_j^\star&=-\sqrt \varrho\,A(\varrho^{j-1}+\varrho^{2d-j}),\qquad j=1,\ldots,d.\label{supp:eq:edgesolx}
\end{align}
Its squared radius is
\begin{equation}\label{supp:eq:edgeradius}
 R_H^2=\norm{\xi^\star}^2+\norm{\eta^\star}^2
 =\frac{A^2}{1-\varrho}(1-\varrho^{2d})(1+\varrho^{2d+1}).
\end{equation}
Every point supported within the first $k$ coordinates of both variables satisfies
\begin{equation}\label{supp:eq:edgetail-support}
 \frac{\norm{\xi-\xi^\star}^2+\norm{\eta-\eta^\star}^2}{R_H^2}
 \ge\frac12(\varrho^{2k}-\varrho^{2d}),
\end{equation}
including when $k\ge d$.
\end{lemma}
\begin{proof}
Since $\norm{S_d}_\op\le1$, the norm bound follows. Write $t=c/2$. The definition of $\varrho$ gives
\begin{equation}\label{supp:eq:tq}
 t(1-\varrho)=\sqrt \varrho,\qquad t^2(1-\varrho)^2=\varrho.
\end{equation}
The saddle equations are
\begin{equation}\label{supp:eq:edge-equations}
 \xi+B^\top\eta=0,\qquad \eta-B\xi=be_1.
\end{equation}
For $j<d$, subtract adjacent entries of~\eqref{supp:eq:edgesoly}; the difference is
$A(1-\varrho)(\varrho^{j-1}+\varrho^{2d-j})$.
The same expression holds at $j=d$ with $\eta_{d+1}^\star=0$. Equation~\eqref{supp:eq:tq} proves the first saddle equation. For $j\ge2$, subtract adjacent entries of~\eqref{supp:eq:edgesolx} to obtain
$t(\xi_j^\star-\xi_{j-1}^\star)=\eta_j^\star$. At $j=1$,
\[
 \eta_1^\star-t\xi_1^\star
 =A(1-\varrho^{2d})+\frac{\varrho A}{1-\varrho}(1+\varrho^{2d-1})
 =\frac{A(1+\varrho^{2d+1})}{1-\varrho}=b.
\]
This proves~\eqref{supp:eq:edge-equations}. The matrix $I+BB^\top$ is positive definite, so the solution is unique.

At coordinate $j$, cancellation of cross terms gives
\begin{equation}\label{supp:eq:edge-energy}
 (\xi_j^\star)^2+(\eta_j^\star)^2
 =A^2(1+\varrho)(\varrho^{2j-2}+\varrho^{4d-2j+1}).
\end{equation}
Summing the geometric series yields~\eqref{supp:eq:edgeradius}. For $k<d$, keep just the first nonnegative term in the tail of~\eqref{supp:eq:edge-energy}. The relative tail is at least
\[
 \frac{\varrho^{2k}-\varrho^{2d}}{(1-\varrho^{2d})(1+\varrho^{2d+1})}
 \ge\frac12(\varrho^{2k}-\varrho^{2d}),
\]
since the denominator is at most $2$. For $k\ge d$, the right side is nonpositive.
\end{proof}

\begin{lemma}[Directional support and effective work]\label{supp:lem:directional}
Suppose the first dual direction of Lemma~\ref{supp:lem:edgechain} is initially available, the primal memory is zero, and no other oracle introduces new directions in these private spaces. Let $Q_{\to}$ count $B$ products and $Q_{\leftarrow}$ count $B^\top$ products. Both output variables are supported within
\begin{equation}\label{supp:eq:direction-support}
 \mathcal K_{1+\min\{Q_{\to},Q_{\leftarrow}\}}.
\end{equation}
For any positive $\tau_{\to},\tau_{\leftarrow}$, define
\[
 M_H=\frac{\tau_{\to}Q_{\to}+\tau_{\leftarrow}Q_{\leftarrow}}
            {\tau_{\to}+\tau_{\leftarrow}}.
\]
This effective count need not be integral and imposes no pairing or
balance between directional calls. Then pathwise,
\begin{equation}\label{supp:eq:edge-worktail}
 \frac{\norm{\xi-\xi^\star}^2+\norm{\eta-\eta^\star}^2}{R_H^2}
 \ge\frac1{16}e^{-M_H/\beta_H}-\varrho^{2d},\qquad
 \beta_H\ge c/8.
\end{equation}
The same inequality holds with $M_H$ replaced by either directional count alone.
\end{lemma}
\begin{proof}
Let $a_x,a_y$ bound the available support lengths. Initially $a_x=0,a_y=1$. A forward product maps $\mathcal K_k$ into $\mathcal K_{k+1}$, while a transpose product maps $\mathcal K_k$ into itself. Induction gives
\[
 a_x\le\min\{Q_{\leftarrow},Q_{\to}+1\},\qquad
 a_y\le\min\{Q_{\leftarrow}+1,Q_{\to}+1\}.
\]
Indeed, on a forward call the new dual support is at most $a_x+1$ and $Q_{\to}$ increases by one. On a transpose call the new primal support is at most $a_y$ and $Q_{\leftarrow}$ increases by one. The local gradients of $f$ and $g$ are $\xi$ and $\eta-be_1$; with $e_1$ already available, they do not enlarge support. Linear combinations and the granted projections also preserve these bounds. Equation~\eqref{supp:eq:direction-support} follows.

Since $\min\{Q_{\to},Q_{\leftarrow}\}\le M_H$, Lemma~\ref{supp:lem:edgechain} gives
\[
 \mathrm{err}_H\ge\frac{\varrho^2}{2}e^{-M_H/\beta_H}-\frac12\varrho^{2d}.
\]
For $c\ge2$, $\varrho\ge(\sqrt5-1)/(\sqrt5+1)>3/8$, so $\varrho^2/2\ge1/16$. Enlarging the finite-tail subtraction proves~\eqref{supp:eq:edge-worktail}. The individual versions follow from
$\min\{Q_{\to},Q_{\leftarrow}\}\le Q_{\to},Q_{\leftarrow}$.
Finally, with $s=\sqrt{1+c^2}$, the same logarithm bound as in Lemma~\ref{supp:lem:smoothchain} gives $\beta_H\ge(s-1)/4$. For $c\ge2$, $s-1\ge c/2$, proving $\beta_H\ge c/8$.
\end{proof}

\subsubsection{The common instance and oracle isolation}\label{supp:app:directsum}
Every original variable block has a private smooth chain with parameter $\kappa_\iota$. For each $e=(i,j)\in\mathcal E$, give its endpoints additional private coupling subspaces and use Lemma~\ref{supp:lem:edgechain} with $c=c_e$, dimension $d_e$, amplitude $A_e$, and forcing $b_e$:
\[
 x_i=(u_i,\xi_e),\qquad y_j=(v_j,\eta_e)\quad(e=(i,j)\in\mathcal E).
\]
Unmatched blocks consist only of $u_i$ or $v_j$. The matching guarantees that each block has at most one additional coupling subspace. Let $\vartheta_i$ and $\zeta_j$ be the smooth quadratics of Lemma~\ref{supp:lem:smoothchain}. Define
\begin{align}
 f(x)&=\sum_i\vartheta_i(u_i)+\tfrac12\sum_{e\in\mathcal E}\norm{\xi_e}^2,\label{supp:eq:lower-f}\\
 g(y)&=\sum_j\zeta_j(v_j)+\sum_{e\in\mathcal E}
       \bigl(\tfrac12\norm{\eta_e}^2-b_e(\eta_e)_1\bigr),\label{supp:eq:lower-g}\\
 H(x,y)&=\sum_{e\in\mathcal E}\ip{\eta_e}{B_e\xi_e},\qquad
 B_e=(c_e/2)(I-S_{d_e}).\label{supp:eq:lower-H}
\end{align}
Thus the saddle objective is the orthogonal sum of all private smooth problems and one coupling problem per edge. Its saddle point is the collection of the explicit private solutions.

On an objective gradient subspace the Hessian lies between $I$ and the appropriate $\kappa_{\iota}I$; on each coupling subspace it is $I$. Therefore the joint models hold. The two partial-gradient variations on $e=(i,j)$ have cross coefficient at most $\norm{B_e}_\op\le c_e$. By~\eqref{eq:main-edge-strength}, this is at most both $(\mathsf A_H)_{i,m+j}$ and $(\mathsf A_H)_{m+j,i}$. Every other functional coupling coefficient is zero and hence is bounded by its fixed nonnegative declaration. Undoing normalization multiplies the two cross coefficients by $\sqrt{\mu_i^f\mu_j^g}$, so both original arrays are respected. This proves admissibility with the full prescribed arrays, without changing any of their entries.

The actual normalized full-operator norm is $\max_{e\in\mathcal E}\norm{B_e}_\op\le\max_e c_e\le L_H$ for a nonempty matching, and is zero for an empty matching. When the declared array itself consists only of symmetric matching entries $c_e$, its row bounds on each pair are $c_e$, its input overlaps are one, and its derived norm is $L_H=\max_e c_e$. These are precisely the parameters used for the matching upper comparison.

\paragraph{Partial-gradient oracle isolation.}
For $e=(i,j)$, the return $\nabla_{x_i}H$ is zero on its private smooth
subspace and equals $B_e^\top\eta_e$ on its coupling subspace.
Similarly, $\nabla_{y_j}H$ equals $B_e\xi_e$ only on its coupling
subspace. Unmatched coupling partials are zero; partials at other matched
endpoints cannot advance this chain. Apply Lemma~\ref{supp:lem:directional}
separately to each edge with
\[
 Q_{\to}=N_{H,j}^y,\qquad Q_{\leftarrow}=N_{H,i}^x,
 \qquad\tau_{\to}=\tau_{H,j}^y,\quad\tau_{\leftarrow}=\tau_{H,i}^x.
\]
The remaining input blocks have no effect on its return.

Every coupling return vanishes on all smooth chains. An evaluation to another objective gradient block cannot affect a selected private smooth chain. On coupling subspaces, objective gradient returns are only $\xi_e$ and $\eta_e-b_e e_1$. Each first dual direction was granted initially in Section~\ref{supp:app:lower-model}; neither objective return then extends the directional supports. An objective gradient evaluation may return both a private smooth-chain vector and an identity vector in its coupling subspace. The granted projections allow these vectors to be separated and still cannot bypass either support bound. Thus all private lower estimates hold simultaneously for this one instance.

\subsubsection{Cost-weighted radii and random stopping}\label{supp:app:weighted}
Let $\ell\in\mathcal P=V\mathbin{\dot\cup}\mathcal E$ index the private problems. Their chain parameters are $\varrho_\ell,\beta_\ell,d_\ell$. Retain the objective costs $\tau_\iota$ from~\eqref{supp:eq:tag-dictionary} and extend the cost notation to edges:
\begin{equation}\label{supp:eq:nu-beta}
 \tau_e=\tau_{H,i}^x+\tau_{H,j}^y\ (e=(i,j)),\qquad
 S_\tau=\sum_{\ell\in\mathcal P}\tau_\ell\beta_\ell.
\end{equation}
Assign squared solution radii
\begin{equation}\label{supp:eq:radii}
 R_\ell^2=R^2\frac{\tau_\ell\beta_\ell}{S_\tau},\qquad\sum_\ell R_\ell^2=R^2,
\end{equation}
and choose finite dimensions
\begin{equation}\label{supp:eq:dimensions}
 d_\ell\ge\max\{2,\lceil\beta_\ell\log(1/\varepsilon)\rceil\}.
\end{equation}
Then $\varrho_\ell^{2d_\ell}=\exp(-d_\ell/\beta_\ell)\le\varepsilon$. Explicit amplitudes realizing~\eqref{supp:eq:radii} are
\begin{equation}\label{supp:eq:amplitudes}
 A_\ell=\begin{cases}
 R_\ell\sqrt{(1-\varrho_\ell^2)/(1-\varrho_\ell^{2d_\ell})},&\ell\in V,\\[3pt]
 R_\ell\sqrt{(1-\varrho_\ell)/[(1-\varrho_\ell^{2d_\ell})(1+\varrho_\ell^{2d_\ell+1})]},&\ell\in\mathcal E.
 \end{cases}
\end{equation}
All forcing terms are now fixed by~\eqref{supp:eq:smoothphi} and~\eqref{supp:eq:edgeproblem}. Nothing in the construction depends on the realized random seed.

For each realized history, define
\begin{equation}\label{supp:eq:effective-counts}
 M_\iota=N_\iota,\qquad
 M_e=\frac{\tau_{H,i}^xN_{H,i}^x+\tau_{H,j}^yN_{H,j}^y}{\tau_{H,i}^x+\tau_{H,j}^y}
 \quad(e=(i,j)\in\mathcal E).
\end{equation}
Let $W_{\rm act}=\sum_\ell\tau_\ell M_\ell$ be the work charged to the private
objective gradient chains and all matched coupling partials. No counter is
charged twice because $\mathcal E$ is a matching. Pathwise,
$W_{\rm act}\le W_\tau$; calls to any other coupling row have
nonnegative cost and cannot advance a chain. By orthogonality and Lemmas~\ref{supp:lem:smoothchain} and~\ref{supp:lem:directional},
\begin{equation}\label{supp:eq:pathwise-direct}
 \frac{\norm{z^{\out}-z^\star}^2}{R^2}
 \ge\frac1{16}\sum_\ell\pi_\ell e^{-M_\ell/\beta_\ell}-\varepsilon,
 \qquad\pi_\ell=\frac{\tau_\ell\beta_\ell}{S_\tau}.
\end{equation}
The weights are positive and sum to one. If $\E W_\tau<\infty$, Jensen's inequality first over randomness and then over these weights gives
\begin{align*}
 \sum_\ell\pi_\ell\E e^{-M_\ell/\beta_\ell}
 &\ge\sum_\ell\pi_\ell e^{-\E M_\ell/\beta_\ell}
 \ge\exp\left(-\sum_\ell\pi_\ell\frac{\E M_\ell}{\beta_\ell}\right)\\
 &=\exp(-\E W_{\rm act}/S_\tau)\ge\exp(-\E W_\tau/S_\tau).
\end{align*}
Undoing normalization gives
\begin{equation}\label{supp:eq:path-work}
 \frac{\E\norm{z^{\out}-z^\star}_M^2}{R^2}
 \ge\frac1{16}\exp(-\E W_\tau/S_\tau)-\varepsilon.
\end{equation}
Relative squared error at most $\varepsilon$ implies
\[
 \E W_\tau\ge S_\tau\log\frac1{32\varepsilon},\qquad
 S_\tau\ge\frac18\left[\sum_{i=1}^m\tau_i^f\sqrt{\kappa_i^f}
       +\sum_{j=1}^q\tau_j^g\sqrt{\kappa_j^g}
 +\sum_{(i,j)\in\mathcal E}(\tau_{H,i}^x+\tau_{H,j}^y)c_{ij}\right].
\]
Undoing the known block scaling proves the common-instance lower bound.

Finite tail subtraction is essential: in each support-tail estimate the
terminal term $\varrho_\ell^{2d_\ell}$ is subtracted, so the estimate remains
valid even when the chain has been fully revealed and its error can be
zero. The weighted terminal terms sum to at most $\varepsilon$ by
\eqref{supp:eq:dimensions}, yielding~\eqref{supp:eq:pathwise-direct}
for every realized count, including counts beyond the chosen dimensions.
Taking expectations therefore needs neither a limit in dimension nor an
optional-stopping assertion. A sufficient original block dimension is its smooth-chain dimension, plus $d_e$ if it is an endpoint of edge $e$. These dimensions may grow with accuracy, as allowed by the theorem.

For a separate objective gradient bound, put the entire squared radius in its private chain and set every other forcing to zero. Jensen gives
\[
 \varepsilon\ge\tfrac1{16}e^{-\E N_\iota/\beta_\iota}-\varepsilon,
\]
which yields~\eqref{supp:eq:lower-individual-strong}--\eqref{supp:eq:lower-individual-strong-g}. For the separate coupling instance on $e=(i,j)$, put the entire radius in this coupling chain and apply both individual-direction versions of Lemma~\ref{supp:lem:directional}. Set other forcing terms to zero; their solutions then vanish, while all declared coupling arrays stay fixed. This proves both inequalities in~\eqref{supp:eq:lower-H-strong} on the same endpoint instance.

\section{The convex--concave case}\label{supp:sec:convex}
\subsection{Euclidean smoothness bounds and comparison sets}
Use Assumption~\ref{ass:convex} and the compact product comparison sets
$Z=X\times Y$ from Section~\ref{sec:convex}.
The target is $\E\Gap(z^{\out})\le\varepsilon$, with the restricted gap
in~\eqref{eq:main-gap}.
Retain the original-coordinate array $\mathsf L_H$
of~\eqref{supp:eq:coupling-scalar-matrix}
and the Euclidean row bounds $L_{H,i}^x,L_{H,j}^y$ and objective bound
$L_{\max}$ from Section~\ref{sec:convex}. Define
\begin{equation}\label{supp:eq:coupling-derived-reference}
 L_H:=\norm{\mathsf L_H}_\op.
\end{equation}
Thus $L_H$ is the operator norm of the original-coordinate coupling
array, and the Euclidean row bounds are its row norms.
The original zero pattern gives the same supports and overlap parameters.
For internal proof formulas, use the tags of~\eqref{supp:eq:tag-dictionary}.
Lemma~\ref{supp:lem:coupling-normalization} applies with $I$ in place of
$M$. Every row uses the same sampling and cache rules. A global bound
for $F=\nabla h+T_H$ is
\begin{equation}\label{supp:eq:Lglobal}
 L_F:=L_{\max}+L_H\ge\Lip(F).
\end{equation}
A supplied sharper bound that also dominates $L_{\max}$ may replace
$L_F$ and its occurrences in the stage-budget logarithms.

\subsection{Regularization and an implementable wrapper}
For $\varepsilon>0$, set $\sigma=\varepsilon/R^2$
and put
\begin{equation}\label{supp:eq:regparams}
 \chi_\sigma=1+L_F/\sigma,\qquad
 \rho_\varepsilon=(64\chi_\sigma^2)^{-1}.
\end{equation}
With $L_F=L_{\max}+L_H$ as in~\eqref{supp:eq:Lglobal}, define the common stage budget
\begin{equation}\label{supp:eq:convex-epoch-budget}
 E_\varepsilon:=\left\lceil\log_2\!\left[
 128\left(1+\frac{L_{\max}R^2}{\varepsilon}\right)
 \left(1+\frac{(L_{\max}+L_H)R^2}{\varepsilon}\right)^2
 \right]\right\rceil.
\end{equation}
In original coordinates the regularized saddle objective is
\begin{equation}\label{supp:eq:regularized-saddle}
 \cL_\sigma(x,y)=f_\sigma(x)+H(x,y)-g_\sigma(y),
\end{equation}
where
\[
 f_\sigma(x)=f(x)+\tfrac\sigma2\norm{x-x^0}^2,
 \qquad
 g_\sigma(y)=g(y)+\tfrac\sigma2\norm{y-y^0}^2.
\]
\begin{algorithm}[H]
\caption{Regularization for an expected restricted-gap target}\label{supp:alg:convex}
\begin{algorithmic}[1]
\Require The data of Assumption~\ref{ass:convex}, initial point $z^0$, and $\varepsilon>0$.
\Ensure An output satisfying Theorem~\ref{supp:thm:convex}.
\State Compute the Euclidean bounds~\eqref{supp:eq:coupling-derived-reference}--\eqref{supp:eq:Lglobal} and parameters~\eqref{supp:eq:regparams}.
\State Supply gradient wrappers for $f_\sigma,g_\sigma$ in~\eqref{supp:eq:regularized-saddle}, with $M_\sigma=\sigma I$ and $L_\sigma=L+\sigma I$; retain the coupling oracles and arrays.
\State Run Algorithm~\ref{supp:alg:strong} with these data, start $z^0$, and tolerance $\rho_\varepsilon$, using $E_\varepsilon$ stages from~\eqref{supp:eq:convex-epoch-budget}.
\State \Return its output in original coordinates $z^{\out}$.
\end{algorithmic}
\end{algorithm}

Each regularized objective partial uses one original call plus its known
affine term; the coupling oracles are unchanged. The comparison sets
enter the wrapper only through their certified radius $R$.

\subsection{Uniform gap transfer}\label{supp:sec:gap-transfer}
\begin{lemma}[Uniform nonlinear gap transfer]\label{supp:lem:gap-transfer}
Under Assumption~\ref{ass:convex}, let $z_\sigma$ be the unique zero of
$F_\sigma(z)=F(z)+\sigma(z-z^0)$, with $\sigma=\varepsilon/R^2$. Then $\norm{z_\sigma-z^0}\le R$, and every $z=z_\sigma+e$ satisfies
\begin{equation}\label{supp:eq:gap-transfer}
 \Gap(z)\le\frac\varepsilon2+(\sigma+2L_F)R\norm{e}+L_F\norm{e}^2.
\end{equation}
\end{lemma}
\begin{proof}
Existence and uniqueness follow from the global Lipschitz continuity and $\sigma$-strong monotonicity of $F_\sigma$; see Appendix~\ref{supp:app:existence}. Monotonicity against a zero $z^\star\in Z$ gives
\[
 0\le\ip{F(z_\sigma)-F(z^\star)}{z_\sigma-z^\star}
 =-\sigma\ip{z_\sigma-z^0}{z_\sigma-z^\star}.
\]
Thus $\norm{z_\sigma-z^0}^2\le
\ip{z_\sigma-z^0}{z^\star-z^0}$, proving the radius bound.

For $v=(\xi,\eta)$, convexity in $x$ and concavity in $y$ yield
\begin{equation}\label{supp:eq:gap-residual}
 \cL(x,\eta)-\cL(\xi,y)\le\ip{F(z)}{z-v}.
\end{equation}
At the regularized zero,
\begin{align}\label{supp:eq:gap-at-reg}
 \ip{F(z_\sigma)}{z_\sigma-v}
 &=-\sigma\ip{z_\sigma-z^0}{z_\sigma-v}\nonumber\\
 &=\tfrac\sigma2\left(\norm{v-z^0}^2-\norm{z_\sigma-z^0}^2
                         -\norm{z_\sigma-v}^2\right)
 \le\varepsilon/2.
\end{align}
For $z=z_\sigma+e$, expand the residual pairing:
\begin{align*}
 \ip{F(z)}{z-v}
 &=\ip{F(z_\sigma)}{z_\sigma-v}+\ip{F(z_\sigma)}{e}
   +\ip{F(z)-F(z_\sigma)}{z_\sigma-v+e}\\
 &\le\varepsilon/2+\sigma R\norm{e}
      +L_F\norm{e}(2R+\norm{e}).
\end{align*}
This bound is uniform over the comparison set. Taking the supremum proves~\eqref{supp:eq:gap-transfer}. Unlike a bilinear comparison identity, this argument uses the full nonlinear saddle operator and its global Lipschitz bound.
\end{proof}

\subsection{Convex oracle guarantee and proof}\label{supp:sec:convex-upper}
\begin{theorem}[Simultaneous nonlinear convex bounds]\label{supp:thm:convex}
Under Assumption~\ref{ass:convex}, Algorithm~\ref{supp:alg:convex} applies Theorem~\ref{supp:thm:strong} to the regularized problem at relative mean-square tolerance $\rho_\varepsilon$ in~\eqref{supp:eq:regparams}. Use the default allocation~\eqref{supp:eq:dense-weights}. The output satisfies $\E\Gap(z^{\out})\le\varepsilon$. With $E_\varepsilon$ defined in~\eqref{supp:eq:convex-epoch-budget}, the same run obeys, for all
$1\le i\le m$ and $1\le j\le q$,
\begin{align}
 \E N_i^f
 &\le\left(3+36\sqrt2\sqrt{1+\frac{2L_i^fR^2}{\varepsilon}}\right)E_\varepsilon,
     \label{supp:eq:convex-gradient}\\
 \E N_j^g
 &\le\left(3+36\sqrt2\sqrt{1+\frac{2L_j^gR^2}{\varepsilon}}\right)E_\varepsilon,
     \label{supp:eq:convex-gradient-g}\\
 \E N_{H,i}^x
 &\le\left(39+72\sqrt2\,
       \frac{L_{H,i}^x\sqrt{\Delta_{H,i}^x}R^2}{\varepsilon}\right)E_\varepsilon,
     \label{supp:eq:convex-H}\\
 \E N_{H,j}^y
 &\le\left(39+72\sqrt2\,
       \frac{L_{H,j}^y\sqrt{\Delta_{H,j}^y}R^2}{\varepsilon}\right)E_\varepsilon.
     \label{supp:eq:convex-H-y}
\end{align}
Declared zero Lipschitz rows have overlap one and obey the same
$O(E_\varepsilon)$ evaluation bounds.
Moreover, the explicit common factor satisfies
\begin{equation}\label{supp:eq:convex-log-bound}
 E_\varepsilon
 \le 8+3\log_2\!\left(1+\frac{(L_{\max}+L_H)R^2}{\varepsilon}\right).
\end{equation}
\end{theorem}
\begin{proof}
The regularized lower curvature is $\sigma$ on every block. Normalize by $\sqrt\sigma$. The parameters under strong convexity of $f$ and $g$ are
\[
 \begin{gathered}
 \kappa_\iota^\sigma=1+L_\iota/\sigma,\qquad
 c_\iota^\sigma=\sqrt2\sqrt{1+2L_\iota/\sigma},\qquad L_H^\sigma=L_H/\sigma,\\
 \kappa_{H,i}^{x,\sigma}=L_{H,i}^x/\sigma,\qquad
 \kappa_{H,j}^{y,\sigma}=L_{H,j}^y/\sigma.
 \end{gathered}
\]
The relative distance tolerance $\rho_\varepsilon$ therefore requires exactly the stage budget in~\eqref{supp:eq:convex-epoch-budget}. Undoing the scalar normalization,
\[
 \E\norm{z^{\out}-z_\sigma}^2
 \le\rho_\varepsilon\norm{z^0-z_\sigma}^2\le R^2/(64\chi_\sigma^2).
\]
Take the supremum in Lemma~\ref{supp:lem:gap-transfer} before taking expectation, and use Cauchy--Schwarz. Since $\sigma+2L_F\le2\sigma \chi_\sigma$ and $L_F\le\sigma \chi_\sigma^2$,
\[
 \E\Gap(z^{\out})
 \le\varepsilon/2+\frac{(\sigma+2L_F)R^2}{8\chi_\sigma}
              +\frac{L_FR^2}{64\chi_\sigma^2}
 \le\frac{49}{64}\varepsilon<\varepsilon.
\]
Use Theorem~\ref{supp:thm:epoch} for the number of stages displayed in~\eqref{supp:eq:convex-epoch-budget}, substitute the regularized parameters into~\eqref{supp:eq:count-general}, and then replace $\sigma$ by $\varepsilon/R^2$ and $\chi_\sigma$ by $1+(L_{\max}+L_H)R^2/\varepsilon$. This gives~\eqref{supp:eq:convex-gradient}--\eqref{supp:eq:convex-H-y} with the common factor $E_\varepsilon$. Finally, $1+L_{\max}/\sigma\le\chi_\sigma$ and $\lceil\log_2(128\chi_\sigma^3)\rceil\le8+3\log_2\chi_\sigma$ prove~\eqref{supp:eq:convex-log-bound}.
\end{proof}

\subsection{A lower bound without strong convexity of $f$ and $g$}
For a matching $\mathcal E$, define the original-coordinate strengths
$c_e^0=c_{ij}^0=\min\{L_{ij}^{H,xy},L_{ji}^{H,yx}\}$ for $e=(i,j)$.
\begin{theorem}[Restricted-gap lower bound for non-strongly-convex functions $f$ and $g$]\label{supp:thm:convex-lower}
Fix the original coupling arrays, a matching $\mathcal E$, $R,\varepsilon>0$, positive objective smoothness bounds $L_i^f$ for every $i$ and $L_j^g$ for every $j$, and positive oracle costs. Suppose
\begin{equation}\label{supp:eq:convex-lower-regime}
 \begin{aligned}
 \varepsilon&\le L_i^fR^2/1024\quad(1\le i\le m),\\
 \varepsilon&\le L_j^gR^2/1024\quad(1\le j\le q),\\
 \varepsilon&\le c_e^0R^2/512\quad(e\in\mathcal E).
 \end{aligned}
\end{equation}
Consider the convex class with these objective upper bounds and the fixed declared coupling arrays. The instance and sufficiently large finite dimensions may depend on $\varepsilon$ and the positive cost vector, but not on the realized random seed. There is a single bilinear instance in this class with zero strong-convexity modulus in every original variable block and product comparison sets as in Assumption~\ref{ass:convex}, such that every algorithm of Definition~\ref{supp:def:span} with expected gap at most $\varepsilon$ obeys
\begin{equation}\label{supp:eq:lower-work-convex}
 \begin{aligned}
 \E W_\tau\ge\frac{\log4}{128}
 \left[\sum_{i=1}^m\tau_i^f\sqrt{L_i^fR^2/\varepsilon}
       +\sum_{j=1}^q\tau_j^g\sqrt{L_j^gR^2/\varepsilon}\right]\\
 {}+\frac{\log4}{2048}\sum_{(i,j)\in\mathcal E}
 (\tau_{H,i}^x+\tau_{H,j}^y)c_{ij}^0R^2/\varepsilon.
 \end{aligned}
\end{equation}
For each fixed primal block $i$ there is an instance requiring the first bound below, and for each fixed dual block $j$ there is an instance requiring the second:
\begin{align}
 \E N_i^f&\ge\frac{\log4}{128}\sqrt{L_i^fR^2/\varepsilon},\label{supp:eq:convex-lower-f}\\
 \E N_j^g&\ge\frac{\log4}{128}\sqrt{L_j^gR^2/\varepsilon}.\label{supp:eq:convex-lower-g}
\end{align}
For each $(i,j)\in\mathcal E$, there is also a coupling-only instance in the same fixed-array class requiring both
\[
 \E N_{H,i}^x\ge\frac{\log4}{2048}\frac{c_{ij}^0R^2}{\varepsilon},\qquad
 \E N_{H,j}^y\ge\frac{\log4}{2048}\frac{c_{ij}^0R^2}{\varepsilon}.
\]
Both endpoint bounds hold on one instance per edge; these and the separate objective instances need not coincide. If the prescribed array has only symmetric cross entries on $\mathcal E$, with all other entries zero, then $c_{ij}^0=L_{H,i}^x=L_{H,j}^y$ and both overlaps are one for every $(i,j)\in\mathcal E$. Thus the same blockwise constants enter the lower and upper coupling bounds, with the upper common logarithm remaining.
\end{theorem}
\begin{proof}[Finite-dimensional construction and gap reduction]\label{supp:app:lower-convex}
Introduce the auxiliary active curvature
\begin{equation}\label{supp:eq:activecurvature}
 s=256\varepsilon/R^2.
\end{equation}
Condition~\eqref{supp:eq:convex-lower-regime} is equivalent to $L_\iota\ge4s$ and $c_e^0\ge2s$ for all matched edges. Use the active instance of Theorem~\ref{supp:thm:strong-lower} with curvature $s$ on every block, objective upper bounds $L_\iota$, the fixed coupling array $\mathsf L_H$ in the original coordinates, and target relative error $\varepsilon_0=1/128$. Its curvature-normalized parameters are
\[
 \kappa_\iota^{\rm act}=L_\iota/s,\qquad
 \mathsf A_H^{\rm act}=\mathsf L_H/s,\qquad c_e^{\rm act}=c_e^0/s.
\]
Choose its squared radius in the curvature metric as $sR^2$, equivalently $\norm{z_{\rm act}^\star}^2=R^2$ in Euclidean coordinates. The finite-chain formulas in Appendix~\ref{supp:sec:lower} specify
all amplitudes.

Append one scalar coordinate to every original block: $\bar x_i=(x_i,t_i)$ and $\bar y_j=(y_j,u_j)$. Define
\[
 f_0(\bar x)=f_{\rm act}(x),\qquad
 g_0(\bar y)=g_{\rm act}(y),\qquad
 H_0(\bar x,\bar y)=H_{\rm act}(x,y).
\]
All objectives are constant in the added directions, so each original variable block has strong-convexity modulus exactly zero. Every oracle returns zero in those directions. The joint upper models and coupling Lipschitz bound are unchanged, and the active oracle process is identical to the active hard problem. Select a saddle point with every flat coordinate zero.

Set $\rho_{x,i}=\norm{x_{\rm act,i}^\star}$ and $\rho_{y,j}=\norm{y_{\rm act,j}^\star}$ and define
\begin{equation}\label{supp:eq:comparison-balls}
 X=\prod_i\{\bar x_i:\norm{\bar x_i}\le\rho_{x,i}\},\qquad
 Y=\prod_j\{\bar y_j:\norm{\bar y_j}\le\rho_{y,j}\}.
\end{equation}
These compact product sets contain the selected saddle point and have total squared radius $R^2$. They reveal only scalar radii, not solution directions. Even Euclidean projections onto their individual balls preserve block spans because they are scalar rescalings.

Every output, including one outside the comparison sets, satisfies
\begin{align}\label{supp:eq:lower-gap-distance}
 \Gap(\bar x,\bar y)
 &\ge\cL_0(\bar x,\bar y^\star)-\cL_0(\bar x^\star,\bar y)\nonumber\\
 &=D_{f_{\rm act}}(x,x^\star)+D_{g_{\rm act}}(y,y^\star)
 \ge\tfrac s2\norm{z_{\rm act}-z_{\rm act}^\star}^2.
\end{align}
The equality follows from the bilinear active saddle equations; the flat coordinates do not enter. Hence expected gap at most $\varepsilon$ requires
\begin{equation}\label{supp:eq:active-target}
 \E\norm{z_{\rm act}^{\out}-z_{\rm act}^\star}^2
 \le2\varepsilon/s=R^2/128.
\end{equation}
The active span process has the same oracle counts, so its lower bound under strong convexity of $f$ and $g$ gives
\[
 \E W_\tau\ge\frac{\log4}{8}
 \left[\sum_{\iota\in V}\tau_\iota\sqrt{L_\iota/s}
 +\sum_{(i,j)\in\mathcal E}(\tau_{H,i}^x+\tau_{H,j}^y)c_{ij}^0/s\right].
\]
Substitution of~\eqref{supp:eq:activecurvature} proves Theorem~\ref{supp:thm:convex-lower}. The separate statements follow from the corresponding single-private-problem constructions before appending the flat coordinates. For each private problem $\ell\in\mathcal P$, active chain dimensions $\max\{2,\lceil\beta_\ell\log128\rceil\}$ suffice; one extra coordinate per original block makes each variable block non-strongly convex.
\end{proof}
The aggregate single-edge version follows by declaring only the two cross entries on one pair nonzero, both equal to an aggregate bound $\overline L_H$. It yields the coupling work term $(\log4/2048)(\tau_{H,i_0}^x+\tau_{H,j_0}^y)\overline L_HR^2/\varepsilon$, with the same objective terms and the accuracy condition $\varepsilon\le\overline L_HR^2/512$.

\FloatBarrier

\section{Auxiliary facts and implementation details}\label{supp:app:aux}
\subsection{Monotonicity of a differentiable saddle operator}\label{supp:app:monotonicity}
Let $\psi$ be differentiable, convex in $x$, and concave in $y$. Put
$S(x,y)=(\nabla_x\psi(x,y),-\nabla_y\psi(x,y))$.
For $(x,y)$ and $(v,t)$, convexity gives
\begin{align*}
 \ip{\nabla_x\psi(x,y)}{x-v}&\ge\psi(x,y)-\psi(v,y),\\
 \ip{\nabla_x\psi(v,t)}{v-x}&\ge\psi(v,t)-\psi(x,t),
\end{align*}
and concavity gives
\begin{align*}
 \ip{-\nabla_y\psi(x,y)}{y-t}&\ge\psi(x,t)-\psi(x,y),\\
 \ip{-\nabla_y\psi(v,t)}{t-y}&\ge\psi(v,y)-\psi(v,t).
\end{align*}
Adding the four inequalities cancels the function values and proves
$\ip{S(x,y)-S(v,t)}{(x-v,y-t)}\ge0$.
This argument requires no second derivatives.

Similarly, for the full saddle function, split
\[
 \cL(x,\eta)-\cL(\xi,y)
 =[\cL(x,\eta)-\cL(x,y)]+[\cL(x,y)-\cL(\xi,y)].
\]
Concavity bounds the first difference by
$\ip{\nabla_y\cL(x,y)}{\eta-y}$, and convexity bounds the second by
$\ip{\nabla_x\cL(x,y)}{x-\xi}$. Their sum is exactly~\eqref{supp:eq:gap-residual}.

\subsection{A globally Lipschitz strongly monotone operator has a zero}\label{supp:app:existence}
Let $G:\RR^d\to\RR^d$ be $s$-strongly monotone and $L$-Lipschitz, with $s>0$. Necessarily $L\ge s$. Choose $t=s/L^2$. For the map $A(z)=z-tG(z)$,
\begin{align*}
 \norm{A(z)-A(v)}^2
 &=\norm{z-v}^2-2t\ip{G(z)-G(v)}{z-v}
                  +t^2\norm{G(z)-G(v)}^2\\
 &\le(1-s^2/L^2)\norm{z-v}^2.
\end{align*}
Thus $A$ is a contraction on the complete Euclidean space (possibly with contraction factor zero). The contraction mapping theorem gives a unique fixed point, hence a unique zero of $G$. In the normalized problem with strong convexity of $f$ and $g$ one may use $s=1$ and $L=\kappa_{\max}+L_H$. For the regularized problem use $s=\sigma$ and $L=L_F+\sigma$. This is an existence argument, not an additional algorithm or a claimed oracle-efficient solver.

\subsection{Conjugacy, lifted curvature, and represented divergence}\label{supp:app:conjugate}
From~\eqref{supp:eq:normalized-model},
\[
 \tfrac14\norm{z-v}^2\le D_\phi(z,v)
 \le\tfrac12\norm{z-v}_{\Lambda}^2.
\]
The lower bound makes $\phi$ coercive after subtraction of any linear functional. For every $u$, the maximizer in
$\phi^*(u)=\sup_z\{\ip{u}{z}-\phi(z)\}$ exists and is unique. Its first-order condition is $u=\nabla\phi(z)$, yielding
$\nabla\phi^*(u)=z$ and the inverse relations used in~\eqref{supp:eq:representation}.

To prove the $\Lambda^{-1}$ lower model of $\phi^*$ directly, fix $u$, write $a=\nabla\phi^*(u)$, and upper-bound $\phi$ at $a$:
\[
 \phi(z)\le\phi(a)+\ip{u}{z-a}+\tfrac12\norm{z-a}_\Lambda^2.
\]
For any $v$, taking the supremum of $\ip{v}{z}$ minus this upper bound gives
\[
 \phi^*(v)\ge\phi^*(u)+\ip{a}{v-u}
                         +\tfrac12\norm{v-u}_{\Lambda^{-1}}^2.
\]
Together with the quadratic in the original variables in $r$, this proves~\eqref{supp:eq:r-strong}.

More generally, if a differentiable finite function $r$ is $1$-strongly convex in $D\succ0$, its lower quadratic model at $w=\nabla r^*(\theta)$ yields
\[
 r^*(\theta+s)\le r^*(\theta)+\ip{w}{s}+\tfrac12\norm{s}_{D^{-1}}^2.
\]
This is the conjugate smoothness inequality used in~\eqref{supp:eq:correction-bound}. The relevant conjugate divergence identity follows by inserting the Fenchel equalities:
\[
 D_r(v,w)=D_{r^*}(\nabla r(w),\nabla r(v)).
\]
In particular, with $u=\nabla\phi(a)$ and $u^\star=\nabla\phi(z^\star)$,
\[
 D_{\phi^*}(u^\star,u)=D_\phi(a,z^\star),
\]
which proves~\eqref{supp:eq:physical-divergence}. At $a=z=z^0$, the $\|z-z^\star\|^2/4$ term restores the quadratic subtracted in $\phi$, so the total divergence is $D_h(z^0,z^\star)$.

\subsection{Normalized oracle implementations}\label{supp:app:normalization}
For $\widehat f(\widehat x)=f(M_x^{-1/2}\widehat x)$,
\[
 \nabla_{\widehat x_i}\widehat f(\widehat x)
 ={\mu_i^f}^{-1/2}\nabla_{x_i}f(M_x^{-1/2}\widehat x).
\]
Define $\widehat H(\widehat x,\widehat y)=H(M_x^{-1/2}\widehat x,M_y^{-1/2}\widehat y)$. Then
\begin{align*}
 \nabla_{\widehat x_i}\widehat H(\widehat x,\widehat y)
 &=(\mu_i^f)^{-1/2}\nabla_{x_i}H(M_x^{-1/2}\widehat x,M_y^{-1/2}\widehat y),\\
 \nabla_{\widehat y_j}\widehat H(\widehat x,\widehat y)
 &=(\mu_j^g)^{-1/2}\nabla_{y_j}H(M_x^{-1/2}\widehat x,M_y^{-1/2}\widehat y).
\end{align*}
One original partial-gradient evaluation yields exactly its normalized derivative block by known scalings. The same one-call implementation applies to the scalar normalization used for the regularized problem in Appendix~\ref{supp:sec:convex}. Counts and permitted block spans are preserved.

\section{Experiments and evaluation protocol}\label{supp:sec:experiments}

We evaluate exact implementations of the sampled method, its
full-assembly alternative, and full-oracle extragradient on ten small
data-backed instances, followed by a joint stronger-baseline rerun and
controls that isolate several mechanisms. Six bilinear instances also admit a
structure-exploiting primal baseline. These experiments measure both
separately charged calls and runtime. The application data determine
the matrices and labels; observations, task partitions, loss weights,
and score-shift interactions are constructed as described below.
We do not interpret these instances as field deployments or as tests
of every regime covered by the theorem.

\subsection{Execution and evaluation protocol}\label{supp:sec:experiment-protocol}
This appendix describes the benchmark models, counted methods, execution
settings, and evaluation protocol for the initial suite, the separate
EG sensitivity study, and the joint rerun. Tables and figures report
measurements from the corresponding implementations. Source code,
executable reproduction commands, and raw experimental records are not
included with this submission. The descriptions below document the
experimental procedure and public data sources.

All instances start at $z^0=0$. An independent damped Newton solve of
$F(z)=0$ produces $z_{\mathrm{ref}}$ before any timed method starts.
Its line search decreases the Euclidean operator residual. We check
the normalized residual
\[
 r_{\mathrm{ref}}=\norm{M^{-1/2}F(z_{\mathrm{ref}})}.
\]
In normalized coordinates $F$ is $1$-strongly monotone, so
$\norm{z_{\mathrm{ref}}-z^\star}_M\le r_{\mathrm{ref}}$.
The runner rejects a reference unless
\[
 r_{\mathrm{ref}}
 \le 10^{-3}\sqrt{10^{-6}\norm{z^0-z_{\mathrm{ref}}}_M^2}.
\]
The reported optimization error is
\begin{equation}\label{supp:eq:experiment-error}
 e(z)=\frac{\norm{z-z_{\mathrm{ref}}}_M^2}
                {\norm{z^0-z_{\mathrm{ref}}}_M^2}.
\end{equation}
A run stops at its first recorded point with $e(z)\le10^{-6}$.
The sampled and assembled methods are checked after complete stages;
extragradient and primal acceleration are checked every ten
iterations. Their terminal errors can consequently differ.
The reference and these diagnostics are evaluation infrastructure:
they are not supplied oracles or inputs to the stage updates.
This observed first-crossing criterion does not establish an
expected-error guarantee from five realizations. It also does not
replace the theorem's prescribed stage count.

Each randomized method uses seeds $0,1,2,3,4$. We repeat deterministic
methods five times to measure timing variability; those repetitions
have identical call counts and iterates. Optimization tables report arithmetic
means and sample standard deviations (SD) across the five runs.
These SDs, including the plotted error bars, are not confidence intervals. A fixed
shuffle seed, $4821$, orders the methods and seeds within each
instance. Timed methods execute sequentially. A started stage always
executes all $K$ iterations, including iterations after its sampled
output index, and retains the returned auxiliary coordinate.
The configured limits are 48 stages or 100,000 baseline iterations
and a soft 120-second solver-time limit. The recorded status
distinguishes target attainment from a resource limit. All 180 recorded
runs in the initial ten-instance suite reached $e(z)\le10^{-6}$.

Measurements use a single CPU thread on an Intel Core i5-13420H,
Python 3.14 and NumPy 2.5.3. Dataset parsing, construction of the
instances, reference solves, and diagnostic evaluations are excluded
from solver time. Normalization, component setup, initial evaluations,
stage reinitialization, sampling, arithmetic, and cache maintenance
are included. Runtime describes this NumPy implementation on these
small instances; it is not an arithmetic-complexity bound.
The recorded cache-byte statistic counts selected retained numerical
state arrays. It excludes dataset storage, Python objects, component
index arrays, pre-sampled indices, parameters, and transient
arithmetic, and is therefore not peak resident memory.

Implementation checks compare the compact and expanded updates,
partial and fused returns, and numerical derivatives against the
supplied certificates. The global bounds below establish theorem
applicability.

\subsection{Methods and accounting}\label{supp:sec:experiment-methods}
The \emph{sampled} method implements Algorithm~\ref{alg:epoch} with the
default allocation~\eqref{supp:eq:dense-weights} and unchanged certified
parameters. The \emph{assembly} method uses the construction of
Section~\ref{supp:sec:assembly-guarantee}. Both use the prescribed
normalized oracle wrappers; assembly uses fused matrix operations where
possible and charges each returned coupling partial once.
Every benchmark coupling row has nonempty support and variation
coefficient above one, so the positive sampling floor is inactive:
the unified rules preserve their sampling probabilities, stage lengths,
and updates.

\emph{Extragradient} uses two evaluations of the full saddle operator
per iteration, with normalized step
\[
 \gamma=\frac1{2(\kappa_{\max}+L_H)}.
\]
Each full objective-gradient or coupling-gradient evaluation charges
every returned partial once. Its fused implementation is included in the runtime
comparison. The bound is global and certified; it is not estimated
from the reference or optimized through a parameter search in the primary
suite. A separate, bounded step-factor sensitivity study is reported in
Section~\ref{supp:sec:experiment-sensitivity}.

\emph{Primal AGD} is available only for the bilinear traffic and graph
models. It analytically maximizes over $y$ using the known quadratic
objective $g$. In normalized coordinates the dual maximizer is
$\widehat y(\widehat x)=-\widehat T_{H,y}(\widehat x,0)$,
and the resulting primal objective is $1$-strongly convex with valid
smoothness bound
$L_{\mathrm{pr}}=\max_i\kappa_i^f+L_H^2$.
We use step $1/L_{\mathrm{pr}}$ and momentum
$(\sqrt{L_{\mathrm{pr}}}-1)/(\sqrt{L_{\mathrm{pr}}}+1)$.
Dual reconstruction calls at the output checkpoints are counted.
The implementation's fused objective-gradient call returns all objective
partial gradients,
including its unused dual gradient, and charges them all.
This baseline exploits known analytic structure beyond the supplied
general partial-gradient model. The nonlinear interaction cannot
be eliminated by this linear dual formula, even though its separate
$g$ is quadratic, so we do not run this baseline there.

Each requested partial computes only its output block. Counters include
initialization, updates, corrections, and baseline output reconstruction;
reference and diagnostic evaluations are excluded. Records retain
per-block counts and family sums. Prescribed sampling ratios alone do
not demonstrate a convergence benefit.

\subsection{Benchmark models}\label{supp:sec:experiment-models}
The six base instances use the three model families below. The four
additional high-curvature instances are defined in
Section~\ref{supp:sec:experiment-sweep-results}.

\subsubsection{Traffic-demand corrections from simulated link counts}
Origin--destination estimation from link counts is an established
inverse problem \citep{ODPrimalDual}. We use the public Sioux Falls
and Anaheim network and trip files from
\citet{TransportationNetworks}. Our task estimates signed corrections
to a prior demand under fixed routing. It does not solve congestion
equilibrium or the constrained joint trip-distribution/assignment
problem.

For each network, four origin IDs are selected at equally spaced
positions in the sorted list of origins. For each selected origin,
we retain its twelve largest positive off-diagonal demands, breaking
ties by destination ID. The resulting $x\in\RR^{48}$ has four
blocks of dimension twelve. Each demand is assigned to a single
free-flow shortest route, respecting the network's
\texttt{FIRST THRU NODE} restriction. Let $A$ be the link--route
incidence matrix. Using seed $20261005$, we choose 32 distinct
active links as sensors and put the remaining active links in a
held-out set. The dual vector has four blocks of dimension eight.
Let $A_{\mathrm{obs}}$ be the observed-link rows and set
\[
 C=\frac{1.5}{\norm{A_{\mathrm{obs}}}_\op}A_{\mathrm{obs}}.
\]

The supplied demands $d^{\mathrm{true}}$ serve as simulation truth.
The prior is $d^{\mathrm{prior}}_a=d^{\mathrm{true}}_a\exp(0.25\xi_a)$,
with independent standard normal $\xi_a$, and the unit of a correction
is $s=\operatorname{median}_a d^{\mathrm{true}}_a$. Hence
$x^{\mathrm{true}}=(d^{\mathrm{true}}-d^{\mathrm{prior}})/s$.
We simulate the residual counts
\[
 b=Cx^{\mathrm{true}}+\epsilon,\qquad
 \epsilon\sim\mathcal N(0,\sigma_b^2I),\qquad
 \sigma_b=0.02\max\{\norm{Cx^{\mathrm{true}}}/\sqrt{32},0.01\}.
\]
Thus the noise level refers to the modeled link-count corrections,
rather than the total network flows. Unselected demands are fixed at
the supplied values; their background flows cancel when the prior
counts are subtracted. All sensor observations are simulated.

The saddle model is
\begin{equation}\label{supp:eq:traffic-experiment}
 \min_x\max_y
 \left\{\tfrac12x^\top Qx+y^\top(Cx-b)
                    -\tfrac{\nu}{2}\norm{y}^2\right\},
 \qquad \nu=0.4.
\end{equation}
Write $v=\boldsymbol 1_{48}/\sqrt{48}$ and
$Q=D+0.04vv^\top$. On primal block $i$, the twelve diagonal entries
of $D$ are geometrically spaced between $0.08$ and
$0.08\bar\kappa_i$, with
$\bar\kappa=(1,2,8,20)$. The rank-one term makes $f$ nonseparable.
Since $Q\succeq0.08I$, the primal lower certificate is
$\mu_i^f=0.08$. For a perturbation $u=(u_1,\ldots,u_4)$,
\[
 (v^\top u)^2
 \le4\sum_i\norm{v_i}^2\norm{u_i}^2
 =\sum_i\norm{u_i}^2.
\]
Consequently the joint diagonal upper certificate is
$L_i^f=\max_{a\in i}D_{aa}+0.04$, rather than a bound obtained by varying one block at a time.
The dual certificates are $\mu_j^g=L_j^g=0.4$.
For the submatrix $C_{ji}$ mapping primal block $i$ to dual block $j$,
the only nonzero coupling bounds are
\[
 L_{ij}^{H,xy}=L_{ji}^{H,yx}=\norm{C_{ji}}_\op.
\]
These are global first-order bounds because the interaction is
bilinear. Both traffic instances have $\Delta=4$ for each row.

Eliminating $y$ gives the regularized estimation objective
$x^\top Qx/2+\norm{Cx-b}^2/(2\nu)$.
We report relative demand error
$\norm{d^{\mathrm{prior}}+sx-d^{\mathrm{true}}}/\norm{d^{\mathrm{true}}}$,
the fraction of negative recovered demands, observed residual error
$\norm{Cx-b}/\norm{b}$, and held-out correction error
$\norm{A_{\mathrm{hold}}(x-x^{\mathrm{true}})}/\norm{A_{\mathrm{hold}}x^{\mathrm{true}}}$,
where $A_{\mathrm{hold}}$ contains the held-out incidence rows.
The recovered demand is $d^{\mathrm{prior}}+sx$; positivity is not
enforced. The 48 unknown corrections and 32 sensor observations also
make recovery underdetermined, so optimization accuracy alone cannot
establish demand-recovery accuracy.

\subsubsection{Graph-regularized logistic learning}
Graph regularization links related learning tasks
\citep{MultitaskGraphs}. We construct four task partitions from
LIBSVM's a9a version of the Adult dataset
\citep{LIBSVMData,AdultData}. Among the 123 supplied features, we
select the fifteen most frequent nonzero features in the full
training file, with stable ties by feature index, append an intercept,
and normalize each resulting row by the maximum of its Euclidean
norm and one. The held-out a9a test file uses exactly the same feature
selection and transformation. A permutation with seed $314159$
provides disjoint local training subsets of sizes
$128,256,512,1024$; only these 1,920 records train the models.
These are artificial random partitions, not naturally federated
datasets. The metadata retain the selected feature IDs and training
row IDs.

In both training and test files, each signed class label
$y_{\mathrm{label}}\in\{-1,+1\}$ is mapped to
$c=(y_{\mathrm{label}}+1)/2\in\{0,1\}$.
For $\ell(s;c)=\log(1+e^s)-cs$, let $X_i,c_i$ denote the local data,
$n_i$ its sample size, and define
\[
 f(x)=\sum_{i=1}^4\left\{
  \frac{0.2}{2}\norm{x_i}^2
  +\frac{w_i}{n_i}\sum_{r=1}^{n_i}\ell((X_ix_i)_r;c_{ir})\right\}.
\]
The four primal blocks each have dimension sixteen. The weights
$w_i=0.2(\kappa_i^f-1)4n_i/\norm{X_i}_\op^2$ make the certified
condition numbers $(\kappa_i^f)_{i=1}^4=(2,4,8,20)$. They introduce controlled
curvature heterogeneity, as well as the different local evaluation
costs due to the sample sizes. Since $0\le\ell''\le1/4$, the joint
objective curvature certificate is
\begin{equation}\label{supp:eq:logistic-player-certificate}
 \mu_i^f=0.2,\qquad
 L_i^f=0.2+\frac{w_i\norm{X_i}_\op^2}{4n_i}.
\end{equation}
The objective is separable across tasks, so these bounds form a valid
joint block-diagonal model.

The cycle has edges $(1,2),(2,3),(3,4),(4,1)$; the hub has edges
$(1,2),(1,3),(1,4)$. Edge weights $a_e$ are geometrically spaced
between $0.25$ and $4$ in the displayed edge order. Define
$(Cx)_e=\sqrt{a_e}(x_i-x_j)$ and use
\[
 H(x,y)=y^\top Cx,\qquad g(y)=\frac{0.5}{2}\norm{y}^2.
\]
Each edge is a sixteen-dimensional dual block. This is the exact
saddle representation of
$f(x)+\sum_ea_e\norm{x_i-x_j}^2/(2\cdot0.5)$.
The dual certificates are $\mu_e^g=L_e^g=0.5$; coupling bounds are
the block norms of $C$, as in the traffic model. Cycle rows have
$\Delta=2$. Hub primal rows have $\Delta=2$, while its dual rows
have $\Delta=3$.

Application metrics use the mean of the four test scores and report
its binary accuracy and logistic loss on the supplied test file.
These metrics characterize the chosen optimization objective;
hyperparameters and feature selection were not tuned for predictive
performance.

\subsubsection{Nonlinear group score-shift interaction}
This construction uses the same local logistic objectives and test
transformation, with a genuinely nonlinear coupling on real feature
rows and labels. It is a benchmark of group score shifts, rather than
a claim that an arbitrary robust-learning model satisfies our
assumptions.
For each task, let $\widetilde X_i$ be its first 128 training rows,
multiplied by $a_i$, where $(a_1,a_2,a_3,a_4)=(1,1.5,2,3)$;
retain the corresponding labels
$\widetilde c_i$. Thus each coupling task uses 128 records and
$N=512$ is the common normalization. Objective partials still use all
their local records.

In the local case, $y\in\RR^4$ has scalar blocks and $B_i$ has only
column $i$ nonzero, equal to the constant
$b_i$ in each row, with $(b_1,b_2,b_3,b_4)=(1,1.5,2,3)$. In the shared case,
$y\in\RR^5$ adds a fifth column of ones to every $B_i$.
The interaction and dual objective are
\begin{align}\label{supp:eq:nonlinear-experiment}
 H(x,y)&=\frac1N\sum_{i=1}^4\sum_{r=1}^{128}
  \ell((\widetilde X_ix_i+B_iy)_r;\widetilde c_{ir})
       -\frac1{8N}\sum_{i=1}^4\norm{B_iy}^2,\nonumber\\
 g(y)&=\frac{0.2}{2}\norm{y}^2.
\end{align}
The interaction's quadratic correction certifies concavity globally;
the separate $g$ supplies positive objective curvature. To verify this,
let $W_i=\diag(\ell''(\widetilde X_ix_i+B_iy;
\widetilde c_i))$, so $0\preceq W_i\preceq I/4$. The Hessian blocks
are
\begin{align}
 \nabla^2_{x_ix_i}H&=\widetilde X_i^\top W_i\widetilde X_i/N,
 &\nabla^2_{x_ix_t}H&=0\quad(i\ne t),\nonumber\\
 \nabla^2_{x_iy}H&=\widetilde X_i^\top W_iB_i/N,
 &\nabla^2_{yy}H&=\sum_iB_i^\top(W_i-I/4)B_i/N.
 \label{supp:eq:nonlinear-hessian}
\end{align}
Hence $H$ is convex in the entire $x$ vector and concave in the
entire $y$ vector. The mixed Hessian depends on the scores, so
the coupling is nonlinear. All its Hessian blocks are globally
bounded. Writing $A_i=\norm{\widetilde X_i}_\op$ and
$\beta_{ij}=\norm{(B_i)_{:,j}}$, valid original-coordinate arrays are
\begin{align}
 L_{it}^{H,xx}&=\boldsymbol 1_{\{i=t\}}\frac{A_i^2}{4N},
 &L_{ij}^{H,xy}=L_{ji}^{H,yx}&=\frac{A_i\beta_{ij}}{4N},\nonumber\\
 L_{ju}^{H,yy}&=\frac1{4N}\sum_i\beta_{ij}\beta_{iu}.
 \label{supp:eq:nonlinear-experiment-bounds}
\end{align}
The last bound uses $\norm{W_i-I/4}_\op\le1/4$.
Integrating the Hessian along a line segment gives the required
global first-order inequalities. Zero coefficients certify absent
dependence. The bounds are normalized by both input and output
curvature exactly as in Section~\ref{sec:model}; supports and
multiplicities come from their zero pattern. The resulting overlap
is $\Delta=2$ in the local case and $\Delta=9$ in the shared case.
The shared construction changes dimension and coupling variation
as well as overlap; it is not an isolated intervention on $\Delta$.

\subsection{Application diagnostics}\label{supp:sec:experiment-applications}
Table~\ref{tab:exp-applications} reports the application metrics at the
independent reference solution, with prior traffic errors at the initial
point. Each entry characterizes its constructed objective and uses the
application metrics defined in Section~\ref{supp:sec:experiment-models}.

\begin{table}[ht]\centering\small\setlength{\tabcolsep}{6pt}
\caption{Application diagnostics. Except for prior demand error, values are evaluated once at the independent reference solution. Traffic observations are simulated; traffic errors are relative Euclidean norms, and negative demands are reported as a fraction of the estimated demands. Classification uses the mean of the four test scores.}
\label{tab:exp-applications}
\begin{tabular}{@{}lrrrrr@{}}\toprule
Traffic case & \shortstack{Prior demand\\error} & \shortstack{Fitted demand\\error} & \shortstack{Negative\\fraction} & \shortstack{Observed\\residual} & \shortstack{Held-out\\error}\\\midrule
Sioux Falls & 0.1710 & 0.1512 & $0/48$ & 0.2310 & 0.7899\\
Anaheim & 0.1680 & 0.1168 & $4/48$ & 0.2158 & 0.3210\\
\bottomrule\end{tabular}
\par\medskip
\begin{tabular}{@{}lrr@{}}\toprule
Classification case & Test accuracy & Test log loss\\\midrule
Graph cycle & 0.7693 & 0.4720\\
Graph hub & 0.7693 & 0.4726\\
Nonlinear local & 0.7650 & 0.4847\\
Nonlinear shared & 0.7650 & 0.4847\\
Graph cycle, $200$ & 0.8168 & 0.4043\\
Graph cycle, $2000$ & 0.8298 & 0.3774\\
Nonlinear local, $200$ & 0.7900 & 0.4516\\
Nonlinear local, $2000$ & 0.7980 & 0.4344\\
\bottomrule\end{tabular}
\end{table}

\FloatBarrier

\subsection{Base solver results}\label{supp:sec:experiment-base-results}
Table~\ref{tab:exp-base} reports all six base instances at the common
distance target~\eqref{supp:eq:experiment-error}.

\begin{table}[ht]
\centering\small\setlength{\tabcolsep}{4pt}
\caption{Measured work to observed relative squared distance $10^{-6}$: base configurations. Entries are means $\pm$ sample standard deviations over five runs. Oracle totals are in thousands; time is in seconds. EG and primal AGD have deterministic counts and repeated timing measurements.}
\label{tab:exp-base}
\begin{tabular}{@{}llrrrrr@{}}\toprule
Case & Method & $N^f/10^3$ & $N^g/10^3$ & $N_H^x/10^3$ & $N_H^y/10^3$ & Seconds\\\midrule
Sioux Falls & Sampled & $4.38\,\pm\,0.89$ & $1.32\,\pm\,0.28$ & $38.07\,\pm\,8.08$ & $38.21\,\pm\,8.01$ & $0.769\,\pm\,0.164$\\
 & Assembly & $5.07\,\pm\,0.58$ & $1.50\,\pm\,0.17$ & $37.62\,\pm\,4.46$ & $37.62\,\pm\,4.46$ & $0.222\,\pm\,0.025$\\
 & EG & $3.12\,\pm\,0.00$ & $3.12\,\pm\,0.00$ & $3.12\,\pm\,0.00$ & $3.12\,\pm\,0.00$ & $0.0106\,\pm\,0.0001$\\
 & Primal AGD & $0.32\,\pm\,0.00$ & $0.32\,\pm\,0.00$ & $0.32\,\pm\,0.00$ & $0.35\,\pm\,0.00$ & $0.0030\,\pm\,0.0001$\\
\addlinespace[3pt]
Anaheim & Sampled & $4.93\,\pm\,1.02$ & $1.48\,\pm\,0.32$ & $43.46\,\pm\,8.75$ & $43.52\,\pm\,8.79$ & $0.878\,\pm\,0.172$\\
 & Assembly & $4.38\,\pm\,0.79$ & $1.30\,\pm\,0.24$ & $36.79\,\pm\,6.41$ & $36.79\,\pm\,6.41$ & $0.207\,\pm\,0.036$\\
 & EG & $3.44\,\pm\,0.00$ & $3.44\,\pm\,0.00$ & $3.44\,\pm\,0.00$ & $3.44\,\pm\,0.00$ & $0.0117\,\pm\,0.0004$\\
 & Primal AGD & $0.40\,\pm\,0.00$ & $0.40\,\pm\,0.00$ & $0.40\,\pm\,0.00$ & $0.44\,\pm\,0.00$ & $0.0037\,\pm\,0.0001$\\
\addlinespace[3pt]
Graph cycle & Sampled & $5.43\,\pm\,1.12$ & $1.51\,\pm\,0.31$ & $22.37\,\pm\,4.61$ & $21.17\,\pm\,4.35$ & $0.624\,\pm\,0.126$\\
 & Assembly & $5.21\,\pm\,0.93$ & $1.46\,\pm\,0.23$ & $27.34\,\pm\,4.85$ & $27.34\,\pm\,4.85$ & $0.311\,\pm\,0.058$\\
 & EG & $2.48\,\pm\,0.00$ & $2.48\,\pm\,0.00$ & $2.48\,\pm\,0.00$ & $2.48\,\pm\,0.00$ & $0.052\,\pm\,0.001$\\
 & Primal AGD & $0.28\,\pm\,0.00$ & $0.28\,\pm\,0.00$ & $0.28\,\pm\,0.00$ & $0.31\,\pm\,0.00$ & $0.0077\,\pm\,0.0003$\\
\addlinespace[3pt]
Graph hub & Sampled & $5.42\,\pm\,0.99$ & $1.14\,\pm\,0.20$ & $18.90\,\pm\,3.41$ & $19.93\,\pm\,3.63$ & $0.567\,\pm\,0.104$\\
 & Assembly & $6.83\,\pm\,1.32$ & $1.44\,\pm\,0.27$ & $35.06\,\pm\,6.87$ & $26.29\,\pm\,5.15$ & $0.389\,\pm\,0.080$\\
 & EG & $2.48\,\pm\,0.00$ & $1.86\,\pm\,0.00$ & $2.48\,\pm\,0.00$ & $1.86\,\pm\,0.00$ & $0.051\,\pm\,0.001$\\
 & Primal AGD & $0.28\,\pm\,0.00$ & $0.21\,\pm\,0.00$ & $0.28\,\pm\,0.00$ & $0.23\,\pm\,0.00$ & $0.00717\,\pm\,0.00004$\\
\addlinespace[3pt]
Nonlinear local & Sampled & $4.62\,\pm\,0.35$ & $1.28\,\pm\,0.12$ & $4.96\,\pm\,0.34$ & $5.92\,\pm\,0.42$ & $0.361\,\pm\,0.027$\\
 & Assembly & $4.45\,\pm\,0.90$ & $1.23\,\pm\,0.27$ & $11.78\,\pm\,2.39$ & $11.78\,\pm\,2.39$ & $0.344\,\pm\,0.070$\\
 & EG & $2.00\,\pm\,0.00$ & $2.00\,\pm\,0.00$ & $2.00\,\pm\,0.00$ & $2.00\,\pm\,0.00$ & $0.064\,\pm\,0.001$\\
\addlinespace[3pt]
Nonlinear shared & Sampled & $5.11\,\pm\,0.97$ & $1.75\,\pm\,0.33$ & $12.24\,\pm\,2.35$ & $19.01\,\pm\,3.66$ & $0.856\,\pm\,0.166$\\
 & Assembly & $4.62\,\pm\,1.35$ & $1.58\,\pm\,0.45$ & $13.19\,\pm\,3.76$ & $16.49\,\pm\,4.70$ & $0.398\,\pm\,0.129$\\
 & EG & $2.08\,\pm\,0.00$ & $2.60\,\pm\,0.00$ & $2.08\,\pm\,0.00$ & $2.60\,\pm\,0.00$ & $0.067\,\pm\,0.001$\\
\bottomrule\end{tabular}
\end{table}

Table~\ref{tab:exp-base} shows that extragradient has lower runtime
and fewer primal-objective calls than either stage method on all six
base instances. Sampling reduces coupling calls relative to assembly
on the graph instances and the local nonlinear construction, but
does not consistently reduce runtime. Traffic and shared nonlinear
cases lack the same coupling-call advantage. Primal AGD is fastest
on the bilinear base instances.

\FloatBarrier

\subsection{Exploratory loss-weight sweep}\label{supp:sec:experiment-sweep-results}
For the graph-cycle and local nonlinear constructions, we additionally
replace the fourth certified condition number, $20$, by $200$ and
$2000$. The data, row subsets, ridge penalties, and coupling remain
fixed; only the fourth local loss weight changes
according to~\eqref{supp:eq:logistic-player-certificate}.
Parameter-selection rules and numerical multipliers remain fixed;
derived parameters are recomputed for each instance. These include
$S,p,\eta,K$, the extragradient step, and the primal AGD smoothness
bound, step, and momentum.
These four extra instances probe the cost of evaluating easy objective
blocks when another block is much stiffer. They are exploratory
stress regimes, motivated by the initially unfavorable comparisons
on the six base instances. We retain the base results, and do not
tune method parameters for the added cases. Increasing a loss weight
also changes the minimizer, the relative task emphasis, and the
denominator in~\eqref{supp:eq:experiment-error}. For graph-cycle, the
absolute squared-distance cutoff increases from $3.39\times10^{-6}$
to $1.81\times10^{-5}$ and $6.35\times10^{-5}$, approximately nineteenfold
between the endpoints. For nonlinear-local it increases from
$3.41\times10^{-6}$ to $1.19\times10^{-5}$ and $2.74\times10^{-5}$.
The crossover in this sweep alone cannot be attributed solely to
conditioning. The fixed-saddle control below keeps this denominator
constant; neither finite sweep establishes an asymptotic rate.

\begin{table}[H]
\centering\small\setlength{\tabcolsep}{4pt}
\caption{Measured work to observed relative squared distance $10^{-6}$: objective loss-weight sweep. Entries are means $\pm$ sample standard deviations over five runs. Oracle totals are in thousands; time is in seconds. EG and primal AGD have deterministic counts and repeated timing measurements.}
\label{tab:exp-stress}
\begin{tabular}{@{}llrrrrr@{}}\toprule
Case & Method & $N^f/10^3$ & $N^g/10^3$ & $N_H^x/10^3$ & $N_H^y/10^3$ & Seconds\\\midrule
Graph cycle, $200$ & Sampled & $11.30\,\pm\,2.41$ & $1.59\,\pm\,0.35$ & $24.29\,\pm\,5.11$ & $22.96\,\pm\,4.81$ & $0.866\,\pm\,0.186$\\
 & Assembly & $10.96\,\pm\,3.18$ & $1.59\,\pm\,0.44$ & $29.69\,\pm\,8.55$ & $29.69\,\pm\,8.55$ & $0.533\,\pm\,0.156$\\
 & EG & $16.32\,\pm\,0.00$ & $16.32\,\pm\,0.00$ & $16.32\,\pm\,0.00$ & $16.32\,\pm\,0.00$ & $0.313\,\pm\,0.001$\\
 & Primal AGD & $0.44\,\pm\,0.00$ & $0.44\,\pm\,0.00$ & $0.44\,\pm\,0.00$ & $0.48\,\pm\,0.00$ & $0.0113\,\pm\,0.0001$\\
\addlinespace[3pt]
Graph cycle, $2000$ & Sampled & $28.40\,\pm\,6.99$ & $1.62\,\pm\,0.40$ & $24.40\,\pm\,6.05$ & $22.91\,\pm\,5.64$ & $1.514\,\pm\,0.363$\\
 & Assembly & $25.50\,\pm\,6.33$ & $1.45\,\pm\,0.34$ & $27.42\,\pm\,6.77$ & $27.42\,\pm\,6.77$ & $1.072\,\pm\,0.274$\\
 & EG & $171.44\,\pm\,0.00$ & $171.44\,\pm\,0.00$ & $171.44\,\pm\,0.00$ & $171.44\,\pm\,0.00$ & $3.171\,\pm\,0.013$\\
 & Primal AGD & $1.28\,\pm\,0.00$ & $1.28\,\pm\,0.00$ & $1.28\,\pm\,0.00$ & $1.41\,\pm\,0.00$ & $0.0314\,\pm\,0.0003$\\
\addlinespace[3pt]
Nonlinear local, $200$ & Sampled & $10.36\,\pm\,1.87$ & $1.51\,\pm\,0.28$ & $5.78\,\pm\,1.09$ & $6.87\,\pm\,1.21$ & $0.612\,\pm\,0.106$\\
 & Assembly & $11.25\,\pm\,3.13$ & $1.60\,\pm\,0.43$ & $15.27\,\pm\,4.28$ & $15.27\,\pm\,4.28$ & $0.657\,\pm\,0.190$\\
 & EG & $14.32\,\pm\,0.00$ & $14.32\,\pm\,0.00$ & $14.32\,\pm\,0.00$ & $14.32\,\pm\,0.00$ & $0.441\,\pm\,0.004$\\
\addlinespace[3pt]
Nonlinear local, $2000$ & Sampled & $27.66\,\pm\,3.23$ & $1.60\,\pm\,0.21$ & $6.08\,\pm\,0.71$ & $7.28\,\pm\,0.88$ & $1.281\,\pm\,0.145$\\
 & Assembly & $29.16\,\pm\,4.48$ & $1.65\,\pm\,0.22$ & $15.60\,\pm\,2.47$ & $15.60\,\pm\,2.47$ & $1.342\,\pm\,0.208$\\
 & EG & $124.40\,\pm\,0.00$ & $124.40\,\pm\,0.00$ & $124.40\,\pm\,0.00$ & $124.40\,\pm\,0.00$ & $3.793\,\pm\,0.013$\\
\bottomrule\end{tabular}
\end{table}

The controlled sweep in Table~\ref{tab:exp-stress} changes this
comparison: at condition bound $2000$, both stage methods use fewer
primal-objective calls and less runtime than fixed-step extragradient.
Sampling is not uniformly better than assembly, and primal AGD remains
much faster on the bilinear graph sweep. For the nonlinear $200$ and
$2000$ cases, sampled/assembly mean time differences are small relative
to run standard deviations and establish no reliable ordering.
These comparisons use certified parameters without tuning and do not
establish asymptotic oracle orders; deterministic timing repetitions
are not independent problem draws.

Table~\ref{tab:exp-block-counts} separates the sampled method's four
primal-objective counters. On the graph sweep, mean block-1 counts are
approximately $660,704,679$, while block-4 counts grow to
$2340,7954,25108$. The fourth block absorbs most of the additional
objective calls as its curvature increases. This illustrates the
individual-oracle accounting at the observed stopping points; the
finite sweep does not establish the theorem's asymptotic rates.

\begin{table}[H]\centering\small\setlength{\tabcolsep}{6pt}
\caption{Sampled primal-objective calls across the curvature sweep. Entries are means $\pm$ sample SD over five seeds, rounded to whole calls, at the first recorded point satisfying $e(z)\le10^{-6}$. Each $N_i^f$ counts calls to $\nabla_{x_i}f$, including initialization and all completed-stage work. The first three condition numbers remain $(2,4,8)$.}
\label{tab:exp-block-counts}
\begin{tabular}{@{}lrrrrr@{}}\toprule
Case & $\kappa_4^f$ & $N_1^f$ & $N_2^f$ & $N_3^f$ & $N_4^f$\\\midrule
Graph cycle & 20 & $660\,\pm\,157$ & $961\,\pm\,183$ & $1471\,\pm\,307$ & $2340\,\pm\,481$\\
Graph cycle & 200 & $704\,\pm\,168$ & $1046\,\pm\,224$ & $1594\,\pm\,351$ & $7954\,\pm\,1665$\\
Graph cycle & 2000 & $679\,\pm\,175$ & $1067\,\pm\,255$ & $1548\,\pm\,407$ & $25108\,\pm\,6155$\\
\addlinespace[3pt]
Nonlinear local & 20 & $558\,\pm\,30$ & $879\,\pm\,66$ & $1228\,\pm\,111$ & $1959\,\pm\,152$\\
Nonlinear local & 200 & $645\,\pm\,133$ & $995\,\pm\,195$ & $1434\,\pm\,218$ & $7287\,\pm\,1325$\\
Nonlinear local & 2000 & $678\,\pm\,72$ & $1032\,\pm\,108$ & $1541\,\pm\,166$ & $24413\,\pm\,2897$\\
\bottomrule\end{tabular}
\end{table}

\FloatBarrier

\begin{figure}[H]
\centering
\includegraphics[width=\textwidth]{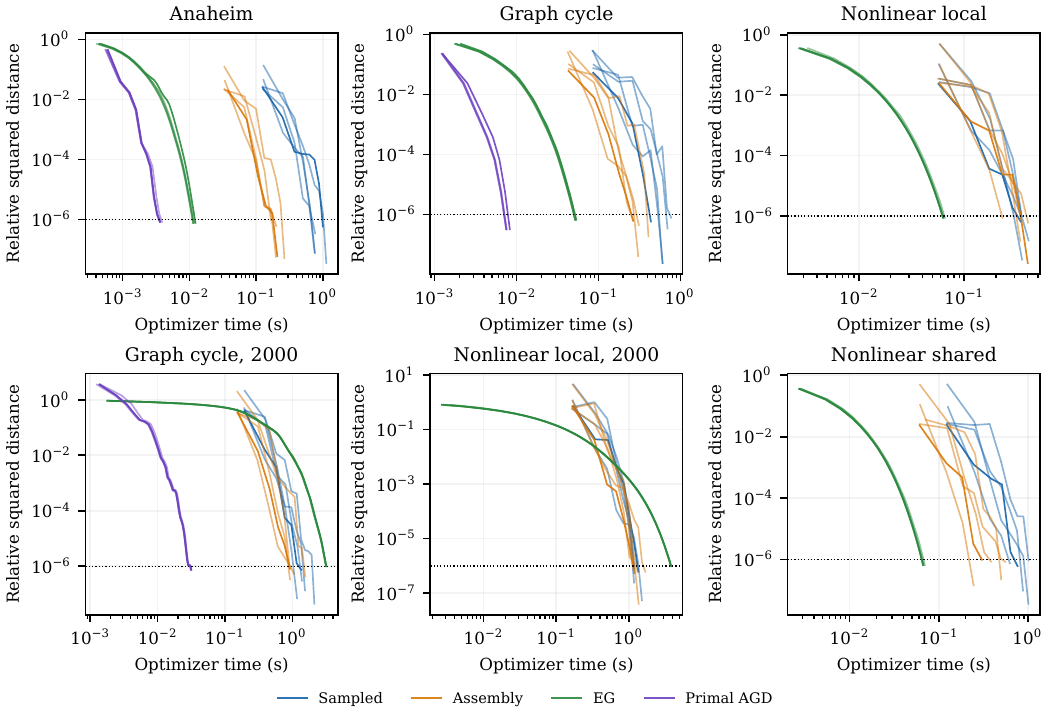}
\caption{Individual convergence trajectories for six representative
instances against measured solver time. The common error is defined
in~\eqref{supp:eq:experiment-error}. Runs are recorded at complete
stage or ten-iteration checkpoints; the associated work is charged
in full. Deterministic methods repeat only for timing variability.}
\label{supp:fig:experiment-convergence}
\end{figure}

\begin{figure}[H]
\centering
\includegraphics[width=\textwidth]{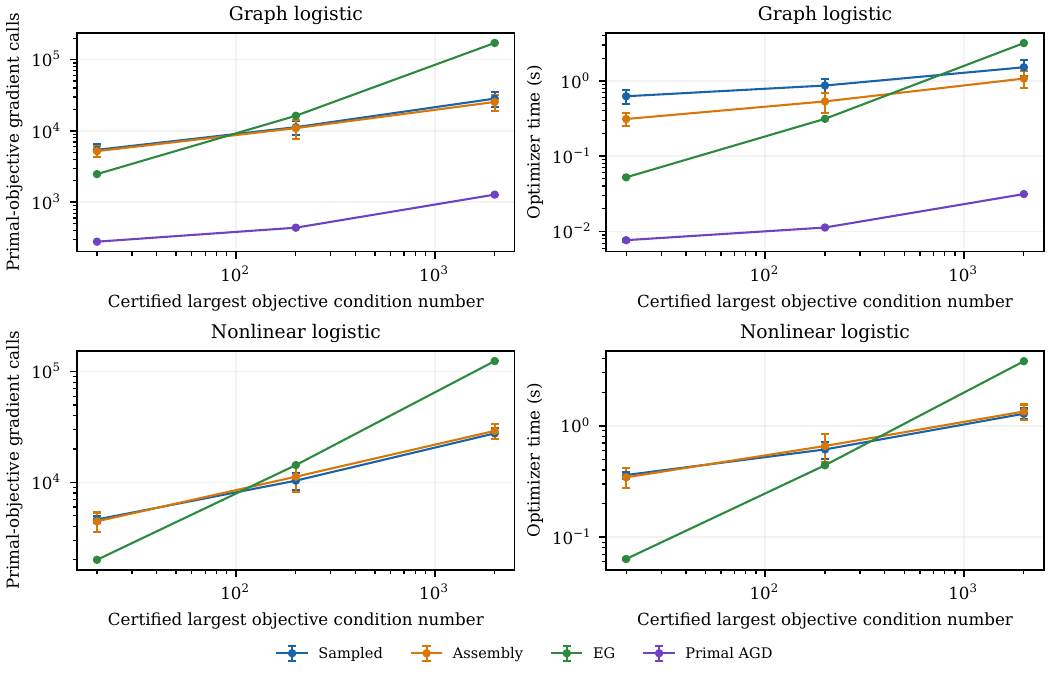}
\caption{Loss-weight sweep over the largest objective condition bound: primal-objective
partial-gradient calls and solver time at the common recorded
accuracy target. Points show means and error bars show sample
standard deviations over five runs. Counters include all completed-stage work.}
\label{supp:fig:experiment-oracles}
\end{figure}

\FloatBarrier

\subsection{Separate EG step-factor sensitivity}\label{supp:sec:experiment-sensitivity}
On graph-cycle-$2000$ and nonlinear-local-$2000$, a predeclared follow-up
compares $\gamma=a/(\kappa_{\max}+L_H)$ for $a=0.5$ and $0.9$, with
five deterministic repetitions per variant and the same reference,
accuracy target, and checkpoints described in
Section~\ref{supp:sec:experiment-protocol}.

\begin{table}[!htbp]
\centering\small\setlength{\tabcolsep}{4pt}
\caption{Separate EG step-factor sensitivity at $\kappa_{\max}=2000$. Both factors use the same certified bound and five sequential timing repetitions in one local environment. Oracle counts are deterministic totals per run; times are means $\pm$ sample standard deviations in seconds. These timings are not combined with the historical sampled timings.}
\label{tab:exp-eg-sensitivity}
\begin{tabular}{@{}llrrrrrr@{}}\toprule
Case & Factor & Iterations & $N^f$ & $N^g$ & $N_H^x$ & $N_H^y$ & Seconds\\\midrule
Graph cycle & 0.5 & 21430 & 171440 & 171440 & 171440 & 171440 & $3.328\,\pm\,0.017$\\
Graph cycle & 0.9 & 11890 & 95120 & 95120 & 95120 & 95120 & $1.851\,\pm\,0.028$\\
Nonlinear local & 0.5 & 15550 & 124400 & 124400 & 124400 & 124400 & $3.923\,\pm\,0.065$\\
Nonlinear local & 0.9 & 8640 & 69120 & 69120 & 69120 & 69120 & $2.177\,\pm\,0.025$\\
\bottomrule\end{tabular}
\end{table}

The $0.5$ counts reproduce the initial EG counts; $0.9$ reduces
iterations and every oracle family by about $44.5\%$ on both cases and
improves EG runtime, demonstrating sensitivity of the reported margin.
The follow-up uses Python 3.14.7 and NumPy 2.4.6 on the same CPU model
with one BLAS thread; the primary suite records NumPy 2.5.3. Its EG
timings are therefore compared only within this study, with no new
crossover claimed against initial-suite sampled times. That study did not
rerun sampled or assembly. The subsequent joint rerun below evaluates
all methods in one environment and supplies the main comparison.

\FloatBarrier

\subsection{Joint rerun and mechanism controls}\label{supp:sec:experiment-controls}
The joint follow-up reruns all ten data-backed cases with sampled,
assembly, both EG factors $0.5$ and $0.9$, and applicable primal AGD. It adds six
affine-recentered cases, three overlap cases, three oracle-cost cases,
and a full-evaluation ablation on the two original $2000$ endpoints.
All cases and jobs are specified before timing. Each included method has five
runs, for 435 timed runs on 22 conditions. Solvers execute sequentially
in one globally shuffled order (seed $4821$), with the same $10^{-6}$
target, zero initialization, complete-stage and ten-iteration checkpoints,
and resource limits as the initial suite. The environment is
Python 3.14.7, NumPy 2.4.6, the same CPU model, and one BLAS thread.
Before generating tables, automated checks validate successful
endpoints and the control invariants described below.
These ratios use only timings from the joint rerun.

Absolute times and counts are means $\pm$ sample SD across five runs.
Runtime ratios divide arithmetic means. Their uncertainty is the
standard error from 20,000 independent nonparametric bootstrap
resamples of the two methods' five timings (seed $4821$); seed labels
are not treated as timing pairs. These SEs describe uncertainty of a
ratio of means within this small experiment, not confidence intervals
or variability across problem draws. Deterministic repetitions assess
only timing variability. Tables retain terminal-error ranges because
different checkpoint schedules produce different overshoots; runtime
at first observed crossing is not a comparison at identical terminal
accuracy.

\paragraph{Joint comparison with stronger EG.}
Tables~\ref{tab:joint-base}--\ref{tab:joint-stress} give all ten cases.
Both EG factors use the same certified bound and fused implementation;
$0.9$ is the stronger baseline for the main table, selected after the
earlier sensitivity study. Stage parameters remain prescribed and
unchanged. The tables and figures above report the earlier $0.5$ comparison.

\begin{table}[!htbp]
\centering\footnotesize\setlength{\tabcolsep}{1.3pt}
\caption{All six base configurations rerun together with both EG factors. Calls are in thousands and times in seconds (mean $\pm$ sample SD). Terminal-error ranges retain checkpoint overshoot.}
\label{tab:joint-base}
\begin{tabular}{@{}llrrrrrl@{}}\toprule
Case & Method & $N^f$ & $N^g$ & $N_H^x$ & $N_H^y$ & Seconds & Final $e$\\\midrule
Sioux & Sampled & $4.38\,\pm\,0.89$ & $1.32\,\pm\,0.28$ & $38.07\,\pm\,8.08$ & $38.21\,\pm\,8.01$ & $0.764\,\pm\,0.153$ & $1.1e-07$--$8.7e-07$\\
 & Assembly & $5.07\,\pm\,0.58$ & $1.50\,\pm\,0.17$ & $37.62\,\pm\,4.46$ & $37.62\,\pm\,4.46$ & $0.226\,\pm\,0.027$ & $7.4e-08$--$7.2e-07$\\
 & EG 0.5 & $3.12\,\pm\,0.00$ & $3.12\,\pm\,0.00$ & $3.12\,\pm\,0.00$ & $3.12\,\pm\,0.00$ & $0.011\,\pm\,0.000$ & $7.5e-07$--$7.5e-07$\\
 & EG 0.9 & $1.68\,\pm\,0.00$ & $1.68\,\pm\,0.00$ & $1.68\,\pm\,0.00$ & $1.68\,\pm\,0.00$ & $0.006\,\pm\,0.000$ & $8.7e-07$--$8.7e-07$\\
 & Primal AGD & $0.32\,\pm\,0.00$ & $0.32\,\pm\,0.00$ & $0.32\,\pm\,0.00$ & $0.35\,\pm\,0.00$ & $0.003\,\pm\,0.000$ & $5.0e-07$--$5.0e-07$\\
Anaheim & Sampled & $4.93\,\pm\,1.02$ & $1.48\,\pm\,0.32$ & $43.46\,\pm\,8.75$ & $43.52\,\pm\,8.79$ & $0.865\,\pm\,0.171$ & $3.5e-08$--$5.8e-07$\\
 & Assembly & $4.38\,\pm\,0.79$ & $1.30\,\pm\,0.24$ & $36.79\,\pm\,6.41$ & $36.79\,\pm\,6.41$ & $0.207\,\pm\,0.036$ & $5.1e-08$--$9.5e-07$\\
 & EG 0.5 & $3.44\,\pm\,0.00$ & $3.44\,\pm\,0.00$ & $3.44\,\pm\,0.00$ & $3.44\,\pm\,0.00$ & $0.012\,\pm\,0.000$ & $7.8e-07$--$7.8e-07$\\
 & EG 0.9 & $1.92\,\pm\,0.00$ & $1.92\,\pm\,0.00$ & $1.92\,\pm\,0.00$ & $1.92\,\pm\,0.00$ & $0.007\,\pm\,0.000$ & $6.7e-07$--$6.7e-07$\\
 & Primal AGD & $0.40\,\pm\,0.00$ & $0.40\,\pm\,0.00$ & $0.40\,\pm\,0.00$ & $0.44\,\pm\,0.00$ & $0.004\,\pm\,0.000$ & $8.4e-07$--$8.4e-07$\\
Graph & Sampled & $5.43\,\pm\,1.12$ & $1.51\,\pm\,0.31$ & $22.37\,\pm\,4.61$ & $21.17\,\pm\,4.35$ & $0.618\,\pm\,0.121$ & $2.4e-08$--$7.9e-07$\\
 & Assembly & $5.21\,\pm\,0.93$ & $1.46\,\pm\,0.23$ & $27.34\,\pm\,4.85$ & $27.34\,\pm\,4.85$ & $0.305\,\pm\,0.056$ & $4.1e-08$--$6.4e-07$\\
 & EG 0.5 & $2.48\,\pm\,0.00$ & $2.48\,\pm\,0.00$ & $2.48\,\pm\,0.00$ & $2.48\,\pm\,0.00$ & $0.054\,\pm\,0.001$ & $6.8e-07$--$6.8e-07$\\
 & EG 0.9 & $1.36\,\pm\,0.00$ & $1.36\,\pm\,0.00$ & $1.36\,\pm\,0.00$ & $1.36\,\pm\,0.00$ & $0.029\,\pm\,0.000$ & $8.5e-07$--$8.5e-07$\\
 & Primal AGD & $0.28\,\pm\,0.00$ & $0.28\,\pm\,0.00$ & $0.28\,\pm\,0.00$ & $0.31\,\pm\,0.00$ & $0.008\,\pm\,0.000$ & $3.3e-07$--$3.3e-07$\\
Hub & Sampled & $5.42\,\pm\,0.99$ & $1.14\,\pm\,0.20$ & $18.90\,\pm\,3.41$ & $19.93\,\pm\,3.63$ & $0.568\,\pm\,0.106$ & $1.8e-08$--$8.3e-07$\\
 & Assembly & $6.83\,\pm\,1.32$ & $1.44\,\pm\,0.27$ & $35.06\,\pm\,6.87$ & $26.29\,\pm\,5.15$ & $0.386\,\pm\,0.074$ & $1.5e-08$--$2.1e-07$\\
 & EG 0.5 & $2.48\,\pm\,0.00$ & $1.86\,\pm\,0.00$ & $2.48\,\pm\,0.00$ & $1.86\,\pm\,0.00$ & $0.053\,\pm\,0.002$ & $7.5e-07$--$7.5e-07$\\
 & EG 0.9 & $1.36\,\pm\,0.00$ & $1.02\,\pm\,0.00$ & $1.36\,\pm\,0.00$ & $1.02\,\pm\,0.00$ & $0.029\,\pm\,0.001$ & $9.1e-07$--$9.1e-07$\\
 & Primal AGD & $0.28\,\pm\,0.00$ & $0.21\,\pm\,0.00$ & $0.28\,\pm\,0.00$ & $0.23\,\pm\,0.00$ & $0.007\,\pm\,0.000$ & $2.4e-07$--$2.4e-07$\\
NL local & Sampled & $4.62\,\pm\,0.35$ & $1.28\,\pm\,0.12$ & $4.96\,\pm\,0.34$ & $5.92\,\pm\,0.42$ & $0.360\,\pm\,0.028$ & $3.6e-08$--$9.9e-07$\\
 & Assembly & $4.45\,\pm\,0.90$ & $1.23\,\pm\,0.27$ & $11.78\,\pm\,2.39$ & $11.78\,\pm\,2.39$ & $0.352\,\pm\,0.072$ & $2.7e-08$--$8.7e-07$\\
 & EG 0.5 & $2.00\,\pm\,0.00$ & $2.00\,\pm\,0.00$ & $2.00\,\pm\,0.00$ & $2.00\,\pm\,0.00$ & $0.065\,\pm\,0.000$ & $8.8e-07$--$8.8e-07$\\
 & EG 0.9 & $1.12\,\pm\,0.00$ & $1.12\,\pm\,0.00$ & $1.12\,\pm\,0.00$ & $1.12\,\pm\,0.00$ & $0.037\,\pm\,0.002$ & $9.1e-07$--$9.1e-07$\\
NL shared & Sampled & $5.11\,\pm\,0.97$ & $1.75\,\pm\,0.33$ & $12.24\,\pm\,2.35$ & $19.01\,\pm\,3.66$ & $0.847\,\pm\,0.167$ & $3.7e-08$--$9.5e-07$\\
 & Assembly & $4.62\,\pm\,1.35$ & $1.58\,\pm\,0.45$ & $13.19\,\pm\,3.76$ & $16.49\,\pm\,4.70$ & $0.387\,\pm\,0.115$ & $1.5e-07$--$9.9e-07$\\
 & EG 0.5 & $2.08\,\pm\,0.00$ & $2.60\,\pm\,0.00$ & $2.08\,\pm\,0.00$ & $2.60\,\pm\,0.00$ & $0.069\,\pm\,0.001$ & $6.7e-07$--$6.7e-07$\\
 & EG 0.9 & $1.20\,\pm\,0.00$ & $1.50\,\pm\,0.00$ & $1.20\,\pm\,0.00$ & $1.50\,\pm\,0.00$ & $0.040\,\pm\,0.001$ & $4.8e-07$--$4.8e-07$\\
\bottomrule\end{tabular}
\end{table}

\FloatBarrier
\begin{table}[!htbp]
\centering\footnotesize\setlength{\tabcolsep}{1.3pt}
\caption{Joint loss-weight stress comparison with both EG factors. Calls are in thousands and times in seconds (mean $\pm$ sample SD). Terminal-error ranges retain checkpoint overshoot.}
\label{tab:joint-stress}
\begin{tabular}{@{}llrrrrrl@{}}\toprule
Case & Method & $N^f$ & $N^g$ & $N_H^x$ & $N_H^y$ & Seconds & Final $e$\\\midrule
Graph 200 & Sampled & $11.30\,\pm\,2.41$ & $1.59\,\pm\,0.35$ & $24.29\,\pm\,5.11$ & $22.96\,\pm\,4.81$ & $0.876\,\pm\,0.185$ & $7.2e-08$--$9.0e-07$\\
 & Assembly & $10.96\,\pm\,3.18$ & $1.59\,\pm\,0.44$ & $29.69\,\pm\,8.55$ & $29.69\,\pm\,8.55$ & $0.539\,\pm\,0.156$ & $5.7e-08$--$9.4e-07$\\
 & EG 0.5 & $16.32\,\pm\,0.00$ & $16.32\,\pm\,0.00$ & $16.32\,\pm\,0.00$ & $16.32\,\pm\,0.00$ & $0.326\,\pm\,0.006$ & $9.5e-07$--$9.5e-07$\\
 & EG 0.9 & $9.04\,\pm\,0.00$ & $9.04\,\pm\,0.00$ & $9.04\,\pm\,0.00$ & $9.04\,\pm\,0.00$ & $0.181\,\pm\,0.003$ & $9.1e-07$--$9.1e-07$\\
 & Primal AGD & $0.44\,\pm\,0.00$ & $0.44\,\pm\,0.00$ & $0.44\,\pm\,0.00$ & $0.48\,\pm\,0.00$ & $0.011\,\pm\,0.000$ & $8.6e-07$--$8.6e-07$\\
Graph 2000 & Sampled & $28.40\,\pm\,6.99$ & $1.62\,\pm\,0.40$ & $24.40\,\pm\,6.05$ & $22.91\,\pm\,5.64$ & $1.563\,\pm\,0.404$ & $4.4e-08$--$7.5e-07$\\
 & Assembly & $25.50\,\pm\,6.33$ & $1.45\,\pm\,0.34$ & $27.42\,\pm\,6.77$ & $27.42\,\pm\,6.77$ & $1.112\,\pm\,0.304$ & $3.4e-07$--$8.8e-07$\\
 & EG 0.5 & $171.44\,\pm\,0.00$ & $171.44\,\pm\,0.00$ & $171.44\,\pm\,0.00$ & $171.44\,\pm\,0.00$ & $3.344\,\pm\,0.064$ & $1.0e-06$--$1.0e-06$\\
 & EG 0.9 & $95.12\,\pm\,0.00$ & $95.12\,\pm\,0.00$ & $95.12\,\pm\,0.00$ & $95.12\,\pm\,0.00$ & $1.856\,\pm\,0.021$ & $1.0e-06$--$1.0e-06$\\
 & Primal AGD & $1.28\,\pm\,0.00$ & $1.28\,\pm\,0.00$ & $1.28\,\pm\,0.00$ & $1.41\,\pm\,0.00$ & $0.032\,\pm\,0.001$ & $7.5e-07$--$7.5e-07$\\
NL local 200 & Sampled & $10.36\,\pm\,1.87$ & $1.51\,\pm\,0.28$ & $5.78\,\pm\,1.09$ & $6.87\,\pm\,1.21$ & $0.610\,\pm\,0.110$ & $4.3e-08$--$5.4e-07$\\
 & Assembly & $11.25\,\pm\,3.13$ & $1.60\,\pm\,0.43$ & $15.27\,\pm\,4.28$ & $15.27\,\pm\,4.28$ & $0.664\,\pm\,0.186$ & $3.4e-08$--$8.1e-07$\\
 & EG 0.5 & $14.32\,\pm\,0.00$ & $14.32\,\pm\,0.00$ & $14.32\,\pm\,0.00$ & $14.32\,\pm\,0.00$ & $0.452\,\pm\,0.006$ & $9.9e-07$--$9.9e-07$\\
 & EG 0.9 & $8.00\,\pm\,0.00$ & $8.00\,\pm\,0.00$ & $8.00\,\pm\,0.00$ & $8.00\,\pm\,0.00$ & $0.250\,\pm\,0.002$ & $9.5e-07$--$9.5e-07$\\
NL local 2000 & Sampled & $27.66\,\pm\,3.23$ & $1.60\,\pm\,0.21$ & $6.08\,\pm\,0.71$ & $7.28\,\pm\,0.88$ & $1.331\,\pm\,0.177$ & $5.6e-08$--$8.3e-07$\\
 & Assembly & $29.16\,\pm\,4.48$ & $1.65\,\pm\,0.22$ & $15.60\,\pm\,2.47$ & $15.60\,\pm\,2.47$ & $1.371\,\pm\,0.219$ & $4.1e-08$--$7.3e-07$\\
 & EG 0.5 & $124.40\,\pm\,0.00$ & $124.40\,\pm\,0.00$ & $124.40\,\pm\,0.00$ & $124.40\,\pm\,0.00$ & $3.912\,\pm\,0.043$ & $1.0e-06$--$1.0e-06$\\
 & EG 0.9 & $69.12\,\pm\,0.00$ & $69.12\,\pm\,0.00$ & $69.12\,\pm\,0.00$ & $69.12\,\pm\,0.00$ & $2.163\,\pm\,0.020$ & $1.0e-06$--$1.0e-06$\\
\bottomrule\end{tabular}
\end{table}

\FloatBarrier
The stronger EG reduces the stage methods' margins: at $2000$,
sampled/EG $0.9$ runtime is $0.84\pm0.09$ on graph and
$0.62\pm0.03$ on nonlinear; assembly/EG is $0.60\pm0.07$ and
$0.63\pm0.04$ (bootstrap SE). All six base configurations still favor EG,
and AGD remains fastest on applicable bilinear cases.

\begin{table}[!htbp]
\centering\footnotesize\setlength{\tabcolsep}{3pt}
\caption{Mean individual sampled objective calls in thousands, including complete-stage work. Block bounds $1$--$3$ stay $(2,4,8)$; SDs and intermediate weights appear in Appendix~\ref{supp:sec:experiment-sweep-results}.}
\label{tab:individual-main}
\begin{tabular}{@{}lrrrrr@{}}\toprule
Case & $\kappa_4$ & $N_1^f$ & $N_2^f$ & $N_3^f$ & $N_4^f$\\\midrule
Graph & 20 & $0.66$ & $0.96$ & $1.47$ & $2.34$\\
Graph & 2000 & $0.68$ & $1.07$ & $1.55$ & $25.11$\\
Nonlin. & 20 & $0.56$ & $0.88$ & $1.23$ & $1.96$\\
Nonlin. & 2000 & $0.68$ & $1.03$ & $1.54$ & $24.41$\\
\bottomrule\end{tabular}
\end{table}

\FloatBarrier
Individual sampled objective counts show easy-block savings as block
four stiffens. Prescribed sampling frequencies explain this allocation;
they alone establish neither a convergence benefit nor the asymptotic
bounds.

\paragraph{Fixed-saddle conditioning control.}
For each graph-cycle or nonlinear-local family, let $x_\star,y_\star$
be its independent reference at $\kappa_4=20$. At each loss weight
$a$, replace the primal objective by
\[
 \widetilde f_a(x)=f_a(x)-
 \ip{\nabla f_a(x_\star)-\nabla f_{20}(x_\star)}{x}.
\]
The correction is affine and supported only on block four, so it
changes no Hessian, certified objective curvature, coupling gradient,
data, or dual quadratic. It makes the same saddle solve every case.
Each recentered reference is nevertheless recomputed by independent
Newton; the reporter verifies agreement to $10^{-8}$, equal $\mu$,
and equal coupling bounds. With $z^0=0$, initial squared $M$-distance
and absolute accuracy cutoffs now stay fixed:
$3.39\times10^{-6}$ for graph and $3.41\times10^{-6}$ for nonlinear.
This is a constructed affine control, not the original statistical
objective. It removes saddle motion and denominator drift while
preserving the loss-weight Hessian intervention.

\begin{table}[!htbp]
\centering\footnotesize\setlength{\tabcolsep}{1.3pt}
\caption{Affine-recentered control: fixed saddle, metric, initial distance and coupling within each family. Calls are in thousands and times in seconds (mean $\pm$ sample SD). Terminal-error ranges retain checkpoint overshoot.}
\label{tab:fixed-saddle}
\begin{tabular}{@{}llrrrrrl@{}}\toprule
Case & Method & $N^f$ & $N^g$ & $N_H^x$ & $N_H^y$ & Seconds & Final $e$\\\midrule
Graph 20 & Sampled & $5.43\,\pm\,1.12$ & $1.51\,\pm\,0.31$ & $22.37\,\pm\,4.61$ & $21.17\,\pm\,4.35$ & $0.634\,\pm\,0.135$ & $2.4e-08$--$7.9e-07$\\
 & Assembly & $5.21\,\pm\,0.93$ & $1.46\,\pm\,0.23$ & $27.34\,\pm\,4.85$ & $27.34\,\pm\,4.85$ & $0.320\,\pm\,0.055$ & $4.1e-08$--$6.4e-07$\\
 & EG 0.9 & $1.36\,\pm\,0.00$ & $1.36\,\pm\,0.00$ & $1.36\,\pm\,0.00$ & $1.36\,\pm\,0.00$ & $0.030\,\pm\,0.000$ & $8.5e-07$--$8.5e-07$\\
 & Primal AGD & $0.28\,\pm\,0.00$ & $0.28\,\pm\,0.00$ & $0.28\,\pm\,0.00$ & $0.31\,\pm\,0.00$ & $0.008\,\pm\,0.000$ & $3.3e-07$--$3.3e-07$\\
Graph 200 & Sampled & $12.17\,\pm\,2.69$ & $1.72\,\pm\,0.39$ & $26.15\,\pm\,5.58$ & $24.72\,\pm\,5.34$ & $0.992\,\pm\,0.224$ & $2.3e-08$--$2.6e-07$\\
 & Assembly & $11.25\,\pm\,3.76$ & $1.63\,\pm\,0.52$ & $30.47\,\pm\,10.10$ & $30.47\,\pm\,10.10$ & $0.596\,\pm\,0.202$ & $7.2e-08$--$8.7e-07$\\
 & EG 0.9 & $8.32\,\pm\,0.00$ & $8.32\,\pm\,0.00$ & $8.32\,\pm\,0.00$ & $8.32\,\pm\,0.00$ & $0.181\,\pm\,0.003$ & $1.0e-06$--$1.0e-06$\\
 & Primal AGD & $0.44\,\pm\,0.00$ & $0.44\,\pm\,0.00$ & $0.44\,\pm\,0.00$ & $0.48\,\pm\,0.00$ & $0.012\,\pm\,0.000$ & $3.6e-07$--$3.6e-07$\\
Graph 2000 & Sampled & $29.13\,\pm\,6.78$ & $1.66\,\pm\,0.39$ & $25.03\,\pm\,5.90$ & $23.49\,\pm\,5.48$ & $1.740\,\pm\,0.422$ & $6.5e-08$--$5.2e-07$\\
 & Assembly & $26.95\,\pm\,5.55$ & $1.54\,\pm\,0.30$ & $28.98\,\pm\,5.90$ & $28.98\,\pm\,5.90$ & $1.297\,\pm\,0.254$ & $1.2e-07$--$9.3e-07$\\
 & EG 0.9 & $98.40\,\pm\,0.00$ & $98.40\,\pm\,0.00$ & $98.40\,\pm\,0.00$ & $98.40\,\pm\,0.00$ & $2.112\,\pm\,0.010$ & $9.9e-07$--$9.9e-07$\\
 & Primal AGD & $1.24\,\pm\,0.00$ & $1.24\,\pm\,0.00$ & $1.24\,\pm\,0.00$ & $1.36\,\pm\,0.00$ & $0.034\,\pm\,0.000$ & $8.4e-07$--$8.4e-07$\\
NL 20 & Sampled & $4.62\,\pm\,0.35$ & $1.28\,\pm\,0.12$ & $4.96\,\pm\,0.34$ & $5.92\,\pm\,0.42$ & $0.370\,\pm\,0.030$ & $3.6e-08$--$9.9e-07$\\
 & Assembly & $4.45\,\pm\,0.90$ & $1.23\,\pm\,0.27$ & $11.78\,\pm\,2.39$ & $11.78\,\pm\,2.39$ & $0.361\,\pm\,0.076$ & $2.7e-08$--$8.7e-07$\\
 & EG 0.9 & $1.12\,\pm\,0.00$ & $1.12\,\pm\,0.00$ & $1.12\,\pm\,0.00$ & $1.12\,\pm\,0.00$ & $0.037\,\pm\,0.001$ & $9.1e-07$--$9.1e-07$\\
NL 200 & Sampled & $10.07\,\pm\,1.74$ & $1.48\,\pm\,0.26$ & $5.61\,\pm\,1.03$ & $6.68\,\pm\,1.14$ & $0.626\,\pm\,0.107$ & $2.2e-08$--$6.6e-07$\\
 & Assembly & $10.68\,\pm\,2.62$ & $1.51\,\pm\,0.36$ & $14.50\,\pm\,3.61$ & $14.50\,\pm\,3.61$ & $0.659\,\pm\,0.164$ & $2.9e-08$--$9.9e-07$\\
 & EG 0.9 & $8.72\,\pm\,0.00$ & $8.72\,\pm\,0.00$ & $8.72\,\pm\,0.00$ & $8.72\,\pm\,0.00$ & $0.284\,\pm\,0.004$ & $9.9e-07$--$9.9e-07$\\
NL 2000 & Sampled & $26.93\,\pm\,1.95$ & $1.56\,\pm\,0.15$ & $5.92\,\pm\,0.45$ & $7.08\,\pm\,0.52$ & $1.394\,\pm\,0.096$ & $7.8e-08$--$6.6e-07$\\
 & Assembly & $26.97\,\pm\,4.18$ & $1.54\,\pm\,0.22$ & $14.42\,\pm\,2.22$ & $14.42\,\pm\,2.22$ & $1.374\,\pm\,0.208$ & $6.7e-08$--$4.5e-07$\\
 & EG 0.9 & $85.12\,\pm\,0.00$ & $85.12\,\pm\,0.00$ & $85.12\,\pm\,0.00$ & $85.12\,\pm\,0.00$ & $2.788\,\pm\,0.027$ & $1.0e-06$--$1.0e-06$\\
\bottomrule\end{tabular}
\end{table}

\FloatBarrier
With the cutoff fixed, the $2000$ sampled/EG ratios are
$0.82\pm0.08$ on graph and $0.50\pm0.01$ on nonlinear; assembly/EG
ratios are $0.61\pm0.05$ and $0.49\pm0.03$. At $20$ and $200$,
EG is faster. The crossover therefore persists in this control;
the finite results do not establish a universal threshold or rate.

\paragraph{Fixed-dimension overlap control.}
Use four two-dimensional primal blocks and four two-dimensional dual
blocks, $M=I$, diagonal objective Hessians with spectra $(1,\kappa_i)$, primal bounds $(2,4,8,20)$,
and unit dual bounds. Set $z_\star=\boldsymbol 1/4$ in dimension 16
and $z^0=0$, giving initial squared distance one. Let
\[
 C_\Delta=2\bigl(I_{4/\Delta}\otimes U_\Delta\bigr)\otimes I_2,
 \qquad \Delta\in\{1,2,4\},
\]
where $U_\Delta$ is the normalized Sylvester Hadamard matrix.
Choose affine objective terms and $b=C_\Delta x_\star$ so the operator is
$A(z-z_\star)+T_{H,\Delta}(z-z_\star)$.
Every coupling block row has exactly $\Delta$ inputs and every input
occurs in $\Delta$ rows. Dimension, saddle, initial distance, objective
Hessians, true coupling operator norm $2$, and certificate row norms
$2$ stay fixed. Sampling's row variation coefficients therefore
change only through $\sqrt\Delta$, testing its explicit overlap
penalty. The nonnegative block-bound matrix has norm $2\sqrt\Delta$,
so the certified assembly and EG aggregate bounds also change; their
parameters are recomputed and retained in the records. The control varies coupling
orientation; it is not an identical-operator comparison.

\begin{table}[!htbp]
\centering\small\setlength{\tabcolsep}{3pt}
\caption{Overlap control on fixed $16$-dimensional quadratic saddles. True coupling norm and row certificate norms are $2$; $\Delta$ changes through orthogonal support groups. Entries are means $\pm$ sample SD; coupling calls in thousands. $S$ and $K$ are the prescribed stage parameters, not fitted convergence rates.}
\label{tab:overlap-control}
\begin{tabular}{@{}rlrrrrrr@{}}\toprule
$\Delta$ & Method & $S$ & $K$ & Stages & $N_H^x/10^3$ & $N_H^y/10^3$ & Seconds\\\midrule
1 & Sampled & 71.4 & 1286 & $6.4\,\pm\,1.1$ & $5.24\,\pm\,0.92$ & $5.24\,\pm\,0.93$ & $0.177\,\pm\,0.031$\\
1 & Assembly & 31.8 & 573 & $7.6\,\pm\,0.9$ & $6.28\,\pm\,0.76$ & $6.28\,\pm\,0.76$ & $0.120\,\pm\,0.014$\\
1 & EG 0.9 & -- & -- & -- & $1.20\,\pm\,0.00$ & $1.20\,\pm\,0.00$ & $0.0031\,\pm\,0.0001$\\
2 & Sampled & 90.2 & 1623 & $7.2\,\pm\,1.9$ & $8.34\,\pm\,2.28$ & $8.33\,\pm\,2.18$ & $0.237\,\pm\,0.062$\\
2 & Assembly & 34.2 & 615 & $7.0\,\pm\,0.7$ & $8.20\,\pm\,0.76$ & $8.20\,\pm\,0.76$ & $0.118\,\pm\,0.011$\\
2 & EG 0.9 & -- & -- & -- & $1.28\,\pm\,0.00$ & $1.28\,\pm\,0.00$ & $0.0033\,\pm\,0.0001$\\
4 & Sampled & 116.7 & 2100 & $6.6\,\pm\,0.9$ & $10.76\,\pm\,1.53$ & $10.77\,\pm\,1.45$ & $0.276\,\pm\,0.046$\\
4 & Assembly & 37.5 & 675 & $6.0\,\pm\,0.7$ & $9.90\,\pm\,1.12$ & $9.90\,\pm\,1.12$ & $0.113\,\pm\,0.012$\\
4 & EG 0.9 & -- & -- & -- & $1.36\,\pm\,0.00$ & $1.36\,\pm\,0.00$ & $0.0035\,\pm\,0.0001$\\
\bottomrule\end{tabular}
\end{table}

\FloatBarrier
From $\Delta=1$ to $4$, sampled $S$ rises from $71.4$ to $116.7$,
$K$ from $1286$ to $2100$, and primal coupling calls from
$5.24\pm0.92$ to $10.76\pm1.53$ thousand. Mean stopping stages
stay near six to seven. This demonstrates an observed work penalty
consistent with the prescribed overlap dependence, not its optimality.
EG remains faster at every overlap level.

\paragraph{Controlled unequal oracle costs.}
On the same quadratic geometry, take disjoint edges
$C=\diag(8,1,1,1)\otimes I_2$ and preserve saddle and initial distance.
For each weak-edge coupling row, repeat exact gradient arithmetic
$r\in\{1,4,16\}$ times before returning one identical partial.
Strong-edge and objective rows have unit cost. Full coupling requests
concatenate the same partial backend at every cost level, so no method
bypasses the intervention. Repetitions are internal arithmetic, not
additional requested oracle returns. Nominal weighted work is
\[
 W=\sum_iN_i^f+\sum_jN_j^g+
   \sum_{k=1}^8\tau_kN_{H,k},\qquad
 (\tau_1,\ldots,\tau_8)=(1,r,r,r,1,r,r,r).
\]
Cost multipliers model relative arithmetic work and exclude Python
overhead; they are not measured per-oracle time estimates. Gradients,
certificates, probabilities, stopping stages, and iterates stay fixed
across $r$ for each seed and method. This tests how separate counters
translate into work under controlled costs, without claiming that
the cost-blind allocation adapts to costs or that $W$ equals runtime.

\begin{table}[!htbp]
\centering\small\setlength{\tabcolsep}{3pt}
\caption{Controlled unequal coupling costs: weak-edge oracle arithmetic repeated $r$ times, returning the identical partial. Full calls concatenate those same partials at every cost level. $W$ charges unit objective costs and coupling costs $(1,r,r,r,1,r,r,r)$. Means $\pm$ sample SD; time ratio $\pm$ bootstrap SE. This backend is a constructed cost model.}
\label{tab:cost-control}
\begin{tabular}{@{}rlrrr@{}}\toprule
$r$ & Method & $W/10^3$ & Seconds & Time/EG 0.9\\\midrule
1 & Sampled & $20.47\,\pm\,4.85$ & $0.213\,\pm\,0.054$ & $26.61\,\pm\,2.73$\\
1 & Assembly & $53.91\,\pm\,9.71$ & $0.225\,\pm\,0.041$ & $28.11\,\pm\,2.06$\\
1 & EG 0.9 & $6.72\,\pm\,0.00$ & $0.0080\,\pm\,0.0002$ & $1.00\,\pm\,0.00$\\
4 & Sampled & $32.37\,\pm\,7.77$ & $0.228\,\pm\,0.053$ & $13.07\,\pm\,1.22$\\
4 & Assembly & $159.83\,\pm\,28.74$ & $0.359\,\pm\,0.064$ & $20.59\,\pm\,1.49$\\
4 & EG 0.9 & $14.28\,\pm\,0.00$ & $0.0174\,\pm\,0.0004$ & $1.00\,\pm\,0.00$\\
16 & Sampled & $79.99\,\pm\,19.46$ & $0.296\,\pm\,0.084$ & $5.44\,\pm\,0.62$\\
16 & Assembly & $583.51\,\pm\,104.87$ & $0.875\,\pm\,0.151$ & $16.08\,\pm\,1.13$\\
16 & EG 0.9 & $44.52\,\pm\,0.00$ & $0.0544\,\pm\,0.0012$ & $1.00\,\pm\,0.00$\\
\bottomrule\end{tabular}
\end{table}

\FloatBarrier
As $r$ rises from $1$ to $16$, sampled/assembly weighted work falls
from $0.38$ to $0.14$, and their mean runtime ratio from $0.95$ to
$0.34$. Thus fewer weak-row requests have a practical benefit in
this controlled backend. EG still has lower work and runtime than
either stage method, including at $r=16$.

\paragraph{Isolating evaluation of coupling rows.}
The full/discard ablation retains the sampled method's components,
probabilities, $S,\eta,K$, random draws, cache updates, and selected
row returns. Whenever it requests one coupling row, it evaluates
the fused full coupling and discards all unselected rows; every
computed row is charged. This isolates the evaluation saving without
changing the convergence mechanism or the stage allocation. For all
ten paired trajectories at the $2000$ endpoints, the reporter verifies
equal stopping stages, equal objective counts, and checkpoint iterates
agreeing to floating-point tolerance. Assembly remains a distinct
algorithm and is reported alongside both backends.

\begin{table}[!htbp]
\centering\small\setlength{\tabcolsep}{3pt}
\caption{Coupling-row ablation at $\kappa_4=2000$. Full/discard uses the same sampled components, draws, parameters, updates and stopping stages, but evaluates the full coupling and discards unselected rows. Every returned row is charged. Assembly is the separate stage alternative. Means $\pm$ sample SD; sampled/full time ratio $\pm$ bootstrap SE.}
\label{tab:row-ablation}
\begin{tabular}{@{}llrrrr@{}}\toprule
Case & Backend & $N_H^x/10^3$ & $N_H^y/10^3$ & Seconds & Sampled/full\\\midrule
Graph & Sampled & $24.40\,\pm\,6.05$ & $22.91\,\pm\,5.64$ & $1.563\,\pm\,0.404$ & $0.86\,\pm\,0.13$\\
Graph & Full/discard & $189.24\,\pm\,46.74$ & $189.24\,\pm\,46.74$ & $1.822\,\pm\,0.466$ & --\\
Graph & Assembly & $27.42\,\pm\,6.77$ & $27.42\,\pm\,6.77$ & $1.112\,\pm\,0.304$ & --\\
Nonlinear & Sampled & $6.08\,\pm\,0.71$ & $7.28\,\pm\,0.88$ & $1.331\,\pm\,0.177$ & $0.70\,\pm\,0.05$\\
Nonlinear & Full/discard & $53.42\,\pm\,6.33$ & $53.42\,\pm\,6.33$ & $1.889\,\pm\,0.213$ & --\\
Nonlinear & Assembly & $15.60\,\pm\,2.47$ & $15.60\,\pm\,2.47$ & $1.371\,\pm\,0.219$ & --\\
\bottomrule\end{tabular}
\end{table}

\FloatBarrier
Sampled/full-discard time ratios are $0.86\pm0.13$ on graph and
$0.70\pm0.05$ on nonlinear. The nonlinear result supports a runtime
benefit from evaluating only selected rows at unchanged iterates.
The graph timing difference is uncertain with five runs. Neither
ablation establishes superior convergence from row sampling, and
assembly remains faster on graph ($1.112\pm0.304$ versus
$1.563\pm0.404$ seconds); the nonlinear stage times are close relative
to their run variability.

\end{document}